\documentclass{article}
\usepackage[a4paper]{geometry}
\usepackage{amssymb}
\usepackage{systeme}
\usepackage[english]{babel}
\usepackage{color}
\usepackage{amsmath}
\usepackage{amsthm}
\usepackage{graphicx}
\usepackage{geometry}
\usepackage{mathtools}
\usepackage{comment}
\usepackage[normalem]{ulem}
\allowdisplaybreaks
\newcommand\numberthis{\addtocounter{equation}{1}\tag{\theequation}}
\usepackage[urlcolor=black]{hyperref}
\hypersetup{
    colorlinks=true, 
    linktoc=all,     
    linkcolor=black,  
    citecolor=black
}

\theoremstyle{definition}
\newtheorem{theorem}{Theorem}[section]
\newtheorem{definition}[theorem]{Definition}

\newtheorem{corollary}[theorem]{Corollary}
\newtheorem{lemma}[theorem]{Lemma}
\newtheorem{remark}[theorem]{Remark}
\newtheorem{proposition}[theorem]{Proposition}

\newcommand{\I}{\mathrm{i}}

\newcommand{\D}{\mathrm{d}}

\title{S-duality and the universal isometries of instanton corrected q-map spaces}
\date{\small Department of Mathematics\\
University of Hamburg\\
Bundesstraße 55, D-20146 Hamburg, Germany}
\author{Vicente Cort\'es and Iv\'an Tulli}
\numberwithin{equation}{section}

\begin{document}

\maketitle
\begin{abstract}
Given a conical affine special K\"{a}hler (CASK) manifold together with a compatible mutually local variation of BPS structures, one can construct a quaternionic-K\"{a}hler (QK) manifold. We call the resulting QK manifold an instanton corrected c-map space. Our main aim is to study the isometries of a subclass of instanton corrected c-map spaces associated to  projective special real (PSR) manifolds with a compatible mutually local variation of BPS structures. We call the latter subclass instanton corrected q-map spaces. In the setting of Calabi-Yau compactifications of type IIB string theory,  instanton corrected q-map spaces are related to the hypermultiplet moduli space metric with perturbative corrections, together with worldsheet, D(-1) and D1 instanton corrections. In the physics literature, it has been shown that the hypermultiplet metric with such corrections must have an $\mathrm{SL}(2,\mathbb{Z})$ acting by isometries, related to S-duality. We give a mathematical treatment of this result, specifying under which conditions instanton corrected q-map spaces carry an action by isometries by $\mathrm{SL}(2,\mathbb{Z})$ or some of its subgroups. We further study the universal isometries of instanton corrected q-map spaces, and compare them to the universal isometries of tree-level q-map spaces. Finally, we give an explicit example of a non-trivial instanton corrected q-map space with full $\mathrm{SL}(2,\mathbb{Z})$ acting by isometries and admitting a quotient of finite volume by a discrete group of isometries. \\
\end{abstract}

\tableofcontents
\section{Introduction}

The supergravity c-map assigns to a projective special K\"{a}hler (PSK) manifold $(\overline{M},g_{\overline{M}})$ of complex dimension $n$, a quaternionic-K\"{a}hler (QK) manifold $(\overline{N},g_{\overline{N}})$ of real dimension $4n+4$. In the context of Calabi-Yau compactifications of type IIA/B string theory on a Calabi-Yau threefold $X$,  the c-map takes the vector multiplet moduli space $\mathcal{M}_{\text{VM}}^{\text{IIA/B}}(X)$ to the hypermultiplet moduli space  $\mathcal{M}_{\text{HM}}^{\text{IIB/A}}(X)$ with its string-tree-level metric $g_{\text{FS}}$, also known as the Ferrara-Sabharwal metric \cite{TypeIIgeometry,FSmetric}. Such a construction receives quantum corrections in the string coupling $g_s$ of several kinds:

\begin{itemize}
    \item Perturbative corrections: these produce the so-called 1-loop corrected Ferrara-Sabharwal metric on $\mathcal{M}_{\text{HM}}^{\text{IIA/B}}(X)$ \cite{RLSV,1loopcmap}. In a purely mathematical setting, this construction can be understood as a way of assigning to a PSK manifold $(\overline{M},g_{\overline{N}})$ a 1-parameter family of QK manifolds $\{(\overline{N}_{c_{\ell}},g_{\overline{N},c_{\ell}})\}_{c_{\ell}\in \mathbb{R}}$ \cite{QKPSK}, where $c_{\ell}\in \mathbb{R}$ corresponds to the 1-loop correction.
    \item Non-perturbative quantum corrections: these are divided in D-instanton and NS5-instanton corrections. These have been extensively studied in the physics literature via twistor methods, see for example the reviews \cite{HMreview1,HMreview2} and the references therein. The inclusion of all D-instanton corrections is understood in the physics literature \cite{WCHKQK,APSV}, while the NS5-instanton corrections are still work in progress (see for example \cite{NS5}).
\end{itemize}

When restricted to the simpler setting of mutually local D-instanton corrections, a fairly explicit local formula for the QK metric was given in the physics literature via twistor methods in \cite{HMmetric}\footnote{Here by ``fairly explicit" we mean that the expression is explicit except for a certain function $\mathcal{R}$, which is implicitly defined in terms of the field variables.}. On the other hand, in \cite{CT} a mathematical treatment (based on the geometric approach of \cite{QKPSK}) of a class of QK metrics related to the above mutually local case was given. Namely, if $(\overline{M},g_{\overline{M}})$ is a PSK manifold and  $(M,g_M,\omega_M,\nabla,\xi)$ the associated conical affine special K\"{a}hler (CASK) manifold, then one can complement this data with a mutually local variation of BPS structures $(M,\Gamma,Z,\Omega)$ to construct a new QK metric $(\overline{N},g_{\overline{N}})$ (we suppress the choice of $1$-loop parameter $c_{\ell}$ from the notation). The general notion of variation of BPS structure can be found in \cite{VarBPS} (see also Definition \ref{defvarBPS} below for the specific case to be used throughout this paper). Here $(M,\Gamma,Z,\Omega)$ is assumed to satisfy certain compatibility conditions with the CASK structure $(M,g_M,\omega_M,\nabla,\xi)$ (see Section \ref{instreview} below), and encodes the analog of ``mutually local D-instanton" corrections when compared to the string theory setting.\\

On the other hand, type IIB string theory carries an $\mathrm{SL}(2,\mathbb{Z})$-symmetry called S-duality. In the physics literature, it has been shown that S-duality descends to an action by isometries on $\mathcal{M}_{\text{HM}}^{\text{IIB}}(X)$ when one includes the appropriate amount of quantum corrections for the metric. For example:

\begin{itemize}
    \item When one drops all corrections in the string coupling $g_s$ and takes a large volume limit, one obtains the classical QK metric on $\mathcal{M}_{\text{HM}}^{\text{IIB}}(X)$. This metric has been shown to have an $\mathrm{SL}(2,\mathbb{R})$ acting by isometries \cite{BGHL}, which comes from the $\mathrm{SL}(2,\mathbb{Z})$ S-duality action. Furthermore, it has been shown in \cite{BGHL} that the $\mathrm{SL}(2,\mathbb{R})$ (and even $\mathrm{SL}(2,\mathbb{Z})\subset \mathrm{SL}(2,\mathbb{R}))$ is broken when one includes world-sheet instanton corrections and/or the 1-loop correction in $g_s$. 
    \item On the other hand, it has been shown in \cite{QMS, Sduality} that when one includes world-sheet instanton corrections, the 1-loop correction, and D(-1) and D1 instanton corrections, one recovers again an $\mathrm{SL}(2,\mathbb{Z})$ action by isometries coming from the S-duality action. As a consequence of their result, it also follows that only including perturbative world-sheet instanton corrections, the 1-loop correction in $g_s$ and D(-1) instanton corrections also preserves the isometric S-duality action. For some other extensions of this result in the physics literature see for example \cite{Sduality,NS5,D31,D32}. 
    \item The QK metric of $\mathcal{M}_{\text{HM}}^{\text{IIB}}(X)$ is also expected to retain the $\mathrm{SL}(2,\mathbb{Z})$ S-duality action by isometries when one includes all quantum corrections, but such a metric has not been constructed or understood (as far as the authors know) in the physics literature. 
\end{itemize}
The classical metric obtained in the case when one drops all corrections in $g_s$ and take a large volume limit lies in a subset of c-map metrics called q-map metrics \cite{cmp/1104251224}. Mathematically, the q-map assign to a $n-1\geq 0$ dimensional projective special real (PSR) manifold a $4n+4$ dimensional QK manifold \cite{CHM,CDJL}. A purely differential geometric proof that q-map metrics carry an $\mathrm{SL}(2,\mathbb{R})$ of isometries was given in \cite{CTSduality}, while more traditional supergravity arguments can be found for example in \cite{sugra1,sugra2}. On the other hand,  including world-sheet instanton corrections and the 1-loop correction takes the q-map metric to the class of 1-loop corrected c-map metrics; while including  D(-1) and D1 instanton corrections takes the 1-loop corrected c-map metric to the class of mutually local instanton corrected QK metrics studied in \cite{HMmetric, CT}.\\

Our main objective in this paper is to do a mathematical treatment of the S-duality results in \cite{QMS,Sduality}, and to study the universal isometries of instanton corrected q-map spaces, i.e.\ those isometries that are independent of the initial PSR manifold and the form of the quantum corrections that we restrict ourselves to. Namely, among the class of instanton corrected c-map metrics $(\overline{N},g_{\overline{N}})$ constructed in \cite{CT} we restrict to a subclass, which we call instanton corrected q-map metrics, and show under which conditions they carry an $\mathrm{SL}(2,\mathbb{Z})$-action (or an action by some of its subgroups) by isometries.  This $\mathrm{SL}(2,\mathbb{Z})$-action is furthermore related to the S-duality symmetry in the string theory setting. Furthermore, we study how the universal group of isometries of q-map spaces described in \cite{CTSduality} gets modified for instanton corrected q-map spaces (see Section \ref{sumsec} below for more details). 

\subsection{Summary of main results}\label{sumsec}
The main differences and new results compared to the works in the physics literature \cite{QMS,Sduality}, from which we take a lot of inspiration from, are the following:

\begin{itemize}
    \item  If the S-duality $\mathrm{SL}(2,\mathbb{Z})$ action defined in  \eqref{sl2can} restricts to an action on the domain of definition of the instanton corrected q-map space, then the twistor space argument of \cite{QMS,Sduality} follows and the action is also by isometries. However, verifying that the domain of definition actually carries such an action seems non-trivial. In Theorem \ref{collectiontheorem} (explained in more detail below) we collect results regarding this point. In particular we find that, even in a case where we do not have an $\mathrm{SL}(2,\mathbb{Z})$-invariant domain of definition, one can always find either an $S\in \mathrm{SL}(2,\mathbb{Z})$ or $T\in \mathrm{SL}(2,\mathbb{Z})$-invariant one (where $T$ and $S$ are the usual generators of $\mathrm{SL}(2,\mathbb{Z})$ given in \eqref{sl2gen}). Furthermore, in Section \ref{examplesec} we give an explicit non-trivial (albeit simple) example where one can find such an $\mathrm{SL}(2,\mathbb{Z})$-invariant neighborhood of definition.   
    \item Assuming that the domain of definition of the instanton corrected q-map space carries an action by the S-duality $\mathrm{SL}(2,\mathbb{Z})$, we show that this action must be by isometries by following an argument similar but slightly different from \cite{QMS}. Using their twistor description of the relevant QK manifolds, it is shown in \cite{QMS} that certain ``type IIA" Darboux coordinates for the contact structure of the twistor space can be Poisson resummed, and then it is shown that the later can be related via a gauge transformation to certain ``type IIB" Darboux coordinates. The type IIB coordinates then make transparent that a certain lift of the $\mathrm{SL}(2,\mathbb{Z})$ action to the twistor space acts by twistor space automorphisms, and hence that $\mathrm{SL}(2,\mathbb{Z})$ acts by isometries on the QK metric. In our work, we find it simpler to do a direct Poisson resummation of the contact structure corresponding to the QK metric constructed in \cite{CT}, and then use the resulting expression to show that the ``type IIB" coordinates from \cite{QMS} are Darboux coordinates for the contact structure.
    \item We further study certain universal isometry groups of instanton corrected q-map spaces, and compare with with what happens in the case where no quantum corrections are considered (see Section \ref{universalisosec} and Theorem \ref{collectiontheorem}). In particular, while S-duality is a universal isometry for tree-level q-map spaces, we are not able to guarantee that the the same is true for instanton corrected q-map spaces (see Remark \ref{introremark}). Furthermore, in the example of Section \ref{examplesec} we use the $\mathrm{SL}(2,\mathbb{Z})$ action by isometries together with the universal isometry group to show that our example admits a quotient of finite volume by a discrete group of isometries. 
\end{itemize}

In the following, we explain in more detail the setting and the aforementioned results. As previously mentioned, our main results concerns a class of QK metrics called instanton corrected q-map spaces, defined in Section \ref{settingsec}. Roughly speaking, these are QK manifolds associated to a CASK manifold described by a holomorphic prepotential $(M,\mathfrak{F})$, together with a variation of mutually local BPS structures $(M,\Gamma,Z,\Omega)$ (see \cite{VarBPS} for the general definition of variation of BPS structures) of the following form:

\begin{itemize}
    \item The holomorphic prepotential $\mathfrak{F}:M\subset \mathbb{C}^{n+1}\to \mathbb{C}$ has the form

\begin{equation}
    \mathfrak{F}(Z^i)=-\frac{1}{6}k_{abc}\frac{Z^aZ^bZ^c}{Z^0}+\chi\frac{(Z^0)^2\zeta(3)}{2(2\pi \I )^3}-\frac{(Z^0)^2}{(2\pi \I)^3}\sum_{\hat{\gamma}=q_a\gamma^a \in \Lambda^+}n_{\hat{\gamma}}\mathrm{Li}_3(e^{2\pi \I q_aZ^a/Z^0}),
\end{equation}
where $i=0,...,n$; $a,b,c=1,...,n$; $k_{abc}\in \mathbb{R}$ are symmetric in the indices, $\chi\in \mathbb{Z}$, $n_{\hat{\gamma}}\in \mathbb{Z}$, $\mathrm{Li}_n(x)$ denote the n-th polylogarithms, $\zeta(x)$ is the Riemann zeta function, and $\Lambda^{+}:=\text{span}_{\mathbb{Z}_{\geq 0}}\{\gamma^a\}_{a=1}^n-\{0\}$ is a commutative semigroup freely generated by $n$ elements $\{\gamma^a\}_{a=1}^n$.\\

This choice of $\mathfrak{F}$ is motivated by Calabi-Yau compactifications of string theory. Namely, if $X$ denotes a Calabi-Yau threefold, and if $k_{abc}$ are taken to be the triple intersection numbers of $X$, $\chi=\chi(X)$ the Euler characteristic, and for $\hat{\gamma}\in H_2^{+}(X,\mathbb{Z})$ we have that $n_{\hat{\gamma}}=n_{\hat{\gamma}}^{(0)}$ are the genus zero Gopakumar-Vafa invariants, then $\mathfrak{F}$ denotes the prepotential specifying the PSK geometry of $\mathcal{M}_{\text{VM}}^{\text{IIA}}(X)$ with all worldsheet corrections.  Applying the 1-loop corrected c-map, one obtains from $\mathcal{M}_{\text{VM}}^{\text{IIA}}(X)$ the $1$-loop corrected metric on $\mathcal{M}_{\text{HM}}^{\text{IIB}}(X)$.
\item The charge lattice $\Gamma$ and the central charge $Z$ of $(M,\Gamma,Z,\Omega)$ are canonically determined by $\mathfrak{F}$ (see Section \ref{QKinstdomainsec}), while the BPS indices $\Omega(\gamma)$ are also determined by $\mathfrak{F}$ as follows: with respect to a canonical Darboux basis $(\widetilde{\gamma}_i,\gamma^i)_{i=0}^n$ of $\Gamma$ with respect to its symplectic pairing, we have 
\begin{equation}\label{BPSindintro}
       \begin{cases}
     \Omega(q_0\gamma^0)=-\chi, \quad q_0\in \mathbb{Z}-\{0\}\\
     \Omega(q_0\gamma^0\pm q_a\gamma^a)=\Omega(\pm q_a\gamma^a)=n_{q_a\gamma^a} \quad \text{for $q_a\gamma^a \in \Lambda^+$, $q_0\in \mathbb{Z}$}\\
    \Omega(\gamma)=0 \quad \text{else}.\\
        \end{cases}
\end{equation}
The prescription \eqref{BPSindintro} has previously appeared in the physics literature (see for example \cite[Equation 4.5]{HMreview2}), and is the data required to add the D(-1) and D1 instanton corrections to  $\mathcal{M}_{\text{HM}}^{\text{IIB}}(X)$ in an ``S-duality compatible"-way. Furthermore, in the case of a non-compact Calabi-Yau threefold $X$ without compact divisors the same structure for the BPS indices is expected. Indeed, in \cite{Joyce:2008pc} the appropriate $\Omega(\gamma)$ determining generalized Donaldson-Thomas invariants are constructed. It is then shown that $\Omega(q_0\gamma^0)=-\chi(X)$ \cite[Section 6.3]{Joyce:2008pc}, and conjectured that $\Omega(q_0\gamma^0\pm \hat{\gamma})=n_{\hat{\gamma}}^{(0)}$\cite[Conjecture 6.20]{Joyce:2008pc}.
\end{itemize}

To the above data $(M,\mathfrak{F})$ and $(M,\Gamma,Z,\Omega)$ we can apply the construction of \cite{CT} and obtain a QK manifold $(\overline{N},g_{\overline{N}})$, which we call an instanton corrected q-map space. $(\overline{N},g_{\overline{N}})$ depends on a choice of projective special real (PSR) manifold $(\mathcal{H},g_{\mathcal{H}})$ (determining the first term in $\mathfrak{F}$), the choice of $\chi\in \mathbb{Z}$ and $n_{\hat{\gamma}} \in \mathbb{Z}$, and the choice of 1-loop parameter $c_{\ell}\in \mathbb{R}$ (see Section \ref{settingsec}). Our main results concern the isometries of a lift $(\widetilde{N},g_{\overline{N}})\to (\overline{N},g_{\overline{N}})$ on which we have no periodic directions (see Definition \ref{liftdef} for a more precise statement). In order to state the main results, we consider the following subgroups of the Heisenberg group $\text{Heis}_{2n+3}(\mathbb{R})$ (endowed with standard global coordinates $(\eta^i,\widetilde{\eta}_i,\kappa )$, $i=0,\ldots, n$):

\begin{equation}
    \begin{split}
        H_{2n+2}&:=\{ (\eta^i,\widetilde{\eta}_i,\kappa)\in \text{Heis}_{2n+3}(\mathbb{R}) \; |\quad \; \eta^0=0 \}\\
        H_{2n+2,D}&:=\{ (\eta^i,\widetilde{\eta}_i,\kappa)\in H_{2n+2} \; |\quad \; \text{$\eta^a \in \mathbb{Z}$ for $a=1,...,n$}\}\,.
    \end{split}
    \end{equation}
The following theorem collects our main results:\\

\textbf{Theorem \ref{collectiontheorem}:} Consider an instanton corrected q-map space $(\widetilde{N},g_{\overline{N}})$ of dimension $4n+4$ as defined in Section \ref{settingsec} (in particular $\widetilde{N}$ here is the maximal domain of definition). Furthermore, let $T,S\in \mathrm{SL}(2,\mathbb{Z})$ be as in \eqref{sl2gen}, where $\mathrm{SL}(2,\mathbb{Z})$ acts on the ambient manifold $\overline{\mathcal{N}}_{\text{IIB}}^{\text{cl}}\supset \overline{\mathcal{N}}_{\text{IIB}}\cong \overline{\mathcal{N}}_{\text{IIA}}\supset \widetilde{N}$ as described in \eqref{sl2can}. Then:

\begin{itemize}
\item $(\widetilde{N},g_{\overline{N}})$ has a group acting by isometries of the form

\begin{equation}\label{breakisointroduction}
    \langle T \rangle \ltimes (\mathbb{Z}^n \ltimes H_{2n+2,D})
\end{equation}
where $\langle T \rangle \cong \mathbb{Z}$ denotes the subgroup of $\mathrm{SL}(2,\mathbb{Z})$ generated by $T$.
\item Assume that we take the one-loop parameter to be $c_{\ell}=\frac{\chi}{192\pi}$. Then we can always find a non-empty open subset $\widetilde{N}_S\subset \widetilde{N}$ where $(\widetilde{N}_S,g_{\overline{N}})$ has a group acting by isometries of the form 
 \begin{equation}
      \langle S \rangle \ltimes (\mathbb{Z}^n \ltimes H_{2n+2,D}),
 \end{equation}
 where $\langle S \rangle \cong \mathbb{Z}/4\mathbb{Z}$ is the subgroup generated by $S$. Furthermore, if $\widetilde{N}_{\mathrm{SL}(2,\mathbb{Z})}\subset \widetilde{N}$  
 is an open subset, which is $\mathrm{SL}(2,\mathbb{Z})$-invariant under the S-duality action \eqref{sl2can}, then $\mathrm{SL}(2,\mathbb{Z})$ acts by isometries on $(\widetilde{N}_{\mathrm{SL}(2,\mathbb{Z})},g_{\overline{N}})$. In particular, if $\widetilde{N}$ is already invariant under $\mathrm{SL}(2,\mathbb{Z})$ then \eqref{breakisointroduction} can be enhanced to
\begin{equation}
    \mathrm{SL}(2,\mathbb{Z})\ltimes (\mathbb{Z}^n\ltimes H_{2n+2,D})\,.
\end{equation}
\item Finally, if $n_{\hat{\gamma}}=0$ for all $\hat{\gamma}\in \Lambda^{+}$, then in the previous statements we can replace $\mathbb{Z}^n$ and $H_{2n+2,D}$ by $\mathbb{R}^n$ and $H_{2n+2}$.
If furthermore we take  $\chi=c_{\ell}=0$ and $n_{\hat{\gamma}}=0$ for all $\hat{\gamma}\in \Lambda^{+}$, then we return to the tree-level q-map space case, where there is a connected $3n+6$ dimensional Lie group $G$ acting by isometries on $(\widetilde{N},g_{\overline{N}})$, see \cite[Theorem 3.17]{CTSduality}. The group $G$ in particular contains the S-duality action by $\mathrm{SL}(2,\mathbb{R})$, an action by $\mathbb{R}^{n}\ltimes H_{2n+2}$, and a dilation action by $\mathbb{R}_{>0}$. 
\end{itemize}
\begin{remark}\label{introremark}\leavevmode
\begin{itemize}
\item Note that from Theorem \ref{collectiontheorem} one finds that 
    $\langle T \rangle \ltimes (\mathbb{Z}^n \ltimes H_{2n+2,D})$ is a universal group of isometries, in the sense that it is always an isometry group for any instanton corrected q-map space (provided one takes the maximal domain of definition $\widetilde{N}$ of the metric $g_{\overline{N}}$). On the other hand, even in the case of $c_{\ell}=\frac{\chi}{192\pi}$, Theorem \ref{collectiontheorem} does not guarantee in general that $\mathrm{SL}(2,\mathbb{Z})$ is an isometry group for $(\widetilde{N},g_{\overline{N}})$, but rather one must first check that $\widetilde{N}$ (or an open subset) carries an action by S-duality. In particular, Theorem \ref{collectiontheorem} does not let us conclude that $\mathrm{SL}(2,\mathbb{Z})$ is a universal group of isometries. This should be contrasted to the tree-level q-map space case, where $\mathrm{SL}(2,\mathbb{R})$ is known to always act by isometries.
\item We remark that the action of $S\in \mathrm{SL}(2,\mathbb{Z})$ is perhaps the most interesting and non-trivial within $\mathrm{SL}(2,\mathbb{Z})$, and corresponds to interchanging weak coupling and strong coupling in the type IIB string theory setting. On the other hand, the action by $T\in \mathrm{SL}(2,\mathbb{Z})$ generates the discrete Heisenberg isometry that is missing from $H_{2n+2,D}$. 
\end{itemize}
\end{remark}

In Section \ref{examplesec} we give an explicit example where we can achieve full $\mathrm{SL}(2,\mathbb{Z})$ acting by isometries. More precisely, we consider the case where $\mathfrak{F}$ is given simply by
\begin{equation}\label{exint}
    \mathfrak{F}=-\frac{1}{6}\frac{(Z^1)^3}{Z^0}+\chi \frac{(Z^0)^2 \zeta(3)}{2(2\pi \mathrm{i})^3}, \quad \chi>0\,,
\end{equation}
with variation of BPS structures having BPS indices of the form 
     \begin{equation}\label{exint2}
        \Omega(\gamma)=\begin{cases}
        \Omega(q_0\gamma^0)=-\chi, \quad q_0\neq 0 \\
    \Omega(\gamma)=0 \quad \text{else}.\\
        \end{cases}
    \end{equation}
From this data one obtains an 8-dimensional instanton corrected q-map space. It satisfies the following:\\

\textbf{Corollary \ref{QKex} and \ref{excor2}}: let $\widetilde{N}$ be defined by \eqref{tildeNex} and take $c_{\ell}=\frac{\chi}{192\pi}$. Then the instanton corrected q-map metric $g_{\overline{N}}$ associated to \eqref{exint} and \eqref{exint2} is defined and positive definite on $\widetilde{N}$. Furthermore, it carries an effective action by isometries by a group of the form $\mathrm{SL}(2,\mathbb{Z})\ltimes(\mathbb{R}\ltimes H_4)$.\\

\textbf{Theorem \ref{finitevol}:} Let $(\widetilde{N},g_{\overline{N}})$ be as in Corollary \ref{QKex}. Then: 
\begin{itemize}
    \item There is a free and properly discontinuous action by isometries of a discrete group of the form $\mathrm{SL}(2,\mathbb{Z})'\ltimes \Lambda$, where $\Lambda \subset \mathbb{R}\ltimes H_4$ is a lattice,  $\mathrm{SL}(2,\mathbb{Z})' \subset \mathrm{SL}(2,\mathbb{Z})$ is a finite index subgroup and the QK manifold $(\widetilde{N}/(\mathrm{SL}(2,\mathbb{Z})'\ltimes \Lambda),g_{\overline{N}})$ has finite volume. 
    \item Furthermore, there is a submanifold with boundary $\hat{N}\subset \widetilde{N}$ where $\mathrm{SL}(2,\mathbb{Z})'\ltimes \Lambda$ acts and the quotient $(\hat{N}/(\mathrm{SL}(2,\mathbb{Z})'\ltimes \Lambda),g_{\overline{N}})$ gives a complete QK manifold with boundary\footnote{Recall that a Riemannian manifold with boundary is complete if it is complete as a metric space with the induced distance function.} and of finite volume. The manifold with boundary $\hat{N}$ is of the form $\hat{N}=\widetilde{N}'\cup \partial \hat{N}$, where $\widetilde{N}'$ is defined in the second point of Remark \ref{incompleterem}.
\end{itemize} 

We note that the example $(\widetilde{N},g_{\overline{N}})$ of Corollary \ref{QKex} is incomplete (see Remark \ref{incompleterem}). We do not know if the metric and the $\mathrm{SL}(2,\mathbb{Z})$-action can be extended to a complete manifold (without boundary). 
At the end of Remark \ref{finalremark} we comment about some expectations of a related example associated to the resolved conifold.\\

Finally, we make a short comment related to the swampland program in physics. Among the geometric properties expected from the moduli space $\mathcal{M}$ of a low energy effective theory consistent with quantum gravity, are that $\mathcal{M}$ should be non-compact, have finite volume, and be geodesically complete (see \cite{OV} and \cite[Section 4.7]{Cecottiswamp}). In particular, applied to type IIA/B string theory, they imply that $\mathcal{M}_{\text{HM}}^{\text{IIA/B}}(X)$ must be a non-compact complete QK manifold of finite volume (after including all quantum corrections). On the other hand, the example from Theorem \ref{finitevol} produces a non-compact QK manifold of finite volume, with ``partial completeness" in the sense that it has a complete end, and a boundary where the metric is geodesically incomplete. It would be interesting to see if a suitable extension of the example of Theorem \ref{finitevol} would produce a QK manifold with the required geometric properties expected by the swampland conjectures.

\subsection{Organization of topics}

The topics are organized as follows:

\begin{itemize}
    \item In Section \ref{instreview} we review the construction of instanton corrected c-map metrics from \cite{CT}, and also discuss their twistor description. In particular, the instanton corrected c-map spaces from \cite{CT} are in the image of the HK/QK correspondence, and we want to recall a description of the QK twistor space in terms of the HK data, as done in \cite[Section 4.3]{CTSduality}.
    \item In Section \ref{Sdualitysec} we start by specifying the class of instanton corrected q-map spaces within the class of instanton corrected c-map metrics. Following closely the work in the physics literature of \cite{QMS,Sduality}, we study when an instanton corrected q-map space carries an $\mathrm{SL}(2,\mathbb{Z})$-action by isometries, or at least an action by some of its the subgroups, like $\langle S \rangle \subset \mathrm{SL}(2,\mathbb{Z})$. 
    \item In Section \ref{universalisosec} we study certain universal isometries of instanton corrected q-map spaces and how they are related to the $\mathrm{SL}(2,\mathbb{Z})$ S-duality symmetry. We collect the main results from Section \ref{Sdualitysec} and Section \ref{universalisosec} in Theorem \ref{collectiontheorem}. 
    \item In Section \ref{examplesec}, we give an explicit example of an instanton corrected q-map space where the full $S$-duality acting by isometries is realized. Furthermore, we show that it admits a quotient of finite volume by a discrete group of isometries. 
    \item Finally, in Appendix \ref{Besselappendix} we collect some useful integral identities involving Bessel functions, and in Appendix \ref{appendixtypeiiadc} we include a rather long computation that is not really needed for the main points of the paper, but we include for completeness. 
\end{itemize}

\textbf{Acknowledgements:} this work was supported by the Deutsche Forschungsgemeinschaft (German Research Foundation) under Germany’s Excellence Strategy -- EXC 2121 ``Quantum Universe'' -- 390833306. As with our previous related joint works \cite{CT,CTSduality}, the idea for this work originated from discussions within our Swampland seminar, which is part of the Cluster of Excellence Quantum Universe. The authors would like to thank Murad Alim, J\"{o}rg Teschner and Timo Weigand for their contributions in the aforementioned discussions.
\section{Instanton corrected QK metrics and their twistor description}\label{instreview}

The main aims of this section are the following:

\begin{itemize}
    \item On one hand, we want to recall the main results of \cite{CT}, concerning the construction of QK metrics associated to certain CASK manifolds with mutually local variation of BPS structures. In the setting of Calabi-Yau compactifications of string theory, these metrics are related to the type IIA hypermultiplet metric with mutually local D-instanton corrections, studied in the physics literature in \cite{HMmetric}. 
    \item On the other hand, we want to recall certain general facts about the twistor space of QK metrics in the image of the HK/QK correspondence. This part will be mostly based on \cite[Section 4]{CTSduality}. In particular, the QK metrics from the previous point lie in the image of the HK/QK correspondence, and we want to write down an explicit  expression for the holomorphic contact structure of its twistor space in terms of the HK data (see \eqref{contactlocal} and \eqref{fthetaexpres}). These formulas will be used throughout the rest of this work, and in particular in Section \ref{Sdualitysec}, where we study the isometries of instanton corrected q-map spaces.
    \item Finally, we write down certain ``type IIA" Darboux coordinates for the holomorphic contact structure, see \eqref{Darbouxcoordsinst}. These have been previously written down in the physics literature \cite{WCHKQK}, under slightly different conventions. Using the explicit formula for the contact structure obtained in the previous point, we will give in the Appendix \ref{appendixtypeiiadc} a direct proof of the fact that they are Darboux coordinates. This particular result is not needed for Section \ref{Sdualitysec}, where certain "type IIB" Darboux coordinates are found, but we include it for completeness.  
\end{itemize}
\subsection{QK metrics associated to CASK manifolds with mutually local variations of BPS structures}\label{QKCASKrecap}

We briefly recall the main ingredients in the construction of \cite{CT}. 

\begin{definition}
An integral conical affine special K\"{a}hler (CASK) manifold is a tuple  $(M,g_M,\omega_M,\nabla,\xi,\Gamma)$ where:

\begin{itemize}
    \item $(M,g_M,\omega_M,\nabla)$ is an affine special K\"ahler (ASK) manifold. Namely, $(M,g_M,\omega_M)$ is pseudo-K\"{a}hler, with the complex structure $J$ determined by the metric $g_M$ and the K\"ahler form $\omega_M$ by $g_M(J-,-)=\omega_M(-,-)$; $\nabla$ is a torsion-free flat connection with $\nabla \omega_M=0$; and if $d_{\nabla}:\Omega^k(M,TM)\to \Omega^{k+1}(M,TM)$ denotes the extension of $\nabla:\Omega^0(M,TM)\to \Omega^{1}(M,TM)$ to higher degree forms, then $d_{\nabla}J=0$, where we think of $J$ as an element of $\Omega^1(M,TM)$.
    \item $\Gamma\subset TM$ is a sub-bundle of $\nabla$-flat lattices with $\Gamma\otimes_{\mathbb{Z}}\mathbb{R}=TM$. Around any $p\in M$, we can find a local trivialization $(\widetilde{\gamma}_i,\gamma^i)$ of $\Gamma$ of Darboux frames with respect to $\omega_M$. We denote $\langle -,- \rangle:=\omega_M(-,-)|_{\Gamma\times \Gamma}$, and our conventions are that $\langle \widetilde{\gamma}_i,\gamma^j\rangle=\delta_i^j$.
    \item $\xi$ is a vector field on $M$ such that $\nabla \xi=D\xi=\text{Id}_{TM}$, where $D$ denotes the Levi-Civita connection of $g_M$, and $\text{Id}_{TM}$ is the identity endomorphism of $TM$. Furthermore, we assume that $g_M$ is positive definite on $\mathcal{D}:=\text{span}\{\xi,J\xi\}$ and negative definite on $\mathcal{D}^{\perp}$.
\end{itemize}
\end{definition}

On the other hand, the data corresponding to the mutually local instanton corrections in the string theory setting was specified in terms of the notion of a mutually local variation of BPS structures (see for example \cite{VarBPS} for the more general notion of a variation of BPS structures).

\begin{definition}\label{defvarBPS}
A variation of mutually-local BPS structures over the complex manifold $M'$ is a tuple $(M',\Gamma',Z,\Omega)$ where 

\begin{itemize}
    \item $\Gamma' \to M'$ is a local system of lattices with a skew-pairing $\langle - , - \rangle:\Gamma' \times \Gamma' \to \mathbb{Z}$.
    \item $Z$ is a holomorphic section of $(\Gamma')^*\otimes \mathbb{C}\to M'$, where $(\Gamma')^*$ denotes the dual local system of $\Gamma'$. If $\gamma$ is a local section of $\Gamma'$, then we denote by $Z_{\gamma}:=Z(\gamma)$ the corresponding local holomorphic function on $M'$.
    \item $\Omega: \Gamma' -\{0\}\to \mathbb{Z}$ is a function of sets satisfying $\Omega(\gamma)=\Omega(-\gamma)$ and the following properties
    \begin{itemize}
        \item Mutual-locality: if we define $\text{Supp}(\Omega):=\{\gamma \in \Gamma' -\{0\}\;\; | \;\; \Omega(\gamma)\neq 0\}$, then $\gamma_1,\gamma_2\in \Gamma'_p\cap \text{Supp}(\Omega)$ implies that $\langle \gamma_1,\gamma_2\rangle=0$.
        \item Support property: for any compact set $K\subset  M'$ and a choice of covariantly constant norm $|\cdot |$ on $\Gamma'|_{K}\otimes_{\mathbb{Z}}\mathbb{R}$, there is a constant $C>0$ such that for all $\gamma \in \Gamma'|_{K}\cap \text{Supp}(\Omega)$
        \begin{equation}
        |Z_{\gamma}|>C|\gamma|\,.
        \end{equation}
        \item Convergence property: for any $R>0$, the series
        \begin{equation}\label{convpropBPS}
            \sum_{\gamma\in \Gamma'|_p}\Omega(\gamma)e^{-R|Z_{\gamma}|}
        \end{equation}
        converges normally over compact subsets of $M'$.
        \item The numbers $\Omega(\gamma)$, called BPS indices, are monodromy invariant. Namely if $\gamma$ has monodromy $\gamma \to A\gamma$ around a loop, then $\Omega(\gamma)=\Omega(A\gamma)$.
    \end{itemize} 
\end{itemize}
\end{definition}

Given an integral CASK manifold $(M,g_M,\omega_M,\nabla,\xi,\Gamma)$, we will only consider mutually local variations of BPS structures $(M',\Gamma',Z,\Omega)$ where $(M,\Gamma)=(M',\Gamma')$, $\langle -, - \rangle=\omega_M(-,-)|_{\Gamma\times \Gamma}$,  and where $Z$ is the canonical central charge associated to the integral CASK manifold \cite[Proposition 2.15]{CT}. The later is determined as follows: if $\xi^{1,0}=\frac{1}{2}(\xi-\I J\xi)$, then

\begin{equation}
    Z:=2\omega_M(\xi^{1,0},-)|_{\Gamma}\,.
\end{equation}
In particular, given a local Darboux frame $(\widetilde{\gamma}_i,\gamma^i)$ of $\Gamma$, the locally defined functions $\{Z_{\widetilde{\gamma}_i}\}_{i=0}^{n}$ and $\{Z_{\gamma^i}\}_{i=0}^{n}$ give conjugate systems of holomorphic special coordinates for the CASK geometry, where $n+1=\text{dim}_{\mathbb{C}}(M)$.

\subsubsection{Associated instanton corrected HK manifold}\label{instHKsec}

To the data $(M,g_M,\omega_M,\nabla,\xi,\Gamma)$ and $(M,\Gamma,Z,\Omega)$ one can associate an ``instanton-corrected" hyperk\"{a}hler (HK) geometry \cite[Section 3]{CT}. This HK manifold can be thought as a deformation of the canonical HK manifold associated to $(M,g_M,\omega_M,\nabla,\xi,\Gamma)$ via the rigid c-map (also known as the associated ``semi-flat" HK manifold) \cite{TypeIIgeometry,C,SK,ACD}. In the physics literature, see \cite{GMN}, such instanton corrected HK metrics were studied in the context of $S^1$-compactifications of 4d $\mathcal{N}=2$ SUSY gauge theories. There the description of the HK geometry is in terms of its associated twistor space, which in turn in described in terms of a twistor family of holomorphic Darboux coordinates satisfying ``TBA"-like integral equations.\\ 

In order to describe the instanton corrected HK manifold, we first let $N:=T^*M/\Gamma^*$. This can be canonically identified with

\begin{equation}
    N\cong \{\zeta: \Gamma \to \mathbb{R}/ \mathbb{Z} \; | \; \zeta_{\gamma+\gamma'}=\zeta_{\gamma}+\zeta_{\gamma'}\}\,.
\end{equation}
In particular, slightly abusing notation and denoting by $\zeta$ the evaluation map on $N$, and given a local Darboux frame $(\widetilde{\gamma}_i,\gamma^i)$ of $\Gamma$, we obtain local coordinates on $N$ by $(Z_{\gamma^i},\zeta_{\widetilde{\gamma}_i},\zeta_{\gamma^i})$ (or $(Z_{\widetilde{\gamma}_i},\zeta_{\widetilde{\gamma}_i},\zeta_{\gamma^i})$).  Note that $Z_{\gamma^i}$ and $Z_{\widetilde{\gamma_i}}$ are (pull-backs of local) holomorphic functions on the base manifold $M$ while $\zeta_{\widetilde{\gamma}_i}$ and $\zeta_{\gamma^i}$ are ``fiber coordinates" taking values in the circle. \\

In the following, we will also denote by $\langle -, -\rangle$ the pairing on $\Gamma^*$ induced by the isomorphism $\gamma \mapsto \langle \gamma, - \rangle$. With this definition, the dual of a Darboux basis of $\Gamma$ is a Darboux basis of $\Gamma^*$. We will also denote by $\langle -, -\rangle$ the $\mathbb{C}$-linear extension of the pairing to $\Gamma^*\otimes \mathbb{C}$. \\

Finally, if $K_i: \mathbb{R}_{>0}\to \mathbb{R}$ denotes the $i$-th modified Bessel function of the second kind, and $\gamma$ is a local section of $\Gamma$ with $\gamma \in \text{Supp}(\Omega)$, we define the following local function (resp. 1-form) on $N$:

\begin{equation}\label{VAinst}
    V_{\gamma}^{\text{inst}}:=\frac{1}{2\pi}\sum_{n>0}e^{2\pi\I n\zeta_{\gamma}}K_0(2\pi n|Z_{\gamma}|), \quad A_{\gamma}^{\text{inst}}:=-\frac{1}{4\pi}\sum_{n>0}e^{2\pi\I n\zeta_{\gamma}}|Z_{\gamma}|K_1(2\pi n|Z_{\gamma}|)\Big( \frac{\mathrm{d}Z_{\gamma}}{Z_{\gamma}}-\frac{\mathrm{d}\overline{Z}_{\gamma}}{\overline{Z}_{\gamma}}\Big)\,.
\end{equation}
Due to the convergence property and support property of variations of BPS structures, these expressions  are well-defined local smooth functions (resp. $1$-forms) on $N$ (see \cite[Lemma 3.9]{CT}).\\

Finally, we will need the following compatibility condition between the data $(M,g_M,\omega_M,\nabla,\xi,\Gamma)$ and $(M,\Gamma,Z,\Omega)$:

\begin{definition}\label{compdef}
Let $\pi:N\to M$ be the canonical projection. We will say that $(M,g_M,\omega_M,\nabla,\xi,\Gamma)$ and $(M,\Gamma,Z,\Omega)$ are compatible if the tensor 

\begin{equation}\label{non-deg}
    T:=\pi^*g_{M}- \sum_{\gamma}\Omega(\gamma)V_{\gamma}^{\text{inst}}\pi^*|\mathrm{d}Z_{\gamma}|^2
\end{equation}
on $N$ is horizontally non-degenerate.
\end{definition}

We then have the following:

\begin{theorem} \cite[Theorem 3.13]{CT}\label{HKinsttheorem} Let $(M,g_M,\omega_M,\nabla,\xi,\Gamma)$ and $(M,\Gamma,Z,\Omega)$ be as before. Furthermore, let  $\omega_i\in \Omega^2(N)$ for $i=1,2,3$ be defined by
\begin{equation} \label{holsym}
    \omega_1+\mathrm{i}\omega_2:=-2\pi\left(\langle \mathrm{d}Z\wedge \mathrm{d}\zeta \rangle + \sum_{\gamma}\Omega(\gamma)\left(\mathrm{d}Z_{\gamma}\wedge A_{\gamma}^{\text{inst}}  +\I V_{\gamma}^{\text{inst}}\mathrm{d}\zeta_{\gamma}\wedge \mathrm{d}Z_{\gamma}\right)\right)
\end{equation}
\begin{equation}\label{invKF}
    \omega_{3}:=2\pi\left(\frac{1}{4}\langle \mathrm{d}Z\wedge \mathrm{d}\overline{Z}\rangle-\frac{1}{2} \langle \mathrm{d}\zeta\wedge \mathrm{d}\zeta \rangle -\sum_{\gamma}\Omega(\gamma)\left( \frac{\I}{2} V^{\text{inst}}_{\gamma}\mathrm{d}Z_{\gamma}\wedge \mathrm{d}\overline{Z}_{\gamma}+\mathrm{d}\zeta_{\gamma}\wedge A_{\gamma}^{\text{inst}}\right)\right)\,.
\end{equation}
Then the triple of real $2$-forms $(\omega_1,\omega_2, \omega_3)$ corresponds to the K\"{a}hler forms of a pseudo-HK structure\footnote{Our terminology is such that the signature of the metric is not assumed to be constant, in case $N$ has several components.} on $N$ if and only if $(M,g_M,\omega_M,\nabla,\xi,\Gamma)$ and $(M,\Gamma,Z,\Omega)$ are compatible.
\end{theorem}

\begin{definition}
We denote the resulting instanton corrected HK manifold from the previous theorem by $(N,g_N,\omega_1,\omega_2,\omega_3)$.
\end{definition}
\begin{remark}\label{hkremark}

\begin{itemize}
    \item Compared to \cite[Section 3]{CT}, we have rescaled the above 2-forms $\omega_i$ by a factor of $2\pi$, and rescaled $\zeta$ by $2\pi$ (i.e. $\zeta:\Gamma \to \mathbb{R}/\mathbb{Z}$ instead of $\zeta:\Gamma \to \mathbb{R}/2\pi\mathbb{Z}$). Furthermore, we have changed by a sign the convention of how the BPS indices $\Omega$ enter into the formulas \eqref{holsym}, \eqref{invKF} (i.e.\ the above formulas would correspond in \cite{CT} to the HK metric associated to the mutually local variation of BPS structures $(M,\Gamma,Z,-\Omega)$). We do this change of conventions in order to simplify the formulas taken from \cite[Section 4]{CT} below and also to be able to compare more easily with the physics literature in Section \ref{Sdualitysec} below.
    \item In the expressions \eqref{holsym} and \eqref{invKF} we are combining the wedge $\wedge$ with the pairing $\langle -, - \rangle$ on $\Gamma^*\otimes \mathbb{C}$. For example, with respect to a Darboux frame $(\widetilde{\gamma}_i,\gamma^i)$ of $\Gamma$, we have $\langle \mathrm{d}Z\wedge \D\overline{Z}\rangle=\D Z_{\widetilde{\gamma}_i}\wedge \D \overline{Z}_{\gamma^i}  -\D Z_{\gamma^i}\wedge \D \overline{Z}_{\widetilde{\gamma}_i}$ and $\langle \mathrm{d}\zeta \wedge \mathrm{d}\zeta \rangle=\mathrm{d}\zeta_{\widetilde{\gamma}_i}\wedge \mathrm{d}\zeta_{\gamma^i}-\mathrm{d}\zeta_{\gamma^i}\wedge\mathrm{d}\zeta_{\widetilde{\gamma}_i}=2\mathrm{d}\zeta_{\widetilde{\gamma}_i}\wedge \mathrm{d}\zeta_{\gamma^i}$. Furthermore, the expressions in \eqref{holsym} and \eqref{invKF}  are actually global and well-defined due to the monodromy invariance of $\Omega(\gamma)$, and the support and convergence property of the variations of BPS structures. 
    \item $(N,g_N,\omega_1,\omega_2,\omega_3)$ carries an infinitesimal rotating circle action \cite[Proposition 3.20]{CT}. Namely, there is a vector field $V$ on $N$ such that
\begin{equation}\label{rotvec}
    \mathcal{L}_V(\omega_1+\I\omega_2)=2\I(\omega_1+\I\omega_2), \quad \mathcal{L}_V\omega_3=0\,.
\end{equation}
Note that due to the factor $2$ in (\ref{rotvec}) the vector field $V$ is twice the vector field denoted $V$ in \cite{CT}.
    \item Under the mild assumption on the flow of $\xi$ that it generates a free-action on $M$ of the (multiplicative) monoid $\mathbb{R}_{\geq 1}$, we can guarantee that $g_N$ has signature $(4,4n)$ where $n+1=\text{dim}_{\mathbb{C}}(M)$ \cite[Proposition 3.21]{CT}.
    \item If one sets $\Omega(\gamma)=0$ for all $\gamma \in \Gamma$, then $(N,g_N,\omega_1,\omega_2,\omega_3)$ reduces to the semi-flat HK manifold obtained via the rigid c-map.
\end{itemize}
\end{remark}

\subsubsection{Associated instanton corrected QK manifold via HK/QK correspondence}\label{HK/QKsec}

The instanton corrected QK manifold $(\overline{N},g_{\overline{N}})$ associated to the data of $(M,g_M,\omega_M,\nabla,\xi,\Gamma)$ and $(M,\Gamma,Z,\Omega)$ was constructed in \cite{CT} by applying the HK/QK correspondence to $(N,g_N,\omega_1,\omega_2,\omega_3)$. In order to do this, one needs the following additional data:

\begin{itemize}
    \item A hyperholomorphic principal $S^1$-bundle $(\pi_N:P\to N,\eta)$, where $\pi_N:P\to N$ is constructed via \cite[Proposition 4.2]{CT}, and $\eta$ is a connection on $P$ having curvature
    \begin{equation}
        \D \eta=\pi_N^*(\omega_3 -\frac{1}{2}\D\iota_V g_N)\,.
    \end{equation}
    The connection $\eta$ is given by \cite[Corollary 4.5]{CT}
    \begin{equation}\label{etadef}
    \eta:=\Theta +\pi^*_N\Big(\frac{\pi \I}{2}\pi^*_M( \overline{\partial} r^2-\partial r^2) - \sum_{\gamma}2\pi \Omega(\gamma)\eta_{\gamma}^{\text{inst}}-\frac{1}{2}\iota_Vg_N\Big)
\end{equation}
where $r^2:=g_M(\xi,\xi)$; $V$ is the rotating vector field satisfying \eqref{rotvec}; $\Theta$ is another connection on $P$ having curvature $-\pi\langle \D \zeta \wedge \D \zeta \rangle=2\pi \mathrm{d}\zeta_{\gamma^i}\wedge \mathrm{d}\zeta_{\widetilde{\gamma}_i}$,
and 
\begin{equation}\label{etagammainst}
    \eta_{\gamma}^{\text{inst}}:=\frac{\I}{8\pi^2}\sum_{n>0}\frac{e^{2\pi \I n\zeta_{\gamma}}}{n}|Z_{\gamma}|K_1(2\pi n|Z_{\gamma}|)\Big(\frac{\D Z_{\gamma}}{Z_{\gamma}}-\frac{\D \overline{Z}_{\gamma}}{\overline{Z}_{\gamma}}\Big)\,.
\end{equation}
If $\sigma$ denotes a local coordinate for the $S^1$-fiber, then one can write

\begin{equation}\label{Thetacon}
    \Theta=\pi\left(\D\sigma -\pi_N^*\langle \zeta,\D\zeta \rangle\right)\,.
\end{equation}
\item We need to furthermore specify a Hamiltonian for $\omega_3$ with respect to the rotating vector field $V$ satisfying \eqref{rotvec}, which in this case is given by \cite[Lemma 4.7]{CT}

\begin{equation}\label{fdef}
    f=2\pi (r^2- 8c_{\ell} - \sum_{\gamma}\Omega(\gamma)\iota_V\eta_{\gamma}^{\text{inst}}),\quad c_{\ell} \in \mathbb{R},
\end{equation}
together with the lift $V^P$ of $V$ to $P$ given by

\begin{equation}
    V^P:=\widetilde{V}+f_3\partial_{\sigma}, \quad f_3:=f-\frac{1}{2}g_{N}(V,V)
\end{equation}
where $\widetilde{V}$ denotes the horizontal lift with respect to $\eta$ and $\partial_{\sigma}$ is the vertical vector field of $P$ generating the $S^1$-action.
\item Finally, we consider the open subset $N'\subset N$ given by
\begin{equation}\label{n'def}
        N'=\{p\in N \quad | \quad f(p)\neq 0, \quad f_3(p)\neq 0, \quad g_N(V_p,V_p)\neq 0\},
\end{equation}
and the $1$-forms on $P$ given by
\begin{equation}\label{thetadef}
    \theta_0^P=-\frac{1}{2}\pi_N^*\mathrm{d}f, \quad \theta_3^P:=\eta +\frac{1}{2}\pi_N^*\iota_Vg_N, \quad 
    \theta_1^P:=\frac{1}{2}\pi_N^*\iota_V\omega_2,\quad 
    \theta_2^P:=-\frac{1}{2}\pi_N^*\iota_V\omega_1\,, 
\end{equation}
\end{itemize}

We then have
\begin{theorem}\cite[Theorem 4.10]{CT}
Let $(M,g_M,\omega_M,\nabla,\xi,\Gamma)$ be an integral CASK manifold and $(M,\Gamma,Z,\Omega)$ a compatible mutually local variation of BPS structures. Furthermore, let $(N,g_N,\omega_1,\omega_2,\omega_3)$, $(\pi_N:P\to N,\eta)$, $f$, $f_3$, $\theta_i^P$, $V^P$ and $N'$ the associated data defined in the previous points. Given any submanifold $\overline{N}\subset P|_{N'}$ transverse to $V^P$, the symmetric $2$-tensor
\begin{equation}
    g_{\overline{N}}:=-\frac{1}{f}\left(\frac{2}{f_3}\eta^2+\pi^*_Ng_N-\frac{2}{f}\sum_{i=0}^3(\theta_i^P)^2\right)\Bigg|_{\overline{N}}
\end{equation}
defines a pseudo-QK metric on $\overline{N}$. Furthermore, if $(N,g_N,\omega_1,\omega_2,\omega_3)$ has signature $(4,4n)$, then $g_{\overline{N}}$ is positive definite on $\overline{N}_{+}=\overline{N}\cap\{f>0,f_3<0\}$.
\end{theorem}

\begin{remark}
Recall that we can guarantee that $(N,g_N,\omega_1,\omega_2,\omega_3)$ has signature $(4,4n)$ if the flow of $\xi$ generates a free-action on $M$ of the monoid $\mathbb{R}_{\geq 1}$ (see Section \ref{instHKsec}). 
\end{remark}
\subsubsection{The case of a CASK domain}\label{QKinstdomainsec}

We now specialize the previous construction to the case of a CASK domain. This will be the case of interest in the following Section \ref{Sdualitysec}.

\begin{definition}\label{CASKdomaindef}
A CASK domain is a tuple $(M,\mathfrak{F})$ where 
\begin{itemize}
    \item $M\subset \mathbb{C}^{n+1}-\{0\}$ is a $\mathbb{C}^{\times}$-invariant domain. We denote the canonical holomorphic coordinates by $Z^i$, $i=0,1,...,n$. To avoid inessential coordinate changes later on, we 
    will assume for simplicity that $Z^0$ does not vanish on $M$.
    \item $\mathfrak{F}:M\to \mathbb{C}$ is a holomorphic function, homogeneous of degree $2$ with respect to the natural $\mathbb{C}^{\times}$-action on $M$.
    \item The matrix 
    \begin{equation}
        \text{Im}\left(\tau_{ij}\right), \quad \tau_{ij}:=\frac{\partial^2 \mathfrak{F}}{\partial Z^i\partial Z^j}
    \end{equation}
    has signature $(n,1)$, and $\text{Im}(\tau_{ij})Z^i\overline{Z}^j<0$.
\end{itemize}
\end{definition}

A CASK domain $(M,\mathfrak{F})$ induces in the usual way a CASK manifold $(M,g_M,\omega_M,\nabla,\xi)$ \cite{ACD}. With our conventions on the signature of the CASK manifold, if  $Z_i:=\frac{\partial \mathfrak{F}}{\partial Z^i}$, then $\{Z^i\}$ and $\{-Z_i\}$ are a global system of conjugate conical special holomorphic coordinates. If $x^i=\text{Re}(Z^i)$ and $y_i:=\text{Re}(Z_i)$, then $\nabla$ is defined such that $\D x^i$ and $\D y_i$ are flat. Furthermore

\begin{equation}\label{CASKconvention}
    g_M=-\text{Im}(\tau_{ij})\D Z^i \D\overline{Z}^j, \quad \omega_M=-\frac{\mathrm{i}}{2}\text{Im}(\tau_{ij})\D Z^i\wedge \D\overline{Z}^j=\D x^i\wedge \D y_i, \quad \xi=Z^i\partial_{Z^i} +\overline{Z}^i\partial_{\overline{Z}^i}\,.
\end{equation}
Given a CASK domain $(M,\mathfrak{F})$ we can induce a canonical integral structure on the CASK manifold by defining $\Gamma \to M$ to be $\Gamma=\text{span}_{\mathbb{Z}}\{\partial_{x^i},\partial_{y_i}\}$. In the following, we will assume that:
\begin{itemize}
    \item $(M,g_M,\omega_M,\nabla,\xi,\Gamma)$ is an integral CASK manifold induced by a CASK domain $(M,\mathfrak{F})$ with the canonical integral structure. In this case, we will sometimes use the notation $(\partial_{x^i},\partial_{y_i})=(\widetilde{\gamma}_i,\gamma^i)$.
    \item Given a mutually local variation of BPS structures $(M,\Gamma,Z,\Omega)$ with $(M,\Gamma)$ as in the previous point, we assume that $Z$ is the canonical central charge, and that $\text{Supp}(\Omega)\subset \text{span}_{\mathbb{Z}}\{\partial_{y_i}\}$. In particular, the canonical central charge satisfies in this case
    \begin{equation}
        Z=Z_{\widetilde{\gamma}_i}\mathrm{d}x^i+Z_{\gamma^i}\mathrm{d}y_i=-Z_i\D x^i+Z^i\D y_i\,.
    \end{equation}
\end{itemize}

In order to construct the associated QK metric, we need to choose $\overline{N}\subset P|_{N'}$ transverse to $V^P$.  Denoting by $\pi_M:=\pi\circ \pi_N$ the composition of the projections $\pi_N:P\to N$, and $\pi:N\to M$, we have that $\pi_M(p)=(Z^0,...,Z^n)$. We define $\overline{N}$ by

\begin{equation}\label{QKCASKdomain}
    \overline{N}:=\{p \in P|_{N'} \quad| \quad  \text{Arg}(Z^0)=0,\quad \text{ where $\pi_M(p)=(Z^0,...,Z^n)$} \}\,.
\end{equation}
\begin{definition}\label{QKCASKcoords}
Throughout the paper, we will use coordinates $(\rho,z^a,\zeta^i,\widetilde{\zeta}_i,\sigma)$ on  \eqref{QKCASKdomain} defined as follows:

\begin{itemize}
    \item $16\pi\rho:=2\pi r^2 -16\pi c_{\ell}$ where $r^2=g_M(\xi,\xi) = -\text{Im}(\tau_{ij}) Z^i \overline{Z}^j$ and $c_{\ell}\in \mathbb{R}$. Note that $r^2$ is a K\"ahler potential for the affine special K\"ahler metric $g_M$.
    \item $z^a:=Z^a/Z^0$ for $a=1,...,n$. These are global holomorphic coordinates on the induced projective special K\"{a}hler (PSK) manifold $(\overline{M},g_{\overline{M}})$ induced by the CASK domain. In particular, we have a projection $\pi_{\overline{M}}: M\to \overline{M}$.
    \item $(\zeta^i,\widetilde{\zeta}_i)$ are given by $\zeta^i:=-\zeta_{\partial_{y_i}}=-\zeta_{\gamma^i}$ and $\widetilde{\zeta}_i:=\zeta_{\partial_{x^i}}=\zeta_{\widetilde{\gamma}_i}$, where the latter are the evaluation map on $N$ contracted with $\partial_{x^i}$ and $\partial_{y_i}$.
    \item $\sigma$ is a local coordinate for the $S^1$-fiber of $\pi_N:P\to N$ satisfying \eqref{Thetacon}.
\end{itemize}    
\end{definition}

\begin{remark}
In the string theory setting, the coordinates $(\rho,z^a,\zeta^i,\widetilde{\zeta}_i,\sigma)$ can be identified with certain fields from the type IIA hypermultiplet. Namely,  $\rho$ is the 4d-dilaton, $z^a$ are coordinates for the complex moduli of the Calabi-Yau, $(\zeta^i,\widetilde{\zeta}_i)$
are the RR-axions, and $\sigma$ is the NS-axion. Furthermore, the constant $c_{\ell}$ appearing in the definition of $\rho$ is identified with the $1$-loop correction to the tree-level Ferrara-Sabharwal metric. 
\end{remark}

In \cite[Theorem 5.4]{CT} an explicit expression of the resulting QK metric $(\overline{N},g_{\overline{N}})$ in these coordinates is given (with slightly different conventions). In order to write down the formula we introduce the following notation:

\begin{itemize}
    \item We denote by $\widetilde{Z}_{\gamma}:=Z_{\gamma}/Z^0$ the normalized central charge. In particular, we have $\widetilde{Z}_{\gamma^i}=z^i$ with $z^0=1$. If furthermore we let $\mathcal{K}=-\log(-2\text{Im}(\tau_{ij})z^i\overline{z}^j)$, then $\mathcal{K}$ is a global K\"{a}hler potential for the projective special K\"{a}hler manifold $(\overline{M},g_{\overline{M}})$ induced by the CASK domain. Note that     
    \begin{equation} \label{r2:eq}r^2 = |Z^0|^2\frac{e^{-\mathcal{K}}}{2}.\end{equation}
    \item We denote $N_{ij}=-2\text{Im}(\tau_{ij})$. If $\gamma \in \text{Supp}(\Omega)$ we write $\gamma=q_i(\gamma)\gamma^i$ (recall that we assume that $\text{Supp}(\Omega)\subset \text{span}_{\mathbb{Z}}\{\gamma^i\}$), and define
    \begin{equation}
        N_{ij}^{\text{inst}}:=-2\sum_{\gamma}\Omega(\gamma)V^{\text{inst}}_{\gamma}q_i(\gamma)q_j(\gamma)\,.
    \end{equation}
    \item We let
    \begin{equation}\label{instW}
    W_i:=\D\zeta_{\widetilde{\gamma}_i}+\tau_{ij}\D\zeta_{\gamma^j}, \;\;\;\;\;\; W_i^{\text{inst}}:=-\sum_{\gamma}\Omega(\gamma)q_i(\gamma)( A_{\gamma}^{\text{inst}}-\I V_{\gamma}^{\text{inst}}\D\zeta_{\gamma}),
\end{equation}
and 
\begin{equation}\label{etainstdef}
    \eta^{\text{inst}}:=-\sum_{\gamma}\Omega(\gamma) \eta_{\gamma}^{\text{inst}}-\frac{1}{2}\iota_V\left(\frac{g_N}{2\pi}- \pi^*_Mg_{M}\right) \,.
\end{equation}
\item We split $f=16\pi\rho+f^{\text{inst}}=16\pi(\rho+\rho^{\text{inst}})$ in \eqref{fdef} using that $16\pi \rho=2\pi r^2-16\pi c_{\ell}$, and where $f^{\text{inst}}=16\pi \rho^{\text{inst}}$ contains the terms with the BPS indices $\Omega(\gamma)$, namely
    
    \begin{equation}
        f^{\text{inst}}=-2\pi\sum_{\gamma}\Omega(\gamma)\iota_V\eta_{\gamma}^{\text{inst}}\,.
    \end{equation}
    Finally, we denote $ f_3^{\text{inst}}=16\pi \rho_3^{\text{inst}}=2\pi \iota_V \eta^{\text{inst}}$. We remark that in the case where $\Omega(\gamma)=0$ for all $\gamma$ we have $f^{\text{inst}}=f_3^{\text{inst}}=0$, and similarly for all the other quantities with an $^\text{inst}$ superscript. 
\end{itemize}

We then have

\begin{theorem}\cite[Theorem 5.4, Proposition 5.6]{CT}\label{QKCASKdomainformulas} Let $(M,\mathfrak{F})$ and $(M,\Gamma,Z,\Omega)$ be as before.  By possibly restricting $M$, we assume that $M$ is the maximal open subset where $(M,g_M,\omega_M,\nabla,\xi,\Gamma)$ and $(M,\Gamma,Z,\Omega)$ are compatible. Furthermore, let $\overline{N}$ be as \eqref{QKCASKdomain}. Then in the coordinates $(\rho,z^a,\zeta^i,\widetilde{\zeta}_i,\sigma)$ from Definition \ref{QKCASKcoords} the instanton corrected QK metric $(\overline{N},g_{\overline{N}})$ associated to $(M,g_M,\omega_M,\nabla,\xi,\Gamma)$ and $(M,\Gamma,Z,\Omega)$ has the form:

\begin{equation}\label{coordQKmetric2}
    \begin{split}
        g_{\overline{N}}=& \frac{\rho+c_{\ell}}{\rho+\rho^{\text{inst}}}\Big(g_{\overline{M}}+2e^{\mathcal{K}}\sum_{\gamma}\Omega(\gamma)V_{\gamma}^{\text{inst}}\Big|\mathrm{d}\widetilde{Z}_{\gamma} + \widetilde{Z}_{\gamma}\Big(\frac{\D\rho}{2(\rho+c_{\ell})}+\frac{\D\mathcal{K}}{2}\Big)\Big|^2\Big)\\
        &+\frac{1}{2(\rho+\rho^{\text{inst}})^2}\Big(\frac{\rho +2c_{\ell} -\rho^{\text{inst}}}{2(\rho+c_{\ell})}\D\rho^2+2\D\rho \D \rho^{\text{inst}}|_{\overline{N}}+(\D \rho^{\text{inst}})^2|_{\overline{N}}\Big)\\
        &+\frac{\rho+c_{\ell}+\rho_-^{\text{inst}}}{64(\rho+\rho^{\text{inst}})^2(\rho+2c_{\ell}-\rho_3^{\text{inst}})}\Big(\D \sigma -\langle \zeta,\D\zeta \rangle -4 c_{\ell}\D^c\mathcal{K}+\eta_+^{\text{inst}}|_{\overline{N}}+\frac{\rho_+^{\text{inst}}-c_{\ell}}{\rho+c_{\ell}+\rho_-^{\text{inst}}}\eta_-^{\text{inst}}|_{\overline{N}}\Big)^2\\
        &-\frac{1}{4(\rho+\rho^{\text{inst}})}(W_i+W_i^{\text{inst}}|_{\overline{N}})(N+N^{\text{inst}})^{ij}(\overline{W}_j+\overline{W}_j^{\text{inst}}|_{\overline{N}}) \\
        &+\frac{(\rho+c_{\ell})e^{\mathcal{K}}}{2(\rho+\rho^{\text{inst}})^2}\Big|z^i(W_i+W_i^{\text{inst}}|_{\overline{N}})- \frac{\mathrm{i}}{2}\sum_{\gamma}\Omega(\gamma) A_{\gamma}^{\text{inst}}(V)\Big(\D\widetilde{Z}_{\gamma} + \widetilde{Z}_{\gamma}\Big(\frac{\D\rho}{2(\rho+c_{\ell})}+\frac{\D \mathcal{K}}{2}\Big)\Big)\Big|^2\\
        &
        +\frac{\rho+c_{\ell}+\rho_-^{\text{inst}}}{\rho+\rho^{\text{inst}}}\Big(\frac{\D^c\mathcal{K}}{2}+\frac{1}{8(\rho+c_{\ell}+\rho_-^{\text{inst}})}\eta_-^{\text{inst}}|_{\overline{N}}\Big)^2-\frac{\rho+c_{\ell}}{\rho+\rho^{\text{inst}}}\Big(\frac{\D^c\mathcal{K}}{2}\Big)^2
    \end{split}
\end{equation}
where $\mathrm{d}^c:=\mathrm{i}(\overline{\partial}-\partial)$, $\eta^{\text{inst}}_{\pm}$ are given by
\begin{equation}
    \eta_{\pm}^{\text{inst}}:=\Big( \eta^{\text{inst}}-4\rho_3^{\text{inst}}\widetilde{\eta}\Big)\pm \Big( \sum_{\gamma}\Omega(\gamma)\eta_{\gamma}^{\text{inst}}-4\rho^{\text{inst}}\widetilde{\eta}\Big)\,, \quad \widetilde{\eta}:=\mathrm{d}^c \log(r)\,,
\end{equation}
and 
\begin{equation}
    \rho_{\pm}^{\text{inst}}:=(\rho^{\text{inst}}\pm \rho_3^{\text{inst}})/2 \,.
\end{equation}
Furthermore, the open subset $\overline{N}_{+}=\overline{N}\cap\{f>0,f_3<0\}$ is non-empty, and $g_{\overline{N}}$ is positive-definite on $\overline{N}_{+}$.
\end{theorem}

\begin{remark}

\begin{itemize}
    \item In Theorem~\ref{QKCASKdomainformulas} we have relaxed a bit the possible restriction of $M$ compared to \cite[Theorem 5.4]{CT}. In \cite{CT} we assume that we can restrict to an $M$ invariant under the action of the monoid $\mathbb{R}_{\geq 1}\times S^1$ to make the CASK structure compatible with the BPS structure. This ensures that no matter the point of $\overline{M}$ the metric is defined for $\rho>K$ for some sufficiently big uniform $K$. Our weakened assumption makes it so that the constant $K$ might depend on the point $z^a\in \overline{M}$. 
    \item When $\Omega(\gamma)=0$ for all $\gamma \in \Gamma$ the expression reduces to the 1-loop corrected Ferrara-Sabharwal metric:
    \begin{equation}
    \begin{split}
       g_{\overline{N}}=& \frac{\rho+c_{\ell}}{\rho}g_{\overline{M}}+\frac{\rho +2c_{\ell}}{4\rho^2(\rho+c_{\ell})}\D\rho^2+\frac{\rho+c_{\ell}}{64\rho^2(\rho+2c_{\ell})}\Big(\D \sigma -\langle \zeta,\D\zeta \rangle -4 c_{\ell}\D^c\mathcal{K}\Big)^2\\
        &-\frac{1}{4\rho}\left(N^{ij}-\frac{2(\rho+c_{\ell})e^{\mathcal{K}}}{\rho }z^i\overline{z}^j\right)W_i\overline{W}_j\,. 
    \end{split}
\end{equation}
In particular $(\overline{N},g_{\overline{N}})$ can be thought as a deformation of the $1$-loop corrected metric.
    \item Since the instanton corrections of the HK geometry are exponentially suppressed as $|Z_{\gamma}|\to \infty$ for $\gamma \in \text{Supp}(\Omega)$, it is easy to check that the possibly restricted $M$ from above satisfying the required conditions is never empty. Furthermore, on $\overline{N}$ the instanton corrections of the QK geometry are exponentially suppressed as $\rho \to \infty$ (due to the relation $|Z_{\gamma}|=|Z^0||\widetilde{Z}_{\gamma}|=4\sqrt{\rho+c_{\ell}}e^{\mathcal{K}/2}|\widetilde{Z}_{\gamma}|$), so we can ensure that $(\rho,z^a,\zeta^i,\widetilde{\zeta}_i,\sigma)\in \overline{N}_{+}$ by taking $\rho$ sufficiently big.
    \item The function $f=16\pi\rho+f^{\text{inst}}$ can be thought  in the string theory setting as the D-instanton corrected 4d dilaton (up to different conventions in the normalization) \cite{AMNP}. 
    \item When comparing \eqref{coordQKmetric2} to \cite[Equation 5.5]{CT}, we note that here we are using different conventions for the normalization of $g_{N}$; the rotating vector field $V$; the functions $N_{ij}$, $N_{ij}^{\text{inst}}$,  $e^{-\mathcal{K}}$, $\eta_{\pm}^{\text{inst}}$; the signature of $\text{Im}(\tau_{ij})$; and the coordinates $(\rho,z^a,\zeta^i,\widetilde{\zeta}_i,\sigma)$. Compared to \cite{CT}, $g_{N}$ is scaled by $2\pi$; 
$V$, $N_{ij}$, $N_{ij}^{\text{inst}}$ and $e^{-\mathcal{K}}$ are scaled by $2$; $\eta_{\pm}^{\text{inst}}$ are scaled by $\pi^{-1}$; and the signature of $\text{Im}(\tau_{ij})$ is opposite. Furthermore, the coordinates and $1$-loop constant of \cite{CT} are related to the ones in Definition \ref{QKCASKcoords} by performing the scaling
\begin{equation}\label{scaling}
    \rho \to 16\pi\rho,\quad c_{\ell}\to 16\pi c_{\ell}, \quad \sigma \to \pi \sigma, \quad \zeta^i\to -2\pi\zeta^i, \quad \widetilde{\zeta}_i \to 2\pi \widetilde{\zeta}_i\,.
\end{equation}
 Finally, as mentioned in Remark \ref{hkremark}, the sign with which the $\Omega(\gamma)$ enter the formula in \eqref{coordQKmetric2} is opposite to \cite[Equation 5.5]{CT}. This changes of convention will make the formulas from subsequent section look more simple and more easily comparable to the physics literature. 
\end{itemize}
\end{remark}

In what follows, it will be useful to consider the following lift of $(\overline{N},g_{\overline{N}})$:

\begin{definition}\label{liftdef}
We will denote by $(\widetilde{N},g_{\overline{N}})$ the QK manifold obtained by lifting the QK metric $(\overline{N},g_{\overline{N}})$ obtained in Theorem \ref{QKCASKdomainformulas} to the open subset $\widetilde{N}\subset \mathbb{R}_{>0}\times \overline{M}\times \mathbb{R}^{2n+2}\times \mathbb{R}$ obtained by considering $(\zeta^i,\widetilde{\zeta}_i,\sigma)\in \mathbb{R}^{2n+2}\times \mathbb{R}$ as (non-periodic) global coordinates of $\mathbb{R}^{2n+2}\times \mathbb{R}$. We will call such a space an instanton corrected c-map space.
\end{definition}

\subsection{Twistor space description and Darboux coordinates}\label{twisdescsec}

Let $(\overline{N},g_{\overline{N}},Q)$ denote a QK manifold, where $Q\to \overline{N}$ denotes the associated quaternionic structure. Namely, a parallel subbundle $Q\subset \text{End}(T\overline{N})$ admiting local trivializations $(J_1,J_2,J_3)$ by skew endomorphisms satisfying the quaternion relations. The structure of a QK manifold $(\overline{N},g_{\overline{N}},Q)$ can be encoded in a holomorphic object, known as the twistor space $(\mathcal{Z},\mathcal{I},\lambda,\tau)$ \cite{Salamon1982,LeBrun1989}. Here $\mathcal{Z}\to \overline{N}$ is a sphere subbundle of $Q$ defined by

\begin{equation}
    \mathcal{Z}_p:=\{J\in Q_p \; |\; J^{2}=-1\}, \quad p \in \overline{N};
\end{equation}
$\mathcal{I}$ is a canonical holomorphic structure on $\mathcal{Z}$; $\lambda \in \Omega^1(\mathcal{Z},\mathcal{L})$ defines a holomorphic contact structure on $\mathcal{Z}$, where $\mathcal{L}\to \mathcal{Z}$ is a certain holomorphic line bundle; and $\tau$ is a real structure on $\mathcal{Z}$ (i.e.\ an antiholomorphic involution).\\

In what follows, we consider $(\overline{N},g_{\overline{N}})$ built in the previous Section \ref{QKinstdomainsec}, and its lift  $(\widetilde{N},g_{\overline{N}})$. We start by recalling the following from the discussion in \cite[Section 4]{CTSduality}:

\begin{proposition} \label{localcontprop} Let $(\widetilde{N},g_{\overline{N}})$ be the lift of an instanton corrected QK manifold associated to a CASK domain $(M,\mathfrak{F})$ and mutually local variation of BPS structures $(M,\Gamma,Z,\Omega)$. Then $\mathcal{Z}\cong \widetilde{N}\times \mathbb{C}P^1$ (non-holomorphically), and there is a holomorphic coordinate $t$ on the $\mathbb{C}P^1$ factor and a holomorphic section $s$ of the holomorphic line bundle $\mathcal{L}\to \mathcal{Z}$ vanishing at $t=0,\infty$, such that

\begin{equation}\label{contactlocal}
    \lambda=\left(f\frac{\D t}{t}+t^{-1}\theta_{+}^P|_{\overline{N}}-2\mathrm{i} \theta_3^P|_{\overline{N}} +t \theta_{-}^{P}|_{\overline{N}}\right)\cdot s
\end{equation}
where $\theta_{\pm}^P:=\theta_1^P\pm \mathrm{i}\theta_2^P$, and $\theta^P_i$ for $i=1,2,3$ are defined as in \eqref{thetadef}, and $f$ is defined as in \eqref{fdef}.

\end{proposition}
\begin{proof}
This follows from the discussion in \cite[Section 4.3]{CTSduality}, where the twistor space of a QK manifold $(\overline{N},g_{\overline{N}})$ obtained via HK/QK correspondence is described in terms of the ``HK data" given by $(N,g_{N},\omega_1,\omega_2,\omega_3)$, $(\pi_N:P\to N,\eta)$, $f$, and the associated HK cone. In particular, whenever the QK manifold admits a global chart of coordinates, which is the case of $(\widetilde{N},g_{\overline{N}})$ obtained in Section \ref{QKinstdomainsec}, it follows that $\mathcal{Z}\cong \widetilde{N}\times \mathbb{C}P^1$ non-holomorphically.  The formula \eqref{contactlocal} follows from \cite[Section 4.3.1]{CTSduality}. 
The lift to $\widetilde{N}\rightarrow \overline{N}$ is understood.
\end{proof}

In what follows, we will be concerned in describing Darboux coordinates for the contact structure $\lambda$ expressed as \eqref{contactlocal}, in the case where $(\widetilde{N},g_{\overline{N}})$ is the instanton corrected QK metric obtained in Section \ref{QKinstdomainsec}. For this it will be important to have an explicit expression for $f$, $\theta_{+}^P|_{\overline{N}}$ and $\theta_3^P|_{\overline{N}}$ ($\theta_-^P|_{\overline{N}}$ can be obtained from $\theta_-^P=\overline{\theta_+^P}$).

\begin{lemma}\label{thetalemma1}
Consider the CASK domain $(M,\mathfrak{F})$ and mutually local variation of BPS structures $(M,\Gamma,Z,\Omega)$ as in Section \ref{QKinstdomainsec}. Then $f$, $\theta_{+}^P|_{\overline{N}}$ and $\theta_3^P|_{\overline{N}}$ from \eqref{contactlocal} have the following formulas with respect to the coordinates $(\rho,z^a,\zeta^i,\widetilde{\zeta}_i,\sigma)$:

\begin{align*}
        f&=16\pi\rho + \frac{2R}{\pi}\sum_{\gamma}\Omega(\gamma)\sum_{n>0}\frac{e^{-2\pi\mathrm{i}n\zeta_{\gamma}}}{n}|\widetilde{Z}_{\gamma}|K_1(4\pi Rn|\widetilde{Z}_{\gamma}|)\\
        \theta^{P}_{+}|_{\overline{N}}&=-4\pi R\langle \widetilde{Z},\D\zeta \rangle +2\mathrm{i}R\sum_{\gamma}\Omega(\gamma)\widetilde{Z}_{\gamma}\sum_{n>0}e^{-2\pi\mathrm{i}n\zeta_{\gamma}}K_0(4\pi Rn|\widetilde{Z}_{\gamma}|)\D\zeta_{\gamma}\\
    &\quad \quad +2R^2\sum_{\gamma}\Omega(\gamma)\widetilde{Z}_{\gamma}\sum_{n>0}e^{-2\pi \mathrm{i}n\zeta_{\gamma}}|\widetilde{Z}_{\gamma}|K_1(4\pi R n|\widetilde{Z}_{\gamma}|)\left(\frac{\D\widetilde{Z}_{\gamma}}{\widetilde{Z}_{\gamma}}+\frac{\D\overline{\widetilde{Z}}_{\gamma}}{\overline{\widetilde{Z}}_{\gamma}}+\frac{\D\rho}{(\rho+c_{\ell})}+\D\mathcal{K}\right)\\
    \theta_3^P|_{\overline{N}}&=\pi\D\sigma -\pi \langle\zeta, \D\zeta \rangle-4\pi(\rho+c_{\ell})\D^c\mathcal{K}-\frac{\mathrm{i}R}{2\pi}\sum_{\gamma}\Omega(\gamma)\sum_{n>0}\frac{e^{-2\pi \mathrm{i}n\zeta_{\gamma}}}{n}|\widetilde{Z}_{\gamma}|K_1(4\pi Rn|\widetilde{Z}_{\gamma}|)\left(\frac{\D\widetilde{Z}_{\gamma}}{\widetilde{Z}_{\gamma}}-\frac{\D\overline{\widetilde{Z}}_{\gamma}}{\overline{\widetilde{Z}}_{\gamma}}\right) \numberthis \label{fthetaexpres}
    \end{align*}
where $R:=2\sqrt{\rho+c_{\ell}}e^{\mathcal{K}/2}$, $\widetilde{Z}_{\gamma}=q_0+q_az^a$ and  $\zeta_{\gamma}=-q_i\zeta^i$ for $\gamma=q_i\gamma^i$, and in the last formula we have used $\mathrm{d}^c=\mathrm{i}(\overline{\partial}-\partial)$.
\end{lemma}
\begin{proof}
To obtain the above identities we will use that 
\begin{equation}\label{z0rrel}
    |Z^0|^2=16(\rho+c_{\ell})e^{\mathcal{K}}=4R^2\,,
\end{equation}
together with $\Omega(\gamma)=\Omega(-\gamma)$  and the fact that the rotating vector field is given globally in the case of a CASK domain by 
\begin{equation}\label{vcaskd}
    V=2\mathrm{i}Z^i\partial_{Z^i}-2\mathrm{i}\overline{Z}^i\partial_{\overline{Z}^i}\,.
\end{equation}
To obtain the formula for $f$ we just use \eqref{fdef} together with \eqref{z0rrel} and a relabeling of the sum variable $\gamma \to -\gamma$. To obtain the formula for $\theta_{+}^P|_{\overline{N}}$ we use the formulas \eqref{holsym} for $\omega_1+\mathrm{i}\omega_2$, the definitions for $\theta_1^P$ and $\theta_2^P$ in \eqref{thetadef}, a relabeling $\gamma \to -\gamma$ of the sums over $\gamma$, the CASK relation $Z_{\widetilde{\gamma}_i}=-Z_i=-\tau_{ij}Z^j=-\tau_{ij}Z_{\gamma^j}$, and the fact that

\begin{equation}\label{restrictionrel}
    \left(\frac{\D Z_{\gamma}}{Z_{\gamma}}-\frac{\D\overline{Z}_{\gamma}}{\overline{Z}_{\gamma}}\right)\Bigg|_{\overline{N}}=\left(\frac{\D\widetilde{Z}_{\gamma}}{\widetilde{Z}_{\gamma}}-\frac{\D\overline{\widetilde{Z}}_{\gamma}}{\overline{\widetilde{Z}}_{\gamma}}\right), \quad \D Z_{\gamma}|_{\overline{N}}=|Z^0|\widetilde{Z}_{\gamma}\left(\frac{\D\widetilde{Z}_{\gamma}}{\widetilde{Z}_{\gamma}} +\frac{\D \rho}{2(\rho+c_{\ell})}+\frac{\D\mathcal{K}}{2}\right)\,.
\end{equation}
Finally, the equality for $\theta_3^P|_{\overline{N}}$ follows from the formulas \eqref{etadef} for $\eta$, \eqref{thetadef} for $\theta_3^P$, the first equation from \eqref{restrictionrel}, and again a relabeling $\gamma \to -\gamma$ in the sum over $\gamma$. We also use the fact that

\begin{equation}
    \D^c r^2|_{\overline{N}}=-r^2\D^c \mathcal{K}\,,
\end{equation}
which follows from the second relation in \eqref{restrictionrel}.
\end{proof}

\begin{remark}
    We remark that if we denote $F_i:=\partial_{Z^i}\mathfrak{F}/Z^0=Z_i/Z^0$, then in the expression \eqref{fthetaexpres} we have
    \begin{equation}
        \langle \widetilde{Z},\mathrm{d}\zeta\rangle= \widetilde{Z}_{\widetilde{\gamma}_i}\D\zeta_{\gamma^i}-\widetilde{Z}_{\gamma^i}\D\zeta_{\widetilde{\gamma}_i}=(-F_i)\D(-\zeta^i) - z^i\D\widetilde{\zeta}_i=F_i\D\zeta^i-z^i\mathrm{d}\widetilde\zeta_i\,.
    \end{equation}
    Similarly, we have
    \begin{equation}
        \langle \zeta,\mathrm{d}\zeta\rangle=\zeta_{\widetilde{\gamma}_i}\D\zeta_{\gamma^i}-\zeta_{\gamma^i}\D\zeta_{\widetilde{\gamma}_i}=-\widetilde{\zeta}_i\mathrm{d}\zeta^i+\zeta^i\mathrm{d}\widetilde{\zeta}_i\,.
    \end{equation}
\end{remark}

\subsubsection{Darboux coordinates for c-map spaces associated to CASK domains} 

In this section we focus on the easier case of a c-map space associated to a CASK domain. The following coordinates have been been previously obtained in the physics literature in \cite{NPV} in the $c_{\ell}=0$ case, and in the 1-loop corrected case in \cite{linearpertQK} via slightly different methods. 

\begin{proposition}\label{typeIIA1loop} Consider the QK manifold $(\widetilde{N},g_{\overline{N}})$ obtained from a CASK domain $(M,\mathfrak{F})$ via the c-map (i.e.\ by taking $\Omega(\gamma)=0$ for all $\gamma$ in our previous constructions). If $F_i:=Z_i/Z^0=\partial_{Z^i} \mathfrak{F}/Z^0$, $t$ denotes the twistor fiber coordinate from  Proposition \ref{localcontprop} and $R=2\sqrt{\rho+c_{\ell}}e^{\mathcal{K}/2}$, then the functions on the twistor space given by

\begin{equation}\label{Darbouxcoords1loop}
    \begin{split}
        \xi^{i}&=\zeta^{i}- \mathrm{i}R(t^{-1}z^i +t\overline{z}^i)\\
        \widetilde{\xi}_i&=\widetilde{\zeta}_{i}- \mathrm{i}R(t^{-1}F_{i}+t\overline{F}_i)\\
        \alpha&= \sigma -\mathrm{i}R(t^{-1}\langle \widetilde{Z},\zeta \rangle +t \langle \overline{\widetilde{Z}}, \zeta \rangle)-8\mathrm{i} c_{\ell}\log(t)\,,
    \end{split}
\end{equation}
define Darboux coordinates for the holomorphic contact structure $\lambda$, in the sense that

\begin{equation}
    \lambda=-2\pi\mathrm{i}(\D\alpha +\widetilde{\xi}_i\D\xi^i-\xi^i\D\widetilde{\xi}_i)\cdot s\,.
\end{equation}

\end{proposition}
\begin{proof}
By Proposition \ref{localcontprop}, it is enough to check that 

\begin{equation}
    -2\pi \I(\D\alpha +\widetilde{\xi}_i\D\xi^i-\xi^i\D\widetilde{\xi}_i)=f\frac{\D t}{t}+t^{-1}\theta_{+}^P|_{\overline{N}}-2\mathrm{i} \theta_3^P|_{\overline{N}} +t \theta_{-}^{P}|_{\overline{N}}
\end{equation}
where $f$, $\theta_{+}^P=\overline{\theta_{-}^P}$ and $\theta_3^P$ are obtained by setting $\Omega(\gamma)=0$ for all $\gamma$ in \eqref{fthetaexpres}. That is:

\begin{equation}
    f=16\pi \rho, \quad\quad \theta_3^P|_{\overline{N}}=\pi\D\sigma -\pi\langle\zeta, \D\zeta \rangle-4\pi(\rho+c_{\ell})\D^c\mathcal{K}, \quad\quad
    \theta_{+}^P|_{\overline{N}}=-4\pi R\langle \widetilde{Z},\D\zeta \rangle\,.
\end{equation}

We now compute

\begin{equation}
    \begin{split}
        \widetilde{\xi}_i\D\xi^i-\xi^i\D\widetilde{\xi}_i=&-\langle \zeta, \D\zeta \rangle +\I(t^{-1}\langle\widetilde{Z},\zeta \rangle +t \langle\overline{\widetilde{Z}}, \zeta \rangle)\D R +\I R(-t^{-2}\langle\widetilde{Z},\zeta\rangle \D t+\langle \overline{\widetilde{Z}},\zeta \rangle \D t)+ 8\I(\rho+c_{\ell})\frac{\D t}{t}\\
        & +\I R(t^{-1}\langle \D \widetilde{Z},\zeta \rangle + t \langle \D\overline{\widetilde{Z}},\zeta \rangle) -\I R(t^{-1}\langle \widetilde{Z},\D\zeta \rangle + t \langle \overline{\widetilde{Z}},\D\zeta \rangle)-4(\rho+c_{\ell})\D^c\mathcal{K}\,,
    \end{split}
\end{equation}
where for the terms $8\I (\rho+c_{\ell})\frac{\D t}{t}$ and $-4(\rho+c_{\ell})\D^c\mathcal{K}$ we have used that the CASK relation $F_i=\tau_{ij}z^j$ implies

\begin{equation}
    \begin{split}
    \Big(-8(\rho+c_{\ell})e^{\mathcal{K}}F_i\overline{z}^i+8(\rho+c_{\ell})e^{\mathcal{K}}\overline{F}_iz^i\Big)\frac{\D t}{t}&=    \Big(-8(\rho+c_{\ell})e^{\mathcal{K}}\tau_{ij}z^j\overline{z}^i+8(\rho+c_{\ell})e^{\mathcal{K}}\overline{\tau}_{ij}\overline{z}^jz^i\Big)\frac{\D t}{t}\\
    &=\Big(8\I(\rho+c_{\ell})e^{\mathcal{K}}(-2\text{Im}(\tau_{ij}))z^j\overline{z}^i\Big)\frac{\D t}{t}=8\I(\rho+c_{\ell})\frac{\D t}{t}\\
    \end{split}
\end{equation}
and by using the relation $\D F_i=\tau_{ij}\D z^i$ 
\begin{equation}
    \begin{split}
    -4(\rho+c_{\ell})e^{\mathcal{K}}F_i\D \overline{z}^i&-4(\rho+c_{\ell})e^{\mathcal{K}}\overline{F}_i\D z^i +4(\rho+c_{\ell})e^{\mathcal{K}}\D \overline{F}_iz^i+4(\rho+c_{\ell})e^{\mathcal{K}}\D F_i\overline{z}^i\\
    &=-4(\rho+c_{\ell})e^{\mathcal{K}}(-2\text{Im}(\tau_{ij}))\Big(i\overline{z}^j\D z^i -iz^i\D \overline{z}^j\Big)\\
    &=-4(\rho+c_{\ell})\D ^c\mathcal{K}\,.
    \end{split}
\end{equation}
(Recall that $\mathcal{K}=-\log K$ where $K= -2 \mathrm{Im}(\tau_{ij}) z^i\bar z^j$). 
On the other hand, we find that

\begin{equation}
    \begin{split}
    \D \alpha=&\D \sigma -\I(t^{-1}\langle \widetilde{Z},\zeta \rangle +t \langle \overline{\widetilde{Z}},\zeta \rangle)\D R + \I R(t^{-2}\langle \widetilde{Z},\zeta \rangle - \langle \overline{\widetilde{Z}},\zeta \rangle)\D t\\
    &-\I R(t^{-1}\langle \D \widetilde{Z},\zeta \rangle +t \langle \D \overline{\widetilde{Z}},\zeta \rangle)-\I R(t^{-1}\langle \widetilde{Z},\D \zeta \rangle +t \langle \overline{\widetilde{Z}},\D  \zeta \rangle)-8\I c_{\ell}\frac{\D t}{t},
    \end{split}
\end{equation}
so we conclude that

\begin{equation}
    \begin{split}
         \D \alpha +\widetilde{\xi}_i\D \xi^i-\xi^i\D \widetilde{\xi}_i&=8\I \rho\frac{\D t}{t}+\D \sigma -\langle \zeta, \D \zeta \rangle-2\I R(t^{-1}\langle \widetilde{Z},\D \zeta \rangle + t \langle \overline{\widetilde{Z}},\D \zeta \rangle) -4(\rho+c_{\ell})\D ^c\mathcal{K}\\
    \end{split}
\end{equation}
so that 
\begin{equation}
    \begin{split}
         -2\pi \mathrm{i}\left(\D \alpha +\widetilde{\xi}_i\D \xi^i-\xi^i\D \widetilde{\xi}_i\right)&=f\frac{\D t}{t} -4\pi R(t^{-1}\langle \widetilde{Z},\D \zeta \rangle + t \langle \overline{\widetilde{Z}},\D \zeta \rangle)-2\I \left(\pi\D\sigma -\pi\langle\zeta, \D\zeta \rangle-4\pi(\rho+c_{\ell})\D^c\mathcal{K}\right)\\
         &=f\frac{\D t}{t}+t^{-1}\theta_+^P|_{\overline{N}} -2\mathrm{i}\theta_3^P|_{\overline{N}} +t\theta_-^P|_{\overline{N}}
    \end{split}
\end{equation}
and the result follows.
\end{proof}
\subsubsection{The case with instanton corrections}\label{typeiiacoordssec}
We now consider the case where the BPS indices are not all equal to $0$. We want to write down the modifications of the coordinates \eqref{Darbouxcoords1loop}, such that

\begin{equation}
    \lambda=-2\pi \mathrm{i}\left(\D \alpha + \widetilde{\xi}_i\D \xi^i -\xi^i\D \widetilde{\xi}_i\right)\cdot s=\left(f\frac{\D t}{t}+t^{-1}\theta_+^P|_{\overline{N}} -2\mathrm{i}\theta_3^P|_{\overline{N}} +t\theta_-^P|_{\overline{N}}\right)\cdot s
\end{equation}
where $f$ and $\theta_{\alpha}^{P}$ are as in \eqref{fthetaexpres}.\\

The expressions for the coordinates below have been previously found in the physics literature (using slightly different arguments and conventions), for example in \cite{APSV,WCHKQK}.

\begin{theorem}\label{theorem3}
Consider the instanton corrected QK manifold $(\widetilde{N},g_{\overline{N}})$ associated to a CASK domain and $(M,\Gamma,Z,\Omega)$, as before. Then the functions on the twistor space $\mathcal{Z}$ given by

\begin{equation}\label{Darbouxcoordsinst}
    \begin{split}
        \xi^{i}&=\zeta^{i}- \mathrm{i}R(t^{-1}z^i +t\overline{z}^i)\\
        \widetilde{\xi}_i&=\widetilde{\zeta}_{i}- \mathrm{i}R(t^{-1}F_{i}+t\overline{F}_i)-\frac{1}{8\pi^2 }\sum_{\gamma}\Omega(\gamma)n_i(\gamma) \int_{l_{\gamma}}\frac{\D \zeta}{\zeta}\frac{t+\zeta}{t-\zeta}\log(1-\exp(2\pi\mathrm{i}\xi_{\gamma}(\zeta)))\\
        \alpha&= \sigma -\mathrm{i}R(t^{-1}\langle \widetilde{Z},\zeta \rangle +t \langle \overline{\widetilde{Z}}, \zeta \rangle)-8\mathrm{i} c_{\ell}\log(t)+\frac{1}{4\pi^2\mathrm{i}}\left(t^{-1}\mathcal{W} +t\overline{\mathcal{W}}-\frac{1}{2\pi}\sum_{\gamma}\Omega(\gamma)\int_{l_{\gamma}}\frac{\D \zeta}{\zeta}\frac{t+\zeta}{t-\zeta}\mathrm{L}(\exp{2\pi\mathrm{i}\xi_{\gamma}(\zeta)})\right)\,,
    \end{split}
\end{equation}
where the integration contours are given by $l_{\gamma}:=\mathbb{R}_{<0}\widetilde{Z}_{\gamma}$ oriented from $0$ to $\infty$, $\xi_\gamma(\zeta) = q_i \xi^i(\zeta)$ for $\gamma=q_i\gamma^i$ and where we abbreviate $\xi^i(\zeta)=\xi^i(R,z^i,\zeta^i,t=\zeta)$,

\begin{equation}\label{wdef}
    \mathcal{W}:=R\sum_{\gamma}\Omega(\gamma)\widetilde{Z}_{\gamma}\sum_{n>0}\frac{e^{-2\pi\mathrm{i}n\zeta_{\gamma}}}{n}K_0(4\pi Rn|\widetilde{Z}_{\gamma}|)\,,
\end{equation}
and
\begin{equation}\label{Rogersdilogarithm}
    \mathrm{L}(\exp({2\pi\mathrm{i} \xi_{\gamma}}))=\mathrm{Li}_2(\exp(2\pi\mathrm{i}\xi_{\gamma}))+\pi\mathrm{i}\xi_{\gamma}\log(1-\exp(2\pi\mathrm{i}\xi_{\gamma}));
\end{equation}
define Darboux coordinates for the contact structure $\lambda$ in the sense that

\begin{equation}
    \lambda=-2\pi \I(\D \alpha +\widetilde{\xi}_i\D \xi^i-\xi^i\D \widetilde{\xi}_i)\cdot s\,.
\end{equation}
\end{theorem}

\begin{proof}
A explicit proof is given in Appendix \ref{appendixtypeiiadc}. This result is not needed for the following sections and is only included for completeness. See also the work in the physics literature \cite{WCHKQK}.
\end{proof}

\begin{remark}
\begin{itemize}
\item The function $\mathrm{Li}_2(x)$ appearing in \eqref{Rogersdilogarithm} is the dilogarithm function, while the function $\mathrm{L}(x)=\mathrm{Li}_2(x)+\frac{\log(x)}{2}\log(1-x)$ is the Rogers dilogarithm. 
\item We remark that along the chosen $l_{\gamma}$ contours, the integrands are exponentially decreasing near $0$ and $\infty$. In particular, one can deform the contour $l_{\gamma}$ within the half plane centered at $l_{\gamma}$ without changing the value of the integral, as long as it does not collide with another ray of the form $l_{\gamma}$ for $\gamma \in \text{Supp}(\Omega)$.

\end{itemize}
\end{remark}

\section{S-duality on instanton corrected q-map spaces}\label{Sdualitysec}

The structure of the section is as follows:

\begin{itemize}
    \item In Section \ref{settingsec} we define instanton corrected q-map spaces $(\overline{N},g_{\overline{N}})$ in terms of the construction of \cite{CT} reviewed in Section \ref{QKCASKrecap},
    applied to certain special pairs of compatible initial data $(M,\mathfrak{F})$ and $(M,\Gamma,Z,\Omega)$. These spaces will be the main objects of study in the rest of the section.
    \item In Section \ref{QMsec} we define the quantum corrected mirror map, previously defined in \cite{QMS}, and relating the type IIA variables $(\rho,z^a,\zeta^i,\widetilde{\zeta}_i,\sigma)$ with the type IIB variables $(\tau_1+\I\tau_2,b^a+\I t^a, c^a,c_a,c_0,\psi)$. In terms of the type IIB variables, we define the corresponding $\mathrm{SL}(2,\mathbb{Z})$-action.
    \item In Section \ref{IIBDCsec} we show that the holomorphic contact structure of the twistor space associated to $(\overline{N},g_{\overline{N}})$ admits Darboux coordinates of the form studied in \cite{QMS}. We show this directly by Poisson resumming the contact form \eqref{contactlocal}, which in turn is expressed in terms of the Bessel function expressions \eqref{fthetaexpres} specialized to the particular form of the BPS indices below \eqref{varBPS}. 
    \item In Section \ref{Sdsec} we give conditions that guarantee that either the S-duality $\mathrm{SL}(2,\mathbb{Z})$-transformations or the subgroup $\langle S \rangle \subset \mathrm{SL}(2,\mathbb{Z})$ acts by isometries on an instanton corrected q-map space. 
     
\end{itemize}

\subsection{Setting}\label{settingsec}
Let us first recall the notion of a PSK manifold in the image of the r-map, and the associated CASK domain. 

\begin{definition}
A projective special real (PSR) manifold is a Riemannian manifold $(\mathcal{H},g_{\mathcal{H}})$ where $\mathcal{H}\subset \mathbb{R}^n$ is a hypersurface, for which there is a cubic polynomial $h:\mathbb{R}^n\to \mathbb{R}$ such that $\mathcal{H}\subset \{h(t)=1\}$ and $g_{\mathcal{H}}=-\partial^2h|_{T\mathcal{H}\times T\mathcal{H}}$.
\end{definition}
If we denote the canonical coordinates of $\mathbb{R}^n$ by $t^a$, then we can write
\begin{equation}\label{cubicpsr}
    h(t)=\frac{1}{6}k_{abc}t^at^bt^c,
\end{equation}
with $k_{abc}\in \mathbb{R}$ symmetric in the indices. Now let $U:=\mathbb{R}_{>0}\cdot \mathcal{H}\subset \mathbb{R}^{n}-\{0\}$ and $\overline{M}^{\text{cl}}:=\mathbb{R}^n+\I U\subset \mathbb{C}^n$ with canonical holomorphic coordinates $z^a:=b^a+\I t^a$. On $\overline{M}^{\text{cl}}$ we have a PSK metric by defining

\begin{equation}\label{8h:eq}
    g_{\overline{M}^{\text{cl}}}:=\frac{\partial^2\mathcal{K}}{\partial z^a \partial \overline{z}^b}\mathrm{d}z^a\mathrm{d}\overline{z}^b\, \quad \; \mathcal{K}=-\log(8h(t))\,.
\end{equation}
The PSK manifold $(\overline{M}^{\text{cl}},g_{\overline{M}^{\text{cl}}})$ is associated to the CASK domain $(M^{\text{cl}},\mathfrak{F}^{\text{cl}})$ of signature $(n,1)$ given by 

\begin{equation}
    M^{\text{cl}}:=\{(Z^0,...,Z^n)=Z^0\cdot(1,z) \in \mathbb{C}^{n+1}\; | \; Z^0\in \mathbb{C}^{\times}, z\in \overline{M}^{\text{cl}}\},\, \quad \mathfrak{F}^{\text{cl}}:=-\frac{1}{6}k_{abc}\frac{Z^aZ^bZ^c}{Z^0}\,.
\end{equation}
Notice in particular the relation $Z^a/Z^0=z^a=b^a+\I t^a$, which we use repeatedly below.

\begin{definition}\label{rmapqmapdef}
The construction given above that associates the PSK manifold $(\overline{M}^{\text{cl}}, g_{\overline{M}^{\text{cl}}})$ to the PSR manifold $(\mathcal{H},g_{\mathcal{H}})$ is called the r-map. Furthermore, the QK metric obtained by applying the c-map (resp. tree-level c-map) to a PSK manifold in the image of the r-map is called a q-map space (resp. tree-level q-map space).
\end{definition}

We now want to consider a tuple $(M,\mathfrak{F})$ of the following special type:
\begin{itemize}
    \item We start with a PSR manifold $(\mathcal{H},g_{\mathcal{H}})$ and consider the associated CASK domain $(M^{\text{cl}},\mathfrak{F}^{\text{cl}})$. We further let $\Lambda^{+}:=\text{span}_{\mathbb{Z}_{\geq 0}}\{\gamma^a\}_{a=1}^{n}-\{0\}$ be a commutative semi-group generated by $\{\gamma^a\}_{a=1}^{n}$, where $n=\text{dim}(\mathcal{H})+1$. We want to consider a new CASK domain with a holomorphic prepotential of the form
    \begin{equation}\label{prepotential}
        \mathfrak{F}(Z^i)=\mathfrak{F}^{\text{cl}}(Z^i)+\mathfrak{F}^{\text{w.s.}}(Z^i)
    \end{equation}
    where $\mathfrak{F}^{\text{cl}}$ is as before, and
    \begin{equation}\label{Fws:eq}
        \mathfrak{F}^{\text{w.s.}}:=\chi\frac{(Z^0)^2\zeta(3)}{2(2\pi \I )^3}-\frac{(Z^0)^2}{(2\pi \I)^3}\sum_{\hat{\gamma}=q_a\gamma^a \in \Lambda^+}n_{\hat{\gamma}}\mathrm{Li}_3(e^{2\pi \I q_aZ^a/Z^0})\,\,.
    \end{equation}
    In the above expression  $\chi \in \mathbb{Z}$, $n_{\hat{\gamma}}\in \mathbb{Z}$, $\mathrm{Li}_n(z)$ denotes the n-th polylogarithm, and $\zeta(x)$ denotes the Riemann zeta function.
    \item We assume that the $n_{\hat{\gamma}}$ are such that the $\mathbb{C}^{\times}$-invariant subset of $\mathbb{C}^{n+1}$ defined by
    \begin{equation}\label{defM}
        M^q:=M^{\text{cl}}\cap\{(Z^0,...,Z^n)\in \mathbb{C}^{n+1} \; | \;  q_at^a>0 \; \text{for all $\hat{\gamma}=q_a\gamma^a\in \Lambda^+$ with $n_{\hat{\gamma}}\neq 0$}\}\;
    \end{equation}
    is open\footnote{This condition is automatic if the commutative semigroup generated by $\{\hat{\gamma}\in \Lambda^+ \; | \; n_{\hat{\gamma}}\neq 0\}$ is finitely generated.}. We further assume that the growth of the $n_{\hat{\gamma}}$ with $\hat{\gamma}$ is such that for any $\epsilon>0$ the series
    \begin{equation}\label{convpropn}
        \sum_{\hat{\gamma}=q_a\gamma^a\in \Lambda^+}|n_{\hat{\gamma}}|e^{-\epsilon q_at^a}
    \end{equation} has compact normal convergence on $M^q$, so that $\mathfrak{F}$ defines a holomorphic function on $M^q$.
     In particular, the condition $q_at^a>0$ ensures that $\mathrm{Li}_3(e^{2\pi \I q_aZ^a/Z^0})$ is well defined and can be expressed as
    
    \begin{equation}
        \mathrm{Li}_3(e^{2\pi \I q_aZ^a/Z^0})=\sum_{k>0}\frac{e^{2\pi \I kq_aZ^a/Z^0}}{k^3}\,.
    \end{equation}
    
    \item We denote by $M\subset M^q$ the maximal open subset of $M^q$ where $\text{Im}(\tau_{ij})=\text{Im}(\partial_i\partial_j\mathfrak{F})$ has signature $(n,1)$ and $\text{Im}(\tau_{ij})Z^i\overline{Z}^j<0$.  These conditions are $\mathbb{C}^{\times}$-invariant, so that $(M,\mathfrak{F})$ defines a CASK domain on each connected component of $M$. 
    \item To such a tuple $(M,\mathfrak{F})$ we associate the canonical lattice $\Gamma \to M$, together with the canonical central charge $Z$ (recall Section \ref{QKinstdomainsec}). Namely, if $Z_i:=\partial_{Z^i}\mathfrak{F}$ and $x^i=\text{Re}(Z^i)$, $y_i=\text{Re}(Z_i)$, then  $(\partial_{x^i},\partial_{y_i})=(\widetilde{\gamma}_i,\gamma^i)$ defines a global Darboux frame for $\Gamma\to M$. For the canonical central charge we then have 
    \begin{equation}\label{Z:eq}
        Z_{\gamma^i}=Z^i, \quad Z_{\widetilde{\gamma}_i}=-Z_i=-\frac{\partial \mathfrak{F}}{\partial Z^i}\,.
    \end{equation}
    We will identify the semigroups $\Lambda^+\cong \text{span}_{\mathbb{Z}_{\geq 0}}\{\partial_{y_a}\}_{a=1}^n-\{0\}$.
\end{itemize}
\begin{remark}\label{Mrem}
\begin{itemize}
\item We remark that given any $(Z^0,...,Z^n)=Z^0\cdot(1,b^a+\I t^a)\in M^q$, all points of the form $Z^0\cdot(1,b^a+\lambda\I t^a)$ with $\lambda>0$ must also be in $M^q$. In particular, for $\lambda>0$ sufficiently big we have $\text{Im}(\tau_{ij})\sim \text{Im}(\partial_i\partial_j\mathfrak{F}^{\text{cl}})$ due to the exponential decay of the terms with polylogarithms. Since $\text{Im}(\partial_i\partial_j\mathfrak{F}^{\text{cl}})$ has signature $(n,1)$ and $\text{Im}(\partial_i\partial_j\mathfrak{F}^{\text{cl}})Z^i\overline{Z}^j<0$, it follows that at the points $Z^0\cdot(1,b^a+\lambda\I t^a)$ for $\lambda>0$ sufficiently big we have that $\text{Im}(\tau_{ij})$ has signature $(n,1)$ and $\text{Im}(\tau_{ij})Z^i\overline{Z}^j<0$, so that the required $M$ is never empty, provided $M^q$ is not empty. 

\item We would also like to comment on the particular form of the prepotential \eqref{prepotential}. In the setting of Calabi-Yau compactifications of type IIB string theory, the term $\mathfrak{F}^{\text{cl}}$ has $k_{abc}\in \mathbb{Z}$ equal to the triple interesection numbers of the Calabi-Yau threefold $X$, and corresponds to the holomorphic prepotential determining (via the c-map) the classical QK geometry of the hypermultiplet moduli space. On the other hand, $\mathfrak{F}^{\text{w.s.}}$
correspond to world-sheet instanton corrections. In such a setting, $\chi$ coincides with the Euler characteristic $\chi(X)$ of $X$, and the numbers $n_{\hat{\gamma}}$ are the so-called genus zero Gopakumar-Vafa invariants, and usually denoted by $n_{\hat{\gamma}}^{(0)}$. Since we are in a setting independent of string theory, we will use the simpler notation of $n_{\hat{\gamma}}$ for the coefficients appearing in \eqref{prepotential}. Furthermore, it will be useful to define $n_{0}:=-\frac{\chi}{2}$. Then, using that $\mathrm{Li}_3(1)=\zeta(3)$ we can write
\begin{equation}\label{Fws}
    \mathfrak{F}^{\text{w.s.}}=-\frac{(Z^0)^2}{(2\pi \I)^3}\sum_{\hat{\gamma} \in \Lambda^+\cup\{0\}}n_{\hat{\gamma}}\mathrm{Li}_3(e^{2\pi \I q_aZ^a/Z^0})\,.
\end{equation}
\item In the physics literature sometimes an  additional term of the form

\begin{equation}
    \frac{1}{2}A_{ij}Z^iZ^j, \quad A_{ij}=A_{ji}\in \mathbb{R}
\end{equation}
is added to $\mathfrak{F}$ in \eqref{prepotential}. While the addition of such a term does not alter the PSK metric $g_{\overline{M}}$ or the CASK metric $g_{M}$ (since $A_{ij}$ is real), the inclusion of such a term turns out to be important for mirror symmetry between D-branes of type IIA and type IIB string theory (see for example the review \cite[Chapter V]{HMreview1}, and the references therein). However, below we will focus on a case analogous to including D(-1) and D1 instanton corrections on type IIB, and we can safely ignore the inclusion of such a term. 
\end{itemize}
\end{remark}

Finally, after possibly restricting $M$, we assume that $M$ is the maximal open subset such that $(M,\Gamma,Z,\Omega)$ with $\Gamma$ and $Z$ as before and with

    \begin{equation}\label{varBPS}
        \begin{cases}
     \Omega(q_0\gamma^0)=-\chi=2n_0, \quad q_0\in \mathbb{Z}-\{0\}\\
     \Omega(q_0\gamma^0\pm q_a\gamma^a)=\Omega(\pm q_a\gamma^a)=n_{q_a\gamma^a} \quad \text{for $q_a\gamma^a \in \Lambda^+$, $q_0\in \mathbb{Z}$}\\
    \Omega(\gamma)=0 \quad \text{else},\\
        \end{cases}
    \end{equation}
is a mutually local variation of BPS structures compatible (in the sense of Definition \ref{compdef}) with the CASK manifold $(M,g_M,\omega_M,\nabla,\xi)$ associated to $(M,\mathfrak{F})$. To check that $(M,\Gamma,Z,\Omega)$ defines a mutually local variation of BPS structures, one only needs to check the support property, since the convergence property follows easily from \eqref{convpropn}, while the mutual locality and invariance under monodromy is obvious.\\

We remark that the BPS indices in \eqref{varBPS} are determined by the holomorphic prepotential $\mathfrak{F}$, see (\ref{Fws:eq}). Such a prescription of BPS indices has previously appeared in the physics literature in \cite{QMS,Sduality} or more explicitly in \cite[Equation 4.5]{HMreview2}. See also \cite{Joyce:2008pc} (in particular, \cite[Section 6.3, Conjecture 6.20]{Joyce:2008pc}), where such a BPS spectrum is conjectured for a non-compact Calabi-Yau 3-folds without compact divisors).\\

In the following, we will consider the associated instanton corrected QK manifold $(\overline{N},g_{\overline{N}})$ associated to $(M,\mathfrak{F})$ and $(M,\Gamma,Z,\Omega)$ as in Theorem \ref{QKCASKdomainformulas}. It will be convenient to lift the QK metric  $(\overline{N},g_{\overline{N}})$ to $(\widetilde{N},g_{\overline{N}})$ as in Definition \ref{liftdef}, where the $(\zeta^i,\widetilde{\zeta}_i,\sigma)$ directions are no longer periodic. In particular, we recall that we have the following:

\begin{itemize}
    \item $\widetilde{N}\subset \mathbb{R}_{>0}\times \overline{M}\times \mathbb{R}^{2n+2}\times \mathbb{R}$ is an open subset, where the splitting is such that $(\rho,z^a,\zeta^i,\widetilde{\zeta}_i,\sigma)\in \mathbb{R}_{>0}\times \mathbb{C}^n\times \mathbb{R}^{2n+2}\times \mathbb{R}$, where $\overline{M}$ is the associated PSK manifold.
    \item The twistor space $\mathcal{Z}$ of $(\widetilde{N},g_{\overline{N}})$ smoothly decomposes as $\mathcal{Z}=\widetilde{N}\times \mathbb{C}P^1$. In particular, the expression \eqref{contactlocal} for the contact structure holds for the global coordinates $(\rho,z^a,\zeta^i,\widetilde{\zeta}_i,\sigma,t)$ of $\mathcal{Z}$. 
    Here we are slightly abusing notation by considering $t$, the identity map on $\mathbb{C}P^1$, as a global coordinate.)
\end{itemize}

\begin{definition}
If $(M,\mathfrak{F})$ and $(M,\Gamma,Z,\Omega)$ are given as before, we will call the resulting QK manifold $(\overline{N},g_{\overline{N}})$ (or its lift $(\widetilde{N},g_{\overline{N}})$) an instanton corrected q-map space.
\end{definition}

The reason for the calling $(\overline{N},g_{\overline{N}})$ (or $(\widetilde{N},g_{\overline{N}})$) an instanton corrected q-map space is the following:

\begin{itemize}
    \item By setting the one-loop parameter to $c_{\ell}=0$, and $n_{\hat{\gamma}}=0$ for all $\hat{\gamma}\in \Lambda^+\cup\{0\}$ (and hence, also $\Omega(\gamma)=0$ for all $\gamma$), one recovers a q-map space. That is the QK metric obtained by applying the c-map to a PSK manifold in the image of the r-map:
    \begin{equation}\label{QKtree}
    g_{\overline{N}}=g_{\overline{M}}+\frac{\D\rho^2}{4\rho^2}-\frac{1}{4\rho}(N^{ij}-2e^{\mathcal{K}}z^i\overline{z}^j)W_i\overline{W}_j + \frac{1}{64\rho^2}(\D\sigma + \widetilde{\zeta}_{i}\D\zeta^{i}-\zeta^{i}\D\widetilde{\zeta}_{i})^2\,.
\end{equation}
where $g_{\overline{M}}$, $N_{ij}$, $\mathcal{K}$ and $W_i$ are constructed in terms of $\mathfrak{F}^{\text{cl}}$ (see Section \ref{QKinstdomainsec} for the definition of the above terms).
\item In the setting of Calabi-Yau compactification of type IIB string theory, the terms due to the BPS structure $(M,\Gamma,Z,\Omega)$ are thought as D(-1), D1 instanton corrections, and those due to $\mathfrak{F}^{\text{w.s.}}$ are thought as world-sheet instantons corrections. Hence, the QK metric obtained by the above $(M,\mathfrak{F})$ and $(M,\Gamma,Z,\Omega)$ can be thought as a q-map with the inclusion of the analog of the above corrections. 
\end{itemize} 
\subsection{The quantum corrected mirror map and the S-duality action}\label{QMsec}

In the sections that follow we will rescale the twistor fiber coordinate $t\to -\I t$. The contact structure on the twistor space is then expressed by (compare with \eqref{contactlocal})
    
    \begin{equation}\label{contstr}
    \lambda=\left(f\frac{\D t}{t}+t^{-1}\I \theta_+^P|_{\overline{N}} -2\I \theta_3^P|_{\overline{N}} -t\I \theta_-^P|_{\overline{N}}\right)\cdot s\,.
\end{equation}

In order to define the S-duality action, we consider the following diffeomorphism, first defined (under slightly different conventions) in \cite{QMS}:

\begin{definition} Let $\overline{M}:=\{z \in \mathbb{C}^n \; | \; (1,z)\in M\}$ where $M$ was given in the previous section, and
consider the manifold $\overline{\mathcal{N}}_{\text{IIA}}:=\mathbb{R}_{>-c_{\ell}}\times \overline{M}\times \mathbb{R}^{2n+2}\times \mathbb{R}$ with global coordinates $(\rho,z^a,\zeta^i,\widetilde{\zeta}_i,\sigma)$. We will call such coordinates type IIA coordinates. On the other hand, if $H\subset \mathbb{C}$ denotes the upper half-plane, we define $\overline{\mathcal{N}}_{\text{IIB}}:=H\times \overline{M}\times \mathbb{R}^{2n}\times \mathbb{R}^2$ with global coordinates $(\tau_1+\mathrm{i}\tau_2,b^a+\I t^a,c^a,c_a,c_0,\psi)$. We call the latter type IIB coordinates. The type IIB coordinates are related to the type IIA coordinates via the diffeomorphism (see Remark \ref{diffeoremark} below) $\mathcal{M}: \overline{\mathcal{N}}_{\text{IIB}}\to \overline{\mathcal{N}}_{\text{IIA}}$ defined by
    \begin{equation}\label{MM}
    \begin{split}
     z^a&=b^a+\I t^a,  \;\;\;\;\; \rho=\frac{\tau_2^2}{16}e^{-\mathcal{K}}-c_{\ell}, \;\;\;\;\;\;\; \zeta^0=\tau_1, \;\;\;\;\; \zeta^a=-(c^a-\tau_1b^a),\\
     \widetilde{\zeta}_a&=c_a +\frac{k_{abc}}{2}b^b(c^c-\tau_1b^c)+\widetilde{\zeta}_{a}^{\text{inst}}, \;\;\;\;\;\;  \widetilde{\zeta}_0=c_0-\frac{k_{abc}}{6}b^ab^b(c^c-\tau_1b^c)+\widetilde{\zeta}_{0}^{\text{inst}}\\
    \sigma&= -2(\psi + \frac{1}{2}\tau_1c_0) + c_a(c^a-\tau_1b^a) -\frac{k_{abc}}{6}b^ac^b(c^c-\tau_1b^c)+\sigma^{\text{inst}}\,,
    \end{split}
    \end{equation}
    where  $\mathcal{K}=-\log(-2\text{Im}(\tau_{ij})z^i\overline{z}^j)$ with $z^0=1$,  $\tau_{ij}=\frac{\partial^2 \mathfrak{F}}{\partial Z^i\partial Z^j}$, $c_{\ell}\in \mathbb{R}$ is the 1-loop parameter, and 
    \begin{equation}\label{typeiiacoordcor}
        \begin{split}
        \widetilde{\zeta}^{\text{inst}}_{a}&:=\frac{1}{8\pi^2}\sum_{\hat{\gamma}\in \Lambda^+}n_{\hat{\gamma}}q_a\sum_{\substack{m\in \mathbb{Z}-\{0\}\\ n\in \mathbb{Z}}}\frac{m\tau_1+ n}{m|m\tau+ n|^2}e^{-S_{\hat{\gamma},m,n}}\\
        \widetilde{\zeta}^{\text{inst}}_0&:=\frac{\I}{16\pi^3}\sum_{\hat{\gamma}\in \Lambda^+}n_{\hat{\gamma}}\sum_{\substack{m\in \mathbb{Z}-\{0\}\\ n\in \mathbb{Z}}}\left(\frac{(m\tau_1+ n)^2}{|m\tau + n|^3}+ 2\pi q_a\left(t^a+\I b^a\frac{m\tau_1+n}{|m\tau+ n|}\right)\right)\frac{e^{-S_{\hat{\gamma},m,n}}}{m|m\tau+ n|}\\
        \sigma^{\text{inst}}&:=\tau_1\widetilde{\zeta}_0^{\text{inst}}-(c^a-\tau_1b^a)\widetilde{\zeta}_a^{\text{inst}}-\frac{\I\tau_2^2}{8\pi^2}\sum_{\hat{\gamma}\in \Lambda_{+}}n_{\hat{\gamma}}q_at^a\sum_{n\in \mathbb{Z}-\{0\}}\frac{e^{-S_{\hat{\gamma},0,n}}}{n|n|}\\
    &\quad\quad +\frac{\I}{8\pi^3}\sum_{\hat{\gamma}\in \Lambda^{+}}n_{\hat{\gamma}}\sum_{\substack{m\in \mathbb{Z}-\{0\}\\ n\in \mathbb{Z}}}\left(2-\frac{(m\tau_1+n)^2}{|m\tau+ n|^2}\right)\frac{(m\tau_1+ n)e^{-S_{\hat{\gamma},m,n}}}{m^2|m\tau+ n|^2}
        \end{split}
    \end{equation}
    where
    \begin{equation}\label{mnaction}
       S_{\hat{\gamma},m,n}:=2\pi q_a(|m\tau+ n|t^a +\I mc^a+ \I nb^a)\,.
    \end{equation}
We will refer to the diffeomorphism $\mathcal{M}$ as the quantum corrected mirror map.
\end{definition}
\begin{remark}\label{diffeoremark}
\begin{itemize}
\item Note that since $z=(z^a)\in \overline{M}$ and $M\subset M^{q}$, we have $\text{Re}(S_{\hat{\gamma},m,n})>0$ for all $\hat{\gamma}\in \Lambda^{+}$ with $n_{\hat{\gamma}}\neq 0$, so \eqref{convpropn} implies that the sums in \eqref{typeiiacoordcor} have compact normal convergence on $\overline{\mathcal{N}}_{\text{IIB}}$.  Furthermore, we remark that  \eqref{MM} really defines a diffeomorphism. Indeed, since  $-\text{Im}(\tau_{ij})z^i\overline{z}^j>0$ on $\overline{M}$ and $\tau_2>0$, we can always invert the relation involving $\rho$, while $\widetilde{\zeta}_i^{\text{inst}}$ and $\sigma^{\text{inst}}$ only depend on $\tau=\tau_1+\mathrm{i}\tau_2$, $z^a$ and $c^a$, so it is easy to invert all the other relations. Note, in  particular, that the expressions for $\widetilde{\zeta}_i^{\text{inst}}$ and $\sigma^{\text{inst}}$ are real. Furthermore, if we set $c_{\ell}=0$, $n_{\hat{\gamma}}=\chi=0$ for all $\hat{\gamma}\in \Lambda^{+}$, then we recover the classical mirror map.
\item While the expressions \eqref{typeiiacoordcor} look complicated, they satisfy nice transformation properties with respect to the expected isometry groups of the metric. See Section \ref{universalisosec}, Lemma \ref{lemmainstcoord} and Corollary \ref{coordcorollary}.
\end{itemize}
\end{remark}

\begin{definition}\label{defsl2domain} Let $\overline{\mathcal{N}}_{\text{IIB}}^{\text{cl}}:=H\times \overline{M}^{\text{cl}}\times \mathbb{R}^{2n}\times \mathbb{R}^2$ with global type IIB variables $(\tau=\tau_1+\I \tau_2,b^a+\I t^a,c^a,c_a,c_0,\psi)$. We define an $\mathrm{SL}(2,\mathbb{Z})$-action on $\overline{\mathcal{N}}_{\text{IIB}}^{\text{cl}}$ by 

\begin{equation}\label{sl2can}
    \begin{split}
    \tau&\to \frac{a\tau +b}{c\tau +d},\;\;\;\;\; t^a \to |c\tau +d|t^a, \;\;\;\; c_a \to c_a, \\
    \begin{pmatrix}c^a\\ b^a\end{pmatrix}&\to \begin{pmatrix} a & b \\ c & d\\ \end{pmatrix}\begin{pmatrix}c^a \\ b^a\end{pmatrix}, \;\;\;\; \begin{pmatrix}c_0\\ \psi\end{pmatrix}\to \begin{pmatrix} d & -c \\ -b & a\\ \end{pmatrix}\begin{pmatrix}c_0 \\ \psi \end{pmatrix}, \;\;\;\; \begin{pmatrix} a & b \\ c & d\\ \end{pmatrix}\in \mathrm{SL}(2,\mathbb{Z})
    \end{split}
\end{equation}

We call this action the S-duality action. We note that there is the possibility that S-duality does not act on $\overline{\mathcal{N}}_{\text{IIB}}\subset \overline{\mathcal{N}}_{\text{IIB}}^{\text{cl}}$, since it may happen that $\overline{M}$ is not invariant under the scaling of the $t^a$ in \eqref{sl2can}. 
\end{definition}

\begin{remark} Recall from Section \ref{HK/QKsec} that the quaternionic K\"ahler manifold $\overline{N}$ was constructed in \cite{CT} via HK/QK correspondence (by specializing \cite{QKPSK}) as the hypersurface 
of the bundle $P\to N$ which is 
defined by the equation $\mathrm{Arg}\, Z^0 =0$. 
Using the IIB variables we can now simply match $|Z^0|=\tau_2$. 
\end{remark}

\subsection{Type IIB Darboux coordinates}\label{IIBDCsec}
We now want to write a distinguished set of Darboux coordinates which will be useful for studying S-duality or the action by specific elements of $\mathrm{SL}(2,\mathbb{Z})$. The coordinates below have been previously written in \cite{QMS}. Nevertheless, our approach will be different from \cite{QMS} in the sense that we will start from the mathematical construction of the instanton QK metric obtained in \cite{CT}, and then explicitly show that \eqref{typeiibdc} indeed define Darboux coordinates for the contact structure on its twistor space. \\

Recalling \eqref{MM} and \eqref{typeiiacoordcor}, we let 
    \begin{equation}
        \widetilde{\zeta}_i^{\text{cl}}:=\widetilde{\zeta}_i-\widetilde{\zeta}_i^{\text{inst}}, \quad \sigma^{\text{cl}}:=\sigma - \sigma^{\text{inst}}\,,
    \end{equation}
and define 
    \begin{equation}\label{dcclassical}
    \begin{split}
        \xi^{i,\text{cl}}&:=\zeta^{i}+ R(t^{-1}z^i -t\overline{z}^i)\\
        \widetilde{\xi}_i^{\text{cl}}&:=\widetilde{\zeta}_{i}^{\text{cl}}+ R(t^{-1}F^{\text{cl}}_{i}-t\overline{F}^{\text{cl}}_i)\\
        \alpha^{\text{cl}}&:= \sigma^{\text{cl}} +R(t^{-1}\langle \widetilde{Z}^{\text{cl}},\zeta^{\text{cl}} \rangle -t \langle \overline{\widetilde{Z}}^{\text{cl}}, \zeta^{\text{cl}} \rangle)\,,
    \end{split}
\end{equation}
where $\widetilde{Z}^{\text{cl}}:= Z^{\text{cl}}/Z^0$ and $Z^{\text{cl}}$ is defined replacing 
$Z$ and $\mathfrak F$ in equation (\ref{Z:eq}) by $Z^{\text{cl}}$ and $\mathfrak F^{\text{cl}}$, and $F_i^{\text{cl}}:=\partial_{Z^i}\mathfrak{F}^{\text{cl}}/Z^0$.
The expressions \eqref{dcclassical} match the coordinates \eqref{Darbouxcoords1loop} for the case $\mathfrak{F}=\mathfrak{F}^{\text{cl}}$ after the scaling $t \to -\mathrm{i}t$ done in \eqref{contstr} and setting $c_{\ell}=0$. In particular, they define Darboux coordinates for the contact structure of twistor space of the tree-level q-map space defined by $\mathfrak{F}^{\text{cl}}$.
Finally, if $c\in \mathbb{R}-\{0\}$ and $d\in \mathbb{R}$, we denote by $t_{\pm}^{c,d}$ the roots of $t(c\xi^{0,\text{cl}}+d)=0$ in the variable $t$. Using that  $\zeta^0=\tau_1$ and $R=2\sqrt{\rho+c_{\ell}}e^{\mathcal{K}/2}=\tau_2/2$ (recall \eqref{MM}), we find

\begin{equation}\label{troots}
    t_{\pm}^{c,d}=\frac{c\tau_1+d \pm |c\tau +d|}{c\tau_2}\,.
\end{equation}
\begin{theorem}\label{theorem1}
Consider an instanton corrected q-map space  $(\widetilde{N},g_{\overline{N}})$ with $1$-loop parameter $c_{\ell}=\frac{\chi}{192\pi}$. Then the functions $(\xi^i,\widetilde{\xi}_{i},\alpha)$ of the associated twistor space $\mathcal{Z}\cong \widetilde{N}\times \mathbb{C}P^1$ given by 

    \begin{align*}
        \xi^i&=\xi^{i,\text{cl}}\\
        \widetilde{\xi}_{a}&=\widetilde{\xi}_a^{\text{cl}}+\frac{\tau_2}{8\pi^2}\sum_{\hat{\gamma}\in \Lambda^+\cup\{0\}}n_{\hat{\gamma}}q_a\left(\sum_{(m,n)\in \mathbb{Z}^2-\{0\}}\frac{e^{-S_{\hat{\gamma},m,n}}}{|m\tau+ n|^2}\frac{1+t_+^{m,n}t}{t-t_{+}^{m, n}}\right)\\
        \widetilde{\xi}_0&=\widetilde{\xi}_0^{\text{cl}}+\frac{\I \tau_2}{16\pi^3}\sum_{\hat{\gamma}\in \Lambda^+\cup\{0\}}n_{\hat{\gamma}}\sum_{(m,n)\in \mathbb{Z}^2-\{0\}}\left(\frac{1}{m\xi^0+ n}+\frac{m\tau_1+ n}{|m\tau+ n|^2}\right)\frac{1+t_{+}^{m, n}t}{t-t_{+}^{m, n}}\frac{e^{-S_{\hat{\gamma},m,n}}}{|m\tau + n|^2}\\
        &\quad-\frac{\tau_2}{8\pi^2}\sum_{\hat{\gamma}\in \Lambda^{+}\cup \{0\}}n_{\hat{\gamma}} \sum_{(m,n)\in \mathbb{Z}^2-\{0\}}\left(q_ab^a \frac{1+t_{+}^{m, n}t}{t-t_{+}^{m, n}}+\I q_at^a\frac{1-t_{+}^{m, n}t}{t-t_{+}^{m, n}}\right)\frac{e^{-S_{\hat{\gamma},m,n}}}{|m\tau+ n|^2}\\
        \alpha&=-\frac{1}{2}(\alpha^{\text{cl}}-\xi^{i}\widetilde{\xi}_i^{\text{cl}})+\frac{\I \tau_2^2}{32\pi^3}\sum_{\hat{\gamma}\in \Lambda^{+}\cup\{0\}}n_{\hat{\gamma}}\sum_{(m,n)\in \mathbb{Z}^2-\{0\}}\left((m\tau_1+ n)(t^{-1}-t) -2m\tau_2\right)\frac{1+t_{+}^{m, n}t}{t-t_{+}^{m, n}}\frac{e^{-S_{\hat{\gamma},m,n}}}{|m\tau+n|^4}\,, \numberthis \label{typeiibdc}
    \end{align*}
define Darboux coordinates for the contact structure of $\mathcal{Z}$, in the sense that

\begin{equation}\label{contdarbiib}
    \lambda=4\pi \I\left(\D\alpha -\widetilde{\xi}_i\D \xi^i\right)\cdot s\,.
\end{equation}
In the above expressions, for the sums in the case where $m=0$ we set 

\begin{equation}\label{m=0def}
    \frac{1+t_{+}^{0, n}t}{t-t_{+}^{0, n}}:=\begin{cases}
        1/t, \quad n<0\\
        -t, \quad n>0
    \end{cases}, \quad \quad \quad \quad \frac{1-t_{+}^{0, n}t}{t-t_{+}^{0, n}}:=\begin{cases}
        1/t, \quad n<0\\
        t, \quad n>0\,.
    \end{cases}
\end{equation}
\end{theorem}
\begin{proof}
This will be proven in Section \ref{proofsec} below.
\end{proof}
\begin{remark}
\begin{itemize}
\item The definitions \eqref{m=0def} above for the case $m=0$ (where $t_{\pm}^{m,n}$ is not defined) can be motivated as follows. If $n>0$ then it is easy to check that
\begin{equation}
    \lim_{c\to 0}\frac{1+t_{+}^{c, n}t}{t-t_{+}^{c, n}}=-t\,, \quad \lim_{c\to 0}\frac{1-t_{+}^{c, n}t}{t-t_{+}^{c, n}}=t.
\end{equation}
On the other hand, by using that for $c\neq 0$ we have $t_{+}^{c,d}t_{-}^{c,d}=-1$, one similarly finds for $n<0$
\begin{equation}
    \lim_{c\to 0}\frac{t_-^{c,n}}{t_-^{c,n}}\left(\frac{1+t_{+}^{c, n}t}{t-t_{+}^{c, n}}\right)= \lim_{c\to 0}\frac{t_-^{c,n}-t}{t_-^{c,n}t+1}=\frac{1}{t}\,, \quad\quad \lim_{c\to 0}\frac{t_-^{c,n}}{t_-^{c,n}}\left(\frac{1-t_{+}^{c, n}t}{t-t_{+}^{c, n}}\right)=\lim_{c\to 0}\frac{t_-^{c,n}+t}{t_-^{c,n}t+1}=\frac{1}{t}.
\end{equation}
\item As previously mentioned in the introduction, one can relate the Darboux coordinates \eqref{typeiibdc} to the Darboux coordinates \eqref{Darbouxcoordsinst} by Poisson resumming the later and applying a contact transformation, as done in \cite{QMS}. We will follow a similar approach by Poisson resumming the contact structure directly, and then checking that \eqref{typeiibdc} give Darboux coordinates. 
\end{itemize}
\end{remark}

Using \eqref{convpropn} it is not hard to see that the coordinates are well defined on an open dense subset of $\mathcal{Z}\cong \widetilde{N}\times \mathbb{C}P^1$, namely the points 
\begin{equation}
    \{(\tau,b^a+\I t^a, c^a,c_a,c_0,\psi, t)\in \mathcal{Z} \; | \; t\neq 0, \infty, \;\; t\not\in \{t_{+}^{m,n}(\tau)\}_{m\in \mathbb{Z}-\{0\}, n \in \mathbb{Z}}\}\,,
\end{equation}
where we have used the quantum corrected mirror map \eqref{MM} to put type IIB coordinates on $\widetilde{N}\subset \overline{\mathcal{N}}_{\text{IIA}}$.
In subsequent sections, we will also frequently use the notation 
\begin{equation}
    \widetilde{\xi}_i^{\text{inst}}:=\widetilde{\xi}_i-\widetilde{\xi}_i^{\text{cl}}, \quad \alpha^{\text{inst}}:=\alpha + \frac{1}{2}(\alpha^{\text{cl}}-\xi^{i}\widetilde{\xi}_i^{\text{cl}})\,.
\end{equation}
\begin{remark}
When comparing \eqref{typeiiacoordcor} and \eqref{typeiibdc} with the formulas appearing in \cite{QMS}, we remark that we are using different conventions for the definition of $t_{\pm}^{c,d}$ and $S_{\hat{\gamma},m,n}$. Namely, what we call $S_{\hat{\gamma},m,n}$ in their notation is $S_{\hat{\gamma},-m,-n}$ and $t_{\pm}^{c,d}$ in their notation is $t_{\mp}^{c,d}$. 
\end{remark}

As an immediate corollary of Theorem \ref{theorem1}, we obtain the following

\begin{corollary}\label{theorem1cor}
Consider an instanton corrected q-map space  $(\widetilde{N},g_{\overline{N}})$ with one-loop parameter $c_{\ell}\in \mathbb{R}$. If $(\xi^i,\widetilde{\xi}_i,\alpha)$ are the functions on $\mathcal{Z}$ defined by \eqref{typeiibdc}, then if $\alpha_{c_{\ell}}:=\alpha+4\I \left(c_{\ell}-\frac{\chi}{192\pi}\right)\log(t)$, the functions $(\xi^i,\widetilde{\xi}_{i},\alpha_{c_{\ell}})$ of the associated twistor space are Darboux coordinates for the contact structure $\lambda$ of $\mathcal{Z}$, in the sense that

\begin{equation}
    \lambda=4\pi \I\left(\D\alpha_{c_{\ell}} -\widetilde{\xi}_i\D \xi^i\right)\cdot s\,.
\end{equation}
\end{corollary}

\begin{proof}
Notice that since the one loop parameter only enters in the contact structure via $f$ by \eqref{fthetaexpres} and $\rho=2\pi r^2 -c_{\ell}$, by writing $f=(f+16\pi (c_{\ell}-\frac{\chi}{192\pi}))-16\pi (c_{\ell}-\frac{\chi}{192\pi}) $ we have by Theorem \ref{theorem1} that 
\begin{equation}
    4\pi \I\left(\D\alpha -\widetilde{\xi}_i\D \xi^i\right)=\left(f+16\pi \left(c_{\ell}-\frac{\chi}{192\pi}\right)\right)\frac{\D t}{t}+t^{-1}\I \theta_+^P|_{\overline{N}} -2\I \theta_3^P|_{\overline{N}} -t\I \theta_-^P|_{\overline{N}}\,.
\end{equation}
It then follows immediately that 

\begin{equation}
    4\pi \I\left(\D\alpha_{c_{\ell}} -\widetilde{\xi}_i\D \xi^i\right)=f\frac{\D t}{t}+t^{-1}\I \theta_+^P|_{\overline{N}} -2\I \theta_3^P|_{\overline{N}} -t\I \theta_-^P|_{\overline{N}}\,,
\end{equation}
compare (\ref{contstr}).
\end{proof}

\subsubsection{Preliminary lemmas}\label{prellemsec}
In this section we give some preliminary lemmas and expressions that will be useful to prove Theorem \ref{theorem1}.\\

We will divide each of the $1$-forms $\theta_{+}^{P}|_{\overline{N}}$, $\theta_3^P|_{\overline{N}}$ appearing in the contact structure \eqref{contstr} as follows:

\begin{equation}
    \theta_{+}^P|_{\overline{N}}=\theta_{+}^{P, \text{cl}}+\theta_{+}^{P, \text{w.s.}}+\theta_{+}^{P, \text{inst}}, \quad\quad \theta_{3}^P|_{\overline{N}}=\theta_{3}^{P, \text{cl}}+\theta_{3}^{P, \text{w.s.}}+\theta_{3}^{P, \text{inst}},
\end{equation}
where using that $R=2\sqrt{\rho+c_{\ell}}e^{\mathcal{K}/2}=\tau_2/2$, the decompositions $F_i=F_i^{\text{cl}}+F_i^{\text{ws}}$ where $F_i^{\text{cl}}:=\partial_{Z^i}\mathfrak{F}^{\text{cl}}/Z^0$ and $F_i^{\text{w.s.}}:=\partial_{Z^i}\mathfrak{F}^{\text{w.s.}}/Z^0$, $\widetilde{\zeta}_i=\widetilde{\zeta}_i^{\text{cl}}+\widetilde{\zeta}_i^{\text{inst}}$ and $\sigma=\sigma^{\text{cl}}+\sigma^{\text{inst}}$,  and \eqref{fthetaexpres} we have
\begin{equation}\label{theta+split}
    \begin{split}
       \theta_{+}^{P,\text{cl}}&:=-2\pi\tau_2\left(F_i^{\text{cl}}\D  \zeta^i-z^i\D  \widetilde{\zeta}_i^{\text{cl}}\right)\\
       \theta_{+}^{P,\text{w.s.}}&:=-2\pi\tau_2\left(F_i^{\text{w.s.}}\D  \zeta^i-z^i\D  \widetilde{\zeta}_i^{\text{inst}}\right)\\
       \theta_{+}^{P,\text{inst}}&=\I \tau_2\sum_{\gamma}\Omega(\gamma)\widetilde{Z}_{\gamma}\sum_{n>0}e^{-2\pi\I n\zeta_{\gamma}}K_0(4\pi Rn|\widetilde{Z}_{\gamma}|)\D  \zeta_{\gamma}\\
    &\quad+\frac{\tau_2^2}{2}\sum_{\gamma}\Omega(\gamma)\widetilde{Z}_{\gamma}\sum_{n>0}e^{-2\pi \I n\zeta_{\gamma}}|\widetilde{Z}_{\gamma}|K_1(4\pi Rn|\widetilde{Z}_{\gamma}|)\left(\frac{\D  \widetilde{Z}_{\gamma}}{\widetilde{Z}_{\gamma}}+\frac{\D  \overline{\widetilde{Z}}_{\gamma}}{\overline{\widetilde{Z}}_{\gamma}}+\frac{2}{\tau_2}\D  \tau_2\right).\\
    \end{split}
\end{equation}
For $\theta_3^P|_{\overline{N}}$ we use the relation 
\begin{equation}
    \D^c\mathcal{K}=e^{\mathcal{K}}(-2\text{Im}(\tau_{ij}))\Big(\I \overline{z}^j\D z^i -\I z^i\D \overline{z}^j\Big),
\end{equation}
so that 

\begin{equation}
    4\pi(\rho+c_{\ell})\D ^c\mathcal{K}=\pi R^2(-2\text{Im}(\tau_{ij}))\Big(\I\overline{z}^j\D z^i -\I z^i\D \overline{z}^j\Big)=\pi\tau_2^2\mathrm{Im}(\tau_{0a})\D t^a +\pi\tau_2^2\mathrm{Im}(\tau_{ab})(b^a\D t^b-t^a\D b^b),
\end{equation}
and hence we can split $\theta_3^P|_{\overline{N}}$ as follows using \eqref{fthetaexpres}
\begin{equation}\label{theta3split}
    \begin{split}
       \theta_{3}^{P,\text{cl}}&:=\pi\D  \sigma^{\text{cl}} +\pi\left(\widetilde{\zeta}_i^{\text{cl}}\D  \zeta^i-\zeta^i\D  \widetilde{\zeta}_i^{\text{cl}}\right)-\pi\tau_2^2\mathrm{Im}(\tau_{0a}^{\text{cl}})\D  t^a -\pi\tau_2^2\mathrm{Im}(\tau_{ab}^{\text{cl}})(b^a\D  t^b-t^a\D  b^b)\\
       \theta_{3}^{P,\text{w.s.}}&:=\pi\D  \sigma^{\text{inst}} +\pi\left(\widetilde{\zeta}_i^{\text{inst}}\D  \zeta^i-\zeta^i\D  \widetilde{\zeta}_i^{\text{inst}}\right)-\pi\tau_2^2\mathrm{Im}(\tau_{0a}^{\text{w.s.}})\D  t^a -\pi\tau_2^2\mathrm{Im}(\tau_{ab}^{\text{w.s.}})(b^a\D  t^b-t^a\D  b^b)\\
       \theta_3^{P,\text{inst}}&=-\frac{\I\tau_2}{4\pi}\sum_{\gamma}\Omega(\gamma)\sum_{n>0}\frac{e^{-2\pi\I n\zeta_{\gamma}}}{n}|\widetilde{Z}_{\gamma}|K_1(4\pi Rn|\widetilde{Z}_{\gamma}|)\left(\frac{\D  \widetilde{Z}_{\gamma}}{\widetilde{Z}_{\gamma}}-\frac{\D  \overline{\widetilde{Z}}_{\gamma}}{\overline{\widetilde{Z}}_{\gamma}}\right)\,,
    \end{split}
\end{equation}
where $\tau_{ij}^{\text{cl}}=\partial_{Z^i}\partial_{Z^j}\mathfrak{F}^{\text{cl}}$ and $\tau_{ij}^{\text{w.s.}}=\partial_{Z^i}\partial_{Z^j}\mathfrak{F}^{\text{w.s.}}$.\\

Similarly, using that $f=16\pi \rho+f^{\text{inst}}$ and (using the second relation in \eqref{MM})
\begin{equation}\label{rhodef}
 \rho=\frac{\tau_2^2}{16}e^{-\mathcal{K}}-c_{\ell}=  \frac{\tau_2^2}{16}K-c_{\ell}, \quad K=-2\text{Im}(\tau_{ij})z^i\overline{z}^j\,,
\end{equation} 
we can decompose $f$ as

\begin{equation}\label{fsplitting2}
    f=f^{\text{cl}}+f^{\text{w.s.}}-16\pi c_{\ell}+f^{\text{inst}},
\end{equation}
where

\begin{equation}\label{fsplitting}
    \begin{split}
        f^{\text{cl}}&:=-2\pi\tau_2^2\text{Im}(\tau_{ij}^{\text{cl}})z^i\overline{z}^j\\
        f^{\text{w.s.}}&:=-2\pi\tau_2^2\text{Im}(\tau_{ij}^{\text{w.s.}})z^i\overline{z}^j=\frac{\tau_2^2}{2\pi^2}\sum_{\hat{\gamma}\in \Lambda^+\cup\{0\}}n_{\hat{\gamma}}\text{Re}\left(\mathrm{Li}_3(e^{2\pi \I  q_az^a})+2\pi q_at^a\mathrm{Li}_2(e^{2\pi \I  q_az^a})\right)\\
        f^{\text{inst}}&=\frac{\tau_2}{\pi}\sum_{\gamma}\Omega(\gamma)\sum_{n>0}\frac{e^{-2\pi\I  n\zeta_{\gamma}}}{n}|\widetilde{Z}_{\gamma}|K_1(4\pi Rn|\widetilde{Z}_{\gamma}|)\,.\\
    \end{split}
\end{equation}

In order to prove Theorem \ref{theorem1}, we start with the following lemma, which will allow us to rewrite the Bessel function expressions of the instanton terms $f^{\text{inst}}$, $\theta_{\pm}^{P,\text{inst}}$ and $\theta_3^{P,\text{inst}}$ in terms of expressions similar to those appearing in \eqref{typeiibdc}. 

\begin{definition}
Given $\hat{\gamma}\in \Lambda^{+}$ with $n_{\hat{\gamma}}\neq 0$, and $\nu \in \mathbb{N}$, we define 
\begin{equation}\label{Idef}
\begin{split}
    \mathcal{I}_{\hat{\gamma}}^{(\nu)}&:=2\sum_{q_0\in \mathbb{Z}}\sum_{s=\pm 1}\sum_{n=1}^{\infty}\frac{e^{-2\pi\I ns\zeta_{q_0\gamma^0+\hat{\gamma}}}}{(sn)^{\nu}}K_0(4\pi Rn|\widetilde{Z}_{q_0\gamma^0+\hat{\gamma}}|)\\
    \mathcal{I}_0^{(\nu)}&:=2\sum_{q_0\in \mathbb{Z}-\{0\}}\sum_{s=\pm 1}\sum_{n=1}^{\infty}\frac{e^{-2\pi\I ns\zeta_{q_0\gamma^0}}}{(sn)^\nu}K_0(4\pi Rn|\widetilde{Z}_{q_0\gamma^0}|)\,.
\end{split}
\end{equation}
\end{definition}
Notice that in $\mathcal{I}_{0}^{(\nu)}$ we exclude the term $q_0=0$ from the sum, since $K_0(x)$ diverges at $x=0$. Furthermore, since $\hat{\gamma}\in \Lambda^{+}$ with $n_{\hat{\gamma}}\neq 0$ implies that $\pm (q_0\gamma^0+\hat{\gamma})\in \text{Supp}(\Omega)$ for all $q_0\in \mathbb{Z}$, by the support property of variations of BPS structures one finds that $|Z_{q_0\gamma^0+\hat{\gamma}}|\neq 0$, so that $|\widetilde{Z}_{q_0\gamma^0+\hat{\gamma}}|\neq 0$ on $\widetilde{N}$ (or $\overline{N})$. By the exponential decay of $K_0(x)$ as $x\to \infty$, it is easy to check that the convergence of the series $\mathcal{I}_{\hat{\gamma}}^{(\nu)}$ and $\mathcal{I}_{0}^{(\nu)}$ is compact normal on $\widetilde{N}$. In particular, they define smooth functions on $\widetilde{N}$ and we can interchange sums with differentiation. 

\begin{lemma}\label{PoissonLemma} The function $\mathcal{I}_{\hat{\gamma}}^{(\nu)}$ can be expressed as

\begin{equation}\label{Presum1}
    \mathcal{I}_{\hat{\gamma}}^{(\nu)}=\sum_{n\in \mathbb{Z}}\sum_{m\in \mathbb{Z}-\{0\}}\frac{e^{-S_{\hat{\gamma},m,n}}}{m^{\nu}|m\tau + n|},
\end{equation}
where $S_{\hat{\gamma},m,n}$ was defined in \eqref{mnaction}. Furthermore, we have for $\nu \geq 2$ 
\begin{equation}\label{I0der}
    \partial_{\tau_2}\mathcal{I}_{0}^{(\nu)}=\partial_{\tau_2}\mathcal{I}_{\hat{\gamma}}^{(\nu)}|_{\hat{\gamma}=0}+\frac{2}{\tau_2}\sum_{s=\pm 1}\sum_{n>0}\frac {1}{(sn)^{\nu}}, \quad \partial_{\tau_1}\partial_{\tau_2}\mathcal{I}_{0}^{(\nu)}=\partial_{\tau_1}\partial_{\tau_2}\mathcal{I}_{\hat{\gamma}}^{(\nu)}|_{\hat{\gamma}=0}, \quad \partial_{\tau_1}^2\mathcal{I}_{0}^{(\nu)}=\partial_{\tau_1}^2\mathcal{I}_{\hat{\gamma}}^{(\nu)}|_{\hat{\gamma}=0}\,,
\end{equation}
where  $\partial_{\tau_2}\mathcal{I}_{\hat{\gamma}}^{(\nu)}|_{\hat{\gamma}=0}$ (and similarly for the other derivatives) means taking the derivative of \eqref{Presum1} and then evaluating at $\hat{\gamma}=0$.
\end{lemma}

\begin{proof}
We will use a similar notation to \cite[Appendix B]{QMS} and follow the same idea of Poisson resummation, but in an easier case than the one presented in \cite[Appendix B]{QMS}. We denote $x=q_ab^a$, $y=q_at^a$ and $\Theta:=q_a(\zeta^a-b^a\zeta^0)=-q_ac^a$, so that we can write

\begin{equation}\label{Poisson1}
    \mathcal{I}_{\hat{\gamma}}^{(\nu)}=\sum_{q_0\in \mathbb{Z}}f(x+q_0,y,\zeta^i,R)
\end{equation}
where
\begin{equation}
    f(x,y,\zeta^i,R):=2\sum_{s=\pm 1}\sum_{n=1}^{\infty}\frac{e^{2\pi\I ns(\Theta+x\zeta^0)}}{(sn)^{\nu}}K_0(4\pi Rn|x+\I y|)\,.
\end{equation}
We have used above that $\zeta_{q_i\gamma^i}=-q_i\zeta^i$ (recall Definition \ref{QKCASKcoords}).\\

Equation \eqref{Poisson1} makes it clear that $\mathcal{I}_{\hat{\gamma}}^{(\nu)}$ is a function that is invariant under integer shifts in the $x$ variable, and the idea is to now compute the Poisson resummation with respect to the $x$-variable. Namely, 

\begin{equation}\label{Prformula}
    \mathcal{I}_{\hat{\gamma}}^{(\nu)}=\sum_{q_0\in \mathbb{Z}}f(x+q_0,y,\zeta^i,R)=\sum_{q_0\in \mathbb{Z}}\hat{f}(q_0,y,\zeta^i,R)e^{2\pi \I q_0x},
\end{equation}
where
\begin{equation}
    \hat{f}(w)=\int_{-\infty}^{\infty}\D x\; f(x)e^{-2\pi \I xw}
\end{equation}
denotes the Fourier transform (from now on, we omit the dependence of $f$ of the variables $y$, $\zeta^i$ and $R$ from the notation).\\

Using the integral representation \eqref{Besselintrep} of $K_0$ we have

\begin{equation}
\begin{split}
    K_0(4\pi Rn|x+\I y|)&=\int_{0}^{\infty}\D t\exp(-4\pi Rn|x+\I y|\cosh(t))=\frac{1}{2}\int_{-\infty}^{\infty}\D t\exp(-4\pi Rn|x+\I y|\cosh(t))\\
    &=\frac{1}{2}\int_{0}^{\infty}\frac{\D z}{z}\exp(-2\pi Rn|x+\I y|(z^{-1}+z))\,.
    \end{split}
\end{equation}
Furthermore, since the integral exponentially
decays at both ends, provided $x\neq 0$, we can deform 
 the integration path by changing $z$ to $z\to z\frac{|x+\I y|}{x+\I y}\text{sign}(x)$, obtaining
\begin{equation}
    K_0(4\pi Rn|x+\I y|)=\frac{1}{2}\int_{0}^{\infty}\frac{\D z}{z}\exp(-2\pi Rn\cdot\text{sign}(x)((x+\I y)z^{-1}+(x-\I y)z))\,.
\end{equation}

In particular, we have 

\begin{align*}
        \hat{f}(w)&=\int_{-\infty}^{\infty}\D x\sum_{s=\pm 1}\sum_{n=1}^{\infty}\frac{e^{2\pi \I ns\Theta}}{(sn)^{\nu}}\int_{0}^{\infty} \frac{\D z}{z}\exp(- 2\pi nR\cdot\text{sign}(x)((x+\I y)z^{-1}+(x-\I y)z))e^{2\pi\I x(- w+ns\zeta^0)}\\
        &=\sum_{s=\pm 1}\sum_{n=1}^{\infty}\frac{e^{ 2\pi\I ns\Theta}}{(sn)^{\nu}}\int_{0}^{\infty} \frac{\D z}{z}\int_{-\infty}^{\infty}\D x\exp(- 2\pi nR\cdot\text{sign}(x)((x+\I y)z^{-1}+(x-\I y)z))e^{2\pi\I x(- w+ns\zeta^0)}\\
        &=\sum_{s=\pm 1}\sum_{n=1}^{\infty}\frac{e^{ 2\pi\I ns\Theta}}{(sn)^{\nu}}\int_{0}^{\infty} \frac{\D z}{z}\int_{0}^{\infty}\D x\;e^{2\pi x(-nRz^{-1}-nRz -\I w+\I ns\zeta^0)+2\pi y(-\I nRz^{-1}+\I nRz)}\\
        &\quad +\sum_{s=\pm 1}\sum_{n=1}^{\infty}\frac{e^{2\pi \I ns\Theta}}{(sn)^{\nu}}\int_{0}^{\infty} \frac{\D z}{z}\int_{0}^{\infty}\D x\;e^{2\pi x(-nRz^{-1}-nRz + \I w -\I ns\zeta^0) + 2\pi y(\I nRz^{-1}-\I nRz)}\\
        &=\frac{1}{2\pi}\sum_{s=\pm 1}\sum_{n=1}^{\infty}\frac{e^{2\pi \I ns\Theta}}{(sn)^{\nu}}\left(\int_{0}^{\infty} \frac{\D z}{z}\frac{e^{2\pi y(-\I nRz^{-1}+\I nRz)}}{nRz^{-1}+nRz +\I w-\I ns\zeta^0}+\int_{0}^{\infty} \frac{\D z}{z}\frac{e^{2\pi y(\I nRz^{-1}-\I nRz)}}{nRz^{-1}+nRz -\I w+\I ns\zeta^0}\right). \numberthis \\
\end{align*}
We can combine the $s=1$ and $s=-1$ terms of the previous sum into an integral over $\mathbb{R}$ and a sum over $n\in \mathbb{Z}-\{0\}$ to obtain

\begin{equation}
    \hat{f}(w)=\frac{1}{2\pi}\sum_{n\in \mathbb{Z}-\{0\}}\frac{e^{2\pi \I n\Theta}}{n^{\nu-1}|n|}\int_{-\infty}^{\infty}\frac{\D z}{z}\frac{e^{-2\pi\I nRy(z^{-1}-z)}}{nR(z^{-1}+z) + \I w-\I n\zeta^0}\,.
\end{equation}

The integrand of the $n$-th term of $\hat{f}(m)$ with $m\in \mathbb{Z}$ has poles at $z=\I t_{\pm}^{n,- m}$ where $t_{\pm}^{c,d}$ was defined before in \eqref{troots} as the roots in $t$ of $t(c\xi^0(t)+d)=0$. Since one of our defining conditions on the manifold $M$ is that $\text{sign}(y)=\text{sign}(q_at^a)>0$ for $n_{\hat{\gamma}}\neq 0$ (see \eqref{defM}), we can compute the previous integral by closing the contour in the upper half plane when $n>0$, and in the lower half plane if $n<0$. Independently of the sign of $n$, only the pole at $\mathrm{i}t_{+}^{n,-m}$ contributes to the $n$-th integral. In particular, we obtain

    \begin{equation}
        \begin{split}
        \hat{f}(m)&= \sum_{n\in \mathbb{Z}-\{0\}}\frac{e^{2\pi \I n\Theta}}{n^{\nu}}\frac{e^{-2\pi nRy(1/t_{+}^{n,- m}+t_{+}^{n,- m})}}{nR(t_{+}^{n,- m}-t_{-}^{n,- m})}= \sum_{n\in \mathbb{Z}-\{0\}}\frac{e^{2\pi \I n\Theta}}{n^{\nu}}\frac{e^{-2\pi y|n\tau- m|}}{|n\tau- m|}\,,
        \end{split}
    \end{equation}
where we used that $t_{+}^{n,-m}t_{-}^{n,-m}=-1$ and the expressions \eqref{troots} of $t_{\pm}^{n,-m}$.\\

After plugging the previous result in \eqref{Prformula}, using $2R=\tau_2$, (\ref{mnaction}) and relabeling the summation indices we then obtain
\begin{equation}\label{resum1}
\begin{split}
    \mathcal{I}_{\hat{\gamma}}^{(\nu)}&= \sum_{\substack{m\in \mathbb{Z}-\{0\}\\ n\in \mathbb{Z}}}\frac{e^{2\pi \I m\Theta-2\pi \I nx}}{m^{\nu}}\frac{e^{-2\pi y|m\tau+ n|}}{|m\tau+ n|}= \sum_{\substack{m\in \mathbb{Z}-\{0\}\\ n\in \mathbb{Z}}}\frac{e^{-S_{\hat{\gamma},m,n}}}{m^{\nu}|m\tau + n|}\,.\\
\end{split}
\end{equation}
 Finally, we use the following identities to prove \eqref{I0der}:

\begin{equation}\label{Bessellimit}
    \lim_{z\to 0}zK_1(2\pi\tau_2 nz)=\frac{1}{2\pi n\tau_2}, \quad \lim_{z\to 0}z^2K_1(2\pi\tau_2 n z)=0, \quad \lim_{z\to 0} z^2K_0(2\pi \tau_2 nz)=0\,.
\end{equation}
  We show the first identity in \eqref{I0der}, with the others following similarly. By applying the first limit in \eqref{Bessellimit}, the dominated convergence theorem (here we use that $\nu\geq 2)$ and $K_0'=-K_1$, we obtain: 
\begin{equation}
    \begin{split}
       \partial_{\tau_2}\mathcal{I}_{0}^{(\nu)}&=\lim_{\zeta_{\hat{\gamma}}\to 0,\widetilde{Z}_{\hat{\gamma}}\to 0}\left(\partial_{\tau_2}\mathcal{I}_{\hat{\gamma}}^{(\nu)}+2\sum_{s=\pm 1}\sum_{n>0}\frac{e^{-2\pi \mathrm{i}ns \zeta_{\hat{\gamma}}}}{(sn)^{\nu}}2\pi n|\widetilde{Z}_{\hat{\gamma}}|K_1(2\pi n\tau_2|\widetilde{Z}_{\hat{\gamma}}|)\right)\\
       &=\left(\partial_{\tau_2}\mathcal{I}_{\hat{\gamma}}^{(\nu)}\right)\Big|_{\hat{\gamma}=0}+\frac{2}{\tau_2}\sum_{s=\pm 1}\sum_{n>0}\frac{1}{(sn)^{\nu}}\,. 
    \end{split}
\end{equation}
\end{proof}

\subsubsection{Poisson resummation of the quantum corrections of the contact structure}

In this section, we seek to rewrite the quantum corrections terms in $f$, $\theta_{+}^P$ and $\theta_3^P$ using the Poisson resummation of Lemma \ref{PoissonLemma}. This will help us show that the coordinates \eqref{typeiibdc} are Darboux coordinates for the contact structure describing the instanton corrected q-map metric associated to the data of Section \ref{settingsec}. We start with the following proposition for $f$:

\begin{proposition}\label{typeiibfquantumprop}
The following holds for $f^{\text{w.s.}}$ and $f^{\text{inst}}$ defined in \eqref{fsplitting}: 

\begin{equation}\label{typeiibfquant}
    f^{\text{w.s.}}+f^{\text{inst}}=\frac{\chi}{12}+\frac{\tau_2^2}{(2\pi)^2}\sum_{\hat{\gamma}\in \Lambda^+\cup\{0\}}n_{\hat{\gamma}}\sum_{(m,n)\in \mathbb{Z}^2-\{0\}}\left(\frac{1}{|m\tau+n|}+2\pi q_at^a\right)\frac{e^{-S_{\hat{\gamma},m,n}}}{|m\tau + n|^2}
\end{equation}

\end{proposition}

\begin{proof}
On one hand, we have
\begin{equation}
    \begin{split}
        f^{\text{inst}}&=\frac{2R}{\pi}\sum_{\gamma}\Omega(\gamma)\sum_{n>0}\frac{e^{-2\pi\I n\zeta_{\gamma}}}{n}|\widetilde{Z}_{\gamma}|K_1(4\pi Rn|\widetilde{Z}_{\gamma}|)\\
        &=\frac{2R}{\pi}\sum_{\hat{\gamma}\in \Lambda^+}n_{\hat{\gamma}}\sum_{s=\pm 1}\sum_{q_0\in \mathbb{Z}}\sum_{n>0}\frac{e^{-2\pi\I ns\zeta_{q_0\gamma^0+\hat{\gamma}}}}{n}|\widetilde{Z}_{q_0\gamma^0+\hat{\gamma}}|K_1(4\pi Rn|\widetilde{Z}_{q_0\gamma^0+\hat{\gamma}}|)\\
        &\quad \quad-\frac{R}{\pi}\chi\sum_{s=\pm 1}\sum_{q_0\in \mathbb{Z}-\{0\}}\sum_{n>0}\frac{e^{-2\pi\I ns\zeta_{q_0\gamma^0}}}{n}|\widetilde{Z}_{q_0\gamma^0}|K_1(4\pi Rn|\widetilde{Z}_{q_0\gamma^0}|)\\
        &=-\frac{R}{2\pi^2}\sum_{\hat{\gamma}\in \Lambda^+}n_{\hat{\gamma}}\partial_{\tau_2}\mathcal{I}_{\hat{\gamma}}^{(2)}+\frac{R}{(2\pi)^2}\chi \partial_{\tau_2}\mathcal{I}_{0}^{(2)}\,,\\
    \end{split}
\end{equation}
where we have used the particular form of the BPS indices \eqref{varBPS}, the expressions \eqref{Idef} and the fact that $R=\tau_2/2$. Using Lemma \ref{PoissonLemma} we therefore obtain that

\begin{equation}\label{typeiibfinst}
    \begin{split}
        f^{\text{inst}}&=\frac{\tau_2^2}{(2\pi)^2}\sum_{\hat{\gamma}\in \Lambda^+}n_{\hat{\gamma}}\sum_{\substack{m\in \mathbb{Z}-\{0\}\\ n\in \mathbb{Z}}}\left(\frac{1}{|m\tau+ n|}+2\pi q_at^a\right)\frac{e^{-S_{\hat{\gamma},m,n}}}{|m\tau + n|^2}\\
        &\quad\quad -\frac{\tau_2^2}{2(2\pi)^2}\chi \sum_{\substack{m\in \mathbb{Z}-\{0\}\\ n\in \mathbb{Z}}}\frac{1}{|m\tau + n|^3}+\frac{\tau_2\chi}{2(2\pi)^2}\frac{4}{\tau_2}\sum_{n>0}\frac{1}{n^2}\\
        &=\frac{\chi}{12}+\frac{\tau_2^2}{(2\pi)^2}\sum_{\hat{\gamma}\in \Lambda^+\cup\{0\}}n_{\hat{\gamma}}\sum_{\substack{m\in \mathbb{Z}-\{0\}\\ n\in \mathbb{Z}}}\left(\frac{1}{|m\tau+ n|}+2\pi q_at^a\right)\frac{e^{-S_{\hat{\gamma},m,n}}}{|m\tau + n|^2}\,,\\
    \end{split}
\end{equation}
where in the last equality we have used that $\sum_{n>0}\frac{1}{n^2}=\frac{\pi^2}{6}$ and combined the other two sums by using the convention $n_{0}=-\frac{\chi}{2}$ from Section \ref{settingsec}.
On the other hand, it is not hard to check using \eqref{fsplitting} that one can write

\begin{equation}\label{typeiibfws}
    f^{\text{w.s.}}=\frac{\tau_2^2}{(2\pi)^2}\sum_{\hat{\gamma}\in \Lambda^+\cup\{0\}}n_{\hat{\gamma}}\sum_{n\in \mathbb{Z}-\{0\}}\left(\frac{1}{|n|}+2\pi q_at^a\right)\frac{e^{-S_{\hat{\gamma},0,n}}}{ n^2}\\
\end{equation}
By summing \eqref{typeiibfws} and \eqref{typeiibfinst} we therefore obtain \eqref{typeiibfquant}.
\end{proof}
\begin{remark}
    Combining Proposition \ref{typeiibfquantumprop} with \eqref{fsplitting2} we see that 
    \begin{equation}\label{Poissonresumedf}
        f=8\pi \tau_2^2h(t) +16\pi\left(\frac{\chi}{192\pi}-c_{\ell}\right)+ \frac{\tau_2^2}{(2\pi)^2}\sum_{\hat{\gamma}\in \Lambda^+\cup\{0\}}n_{\hat{\gamma}}\sum_{(m,n)\in \mathbb{Z}^2-\{0\}}\left(\frac{1}{|m\tau+n|}+2\pi q_at^a\right)\frac{e^{-S_{\hat{\gamma},m,n}}}{|m\tau + n|^2}\,,
    \end{equation}
    where $f^{\text{cl}}=-2\pi\tau_2^2\text{Im}(\tau_{ij}^{\text{cl}})z^i\overline{z}^j=8\pi\tau_2^2h(t)$ and $h(t)$ is given in \eqref{cubicpsr}. In particular, we see that the quantum corrected $f$ has the same transformation rule as $f^{\text{cl}}$ under $S$-duality if and only if $c_{\ell}=\frac{\chi}{192\pi}$. Namely, for this value of $c_{\ell}$, $f$ transforms under $\begin{pmatrix}
    a & b\\
    c & d\\ \end{pmatrix}\in \mathrm{SL}(2,\mathbb{Z})
    $ as
    \begin{equation}
    f\to \frac{f}{|c\tau+d|}\,.
    \end{equation}
    This suggests that this particular value of the 1-loop constant $c_{\ell}$ could play a special role in studying $S$-duality, and indeed, we will see in Section \ref{Sdsec} that for this value S-duality acts by isometries on the instanton corrected q-map space (provided $S$-duality acts on the domain of definition of the metric). Furthermore, in the setting of Calabi-Yau compactification of type IIB string theory, the quantity $f$ with $c_{\ell}=\frac{\chi}{192}$ is proportional to the 4d dilaton with perturbative, world-sheet instanton, D(-1) and D1 instanton corrections, and the corresponding transformation property has been previously remarked in the physics literature (see for example \cite[Equation 4.1]{QMS} and the paragraphs below the equation).  
\end{remark}

We now continue with the Poisson resummation of the other terms of the contact form. 
\begin{proposition}\label{typeiibtheta3quantumprop}
The 1-form $-2\I (\theta_3^{P, \text{w.s.}}+\theta_3^{P, \text{inst}})$ can be rewritten in terms of the coordinates $(\tau_2,b^a,t^a,\tau_1=\zeta^0,\zeta^a, \widetilde{\zeta}_i,\sigma)$ as follows

\begin{equation}\label{typeiibtheta3quantum}
\begin{split}
    &-2\I (\theta_3^{P, \text{w.s.}}+\theta_3^{P, \text{inst}})\\
    &=\sum_{\hat{\gamma}\in \Lambda^{+}\cup\{0\}}n_{\hat{\gamma}}\sum_{(m,n)\in \mathbb{Z}^{2}-\{0\}}e^{-S_{\hat{\gamma},m,n}}\Bigg[\frac{\I \tau_2^4}{2\pi}\frac{m^2}{|m\tau +n|^4}q_a\mathrm{d}b^a -\frac{\tau_2^2}{2\pi}\frac{m\tau_1+ n}{|m\tau+ n|^3}q_a\mathrm{d}t^a+\frac{\I \tau_2^2}{2\pi }\frac{m(m\tau_1+ n)}{|m\tau+ n|^4}q_a\mathrm{d}\zeta^a
    \\
    &\quad-\left(\frac{\tau_2^3}{\pi^2} \frac{m^2(m\tau_1+ n)}{|m\tau+ n|^6}+\frac{\tau_2}{2\pi}q_at^a\frac{(m\tau_1+ n)^3+2m^2\tau_2^2(m\tau_1+ n)}{|m\tau+ n|^5}\right)\mathrm{d}\tau_2\\
    &\quad+\left(-\frac{\I \tau_2^2}{2\pi}q_ab^a \frac{m(m\tau_1+ n)}{|m\tau+ n|^4}+\frac{\tau_2^4}{2\pi}q_at^a \frac{m^3}{|m\tau+ n|^5}+\frac{\tau_2^2}{2\pi^2}\frac{m(m^2\tau_2^2-(m\tau_1+ n)^2)}{|m\tau+ n|^6}\right)\mathrm{d}\tau_1\Bigg].\\
\end{split}
\end{equation}
\end{proposition}

\begin{proof}
We follow the same idea as in the previous proposition. Namely, notice that using  \eqref{theta3split} and \eqref{Idef}, $\theta_{3}^{P,\text{inst}}$ can be written as
\begin{equation}
    \begin{split}
     \theta_3^{P,\text{inst}}=&-\frac{\I R}{2\pi}\sum_{\gamma}\Omega(\gamma)\sum_{n>0}\frac{e^{-2\pi \I n\zeta_{\gamma}}}{n}|\widetilde{Z}_{\gamma}|K_1(4\pi Rn|\widetilde{Z}_{\gamma}|)\left(\frac{\D \widetilde{Z}_{\gamma}}{\widetilde{Z}_{\gamma}}-\frac{\D \overline{\widetilde{Z}}_{\gamma}}{\overline{\widetilde{Z}}_{\gamma}}\right)\\
     &=-\frac{\I R}{2\pi}\sum_{\hat{\gamma}\in \Lambda^+}n_{\hat{\gamma}}\sum_{s=\pm 1}\sum_{q_0\in \mathbb{Z}}\sum_{n>0}\frac{e^{-2\pi \I ns\zeta_{q_0\gamma^0+\hat{\gamma}}}}{n}|\widetilde{Z}_{q_0\gamma^0+\hat{\gamma}}|K_1(4\pi Rn|\widetilde{Z}_{q_0\gamma^0+\hat{\gamma}}|)\left(\frac{\D \widetilde{Z}_{\hat{\gamma}}}{\widetilde{Z}_{q_0\gamma^0+\hat{\gamma}}}-\frac{\D \overline{\widetilde{Z}_{\hat{\gamma}}}}{\overline{\widetilde{Z}}_{q_0\gamma^0+\hat{\gamma}}}\right)\\
     &=\frac{\I }{8\pi^2}\sum_{\hat{\gamma}\in \Lambda^+}n_{\hat{\gamma}}\partial_{z^a}\mathcal{I}_{\hat{\gamma}}^{(2)}\D z^a-\frac{\I }{8\pi^2}\sum_{\hat{\gamma}\in \Lambda^+}n_{\hat{\gamma}}\partial_{\overline{z}^a}\mathcal{I}_{\hat{\gamma}}^{(2)}\D \overline{z}^a,
    \end{split}
\end{equation}
where in second equality we have used that no $\hat{\gamma}=0$ terms appear due to $\mathrm{d}\widetilde{Z}_{q_0\gamma^0}=\mathrm{d}q_0=0$.
Using Lemma \ref{PoissonLemma} we then obtain the following:

\begin{equation}\label{typeiibtheta3inst}
    \begin{split}
        -2\I &\theta_3^{P,\text{inst}}\\  &=\frac{\I }{2\pi}\sum_{\hat{\gamma}\in \Lambda^+}n_{\hat{\gamma}}\sum_{\substack{m\in \mathbb{Z}-\{0\}\\ n\in \mathbb{Z}}}\frac{e^{-S_{\hat{\gamma},m,n}}}{m^2}q_a\D b^a+\frac{1}{2\pi}\sum_{\hat{\gamma}\in \Lambda^+}n_{\hat{\gamma}}\sum_{\substack{m\in \mathbb{Z}-\{0\}\\ n\in \mathbb{Z}}}e^{-S_{\hat{\gamma},m,n}}\frac{m\tau_1+ n}{m^2|m\tau+ n|}q_a\D t^a\,.\\
    \end{split}
\end{equation}
Observe for the calculation that the partial derivatives $\partial_{z^a}$ and $\partial_{\overline{z}^a}$ are to be taken with respect to the mixed coordinates $(\tau_2,b^a,t^a,\tau_1=\zeta^0,\zeta^a, \widetilde{\zeta}_i,\sigma)$, which are neither the IIA nor the IIB coordinates.
On the other hand, using directly the formulas \eqref{typeiiacoordcor} and \eqref{Fws} for $\sigma^{\text{inst}}$, $\widetilde{\zeta}_i^{\text{inst}}$ and $\mathfrak{F}^{\text{w.s.}}$ one can compute $-2\I \theta_3^{P,\text{w.s.}}$ in \eqref{theta3split}. The $\D \zeta^a$, $\D \tau_1$ and $\D \tau_2$ components of $-2\I \theta_3^{P,\text{w.s.}}$ can be seen to match the corresponding components in the right-hand side of \eqref{typeiibtheta3quantum}. On the other hand, the $\D t^a$ and $\D b^a$ components of $-2\I \theta_3^{P,\text{w.s.}}$ are as follows
\begin{align*}\label{typeiibtheta3ws}
        -2\I \theta_3^{P,\text{w.s.}}|_{\D b^a}&= -\frac{\I }{2\pi}\sum_{\hat{\gamma}\in \Lambda^+\cup\{0\}}n_{\hat{\gamma}}q_a\sum_{\substack{m\in \mathbb{Z}-\{0\}\\ n\in \mathbb{Z}}}\left(2-\frac{(m\tau_1+ n)^2}{|m\tau+n|^2}\right)\frac{e^{-S_{\hat{\gamma},m,n}}(m\tau_1+ n)^2}{m^2|m\tau+ n|^2}\\
        -2\I \theta_3^{P,\text{w.s.}}|_{\D t^a}&= -\frac{\tau_2^2}{2\pi}\sum_{\hat{\gamma}\in \Lambda^+\cup\{0\}}n_{\hat{\gamma}}q_a\sum_{n\in \mathbb{Z}-\{0\}}\frac{e^{-S_{\hat{\gamma},0,n}}}{n|n|}\\
        & \quad-\frac{1}{2\pi}\sum_{\hat{\gamma}\in \Lambda^+\cup\{0\}}n_{\hat{\gamma}}q_a\sum_{\substack{m\in \mathbb{Z}-\{0\}\\ n\in \mathbb{Z}}}\left(2-\frac{(m\tau_1+ n)^2}{|m\tau+ n|^2}\right)\frac{e^{-S_{\hat{\gamma},m,n}}(m\tau_1+ n)}{m^2|m\tau+ n|}\,.\numberthis\\
    \end{align*}
One can then directly check that the $\D b^a$ and $\D t^a$ components of \eqref{typeiibtheta3inst} and \eqref{typeiibtheta3ws} combine into the corresponding components of \eqref{typeiibtheta3quantum}. We therefore obtain the required expression.
\end{proof}

Finally, we compute a similar expression for $\theta_{+}^P$:
\begin{proposition}\label{typeiibthetapquantumprop}
The 1-form $\I (\theta_{+}^{P,\text{w.s.}}+\theta_{+}^{P,\text{inst}})$ can be rewritten in terms of the coordinates $(\tau_2,b^a,t^a,\tau_1=\zeta^0,\zeta^a, \widetilde{\zeta}_i,\sigma)$ as follows

\begin{equation}\label{typeiibthetaplusquant}
\begin{split}
    &\I (\theta_{+}^{P,\text{w.s.}}+\theta_{+}^{P,\text{inst}})\\
    =&\sum_{\hat{\gamma}\in \Lambda^+\cup \{0\}}n_{\hat{\gamma}}\sum_{(m,n)\in \mathbb{Z}^2-\{0\}}e^{-S_{\hat{\gamma},m,n}}\Bigg[\frac{\I \tau_2^3}{4\pi}\frac{m(|m\tau+n|-(m\tau_1+n))}{|m\tau+ n|^4}q_a\mathrm{d}b^a\\
    &-\frac{\tau_2^3}{4\pi } \frac{m}{|m\tau+ n|^3}q_a\mathrm{d}t^a+\frac{\I \tau_2}{4\pi}\frac{(m\tau_1+ n)(|m\tau+ n|-(m\tau_1+ n))}{|m\tau+ n|^4}q_a\mathrm{d}\zeta^a \\
        &+\frac{\tau_2}{4\pi}\frac{(m\tau_1+n)((m\tau_1+n)-|m\tau+n|)}{|m\tau+n|^4}\left(q_at^a\frac{m\tau_1+n}{|m\tau+n|}+\mathrm{i}q_ab^a\right)\mathrm{d}\tau_1\\
        &-\frac{\tau_2}{8\pi^2}\frac{2(m\tau_1+ n)(m^2\tau_2^2-(m\tau_1+ n)^2) +|m\tau+ n|(2(m\tau_1+ n)^2-m^2\tau_2^2)}{|m\tau+ n|^6}\mathrm{d}\tau_1\\
        &+\frac{\tau_2^2}{4\pi}\frac{m(m\tau_1+n)}{|m\tau+n|^5}\left(q_at^a((m\tau_1+n)-|m\tau+n|)+\frac{1}{2\pi}\frac{4(m\tau_1+n)-3|m\tau+n|}{|m\tau+n|}\right)\mathrm{d}\tau_2\Bigg]\;.\\
\end{split}
\end{equation}
\end{proposition}
\begin{proof}
As before, we first notice that we can write the terms of $\theta_{+}^{P,\text{inst}}$ as follows. For the first sum in \eqref{theta+split} we have

\begin{equation}
    \begin{split}
        &2R\I\sum_{\gamma}\Omega(\gamma)\widetilde{Z}_{\gamma}\sum_{n>0}e^{-2\pi\I n\zeta_{\gamma}}K_0(4\pi Rn|\widetilde{Z}_{\gamma}|)\D \zeta_{\gamma}\\
        &=2R\I\sum_{\hat{\gamma}\in \Lambda^+}n_{\hat{\gamma}}\sum_{q_0\in \mathbb{Z}}\sum_{s=\pm 1}\sum_{n>0}\widetilde{Z}_{q_0\gamma^0+\hat{\gamma}}e^{-2\pi \I ns\zeta_{q_0\gamma^0+\hat{\gamma}}}K_0(4\pi Rn|\widetilde{Z}_{q_0\gamma^0+\hat{\gamma}}|)\D \zeta_{q_0\gamma^0+\hat{\gamma}}\\
        &\quad -R\I \chi\sum_{q_0\in \mathbb{Z}}\sum_{s=\pm 1}\sum_{n>0}\widetilde{Z}_{q_0\gamma^0}e^{-2\pi \I ns\zeta_{q_0\gamma^0}}K_0(4\pi Rn|\widetilde{Z}_{q_0\gamma^0}|)\D \zeta_{q_0\gamma^0}\\
        &=-\frac{1}{(2\pi)^3}\sum_{\hat{\gamma}\in \Lambda^+}n_{\hat{\gamma}}\D \zeta^0\partial_{\zeta^0}\partial_{\tau_2}\partial_{\overline{z}_{\hat{\gamma}}}\mathcal{I}_{\hat{\gamma}}^{(3)}-\frac{\I }{(2\pi)^2}\sum_{\hat{\gamma}\in \Lambda^+}n_{\hat{\gamma}}\D \zeta^a\partial_{\tau_2}\partial_{\overline{z}^a}\mathcal{I}_{\hat{\gamma}}^{(2)}-\frac{R\I }{2(2\pi)^2}\chi \D \zeta^0\partial_{\zeta^0}^2\mathcal{I}_{0}^{(2)}\,,\\
    \end{split}
\end{equation}
where $\partial_{z_{\hat{\gamma}}}:=\frac{1}{\#q}\frac{1}{q_a}\partial_{z^a}$ with the sum in the index $a$ only for the $a=1,...,n$ such that $q_a\neq0$, and $\#q$ is the number of non-zero $q_a$\footnote{It is natural to denote $z_{\hat{\gamma}}:=q_az^a$ for $\hat{\gamma}\in \Lambda^{+}$. With this notation, the differential operator $\partial_{\widetilde{z}_{\hat{\gamma}}}$ satisfies $\partial_{z_{\hat{\gamma}}}z_{\hat{\gamma}}=1$.}. By Lemma \ref{PoissonLemma}, we then find

    \begin{align*}
     &2R\I \sum_{\gamma}\Omega(\gamma)\widetilde{Z}_{\gamma}\sum_{n>0}e^{-2\pi\I n\zeta_{\gamma}}K_0(2Rn|\widetilde{Z}_{\gamma}|)\D \zeta_{\gamma}\\ &=\sum_{\hat{\gamma}\in \Lambda^+}n_{\hat{\gamma}}\sum_{\substack{m\in \mathbb{Z}-\{0\}\\ n\in \mathbb{Z}}}e^{-S_{\hat{\gamma},m,n}}\Bigg[-\frac{\I \tau_2}{8\pi^2}\left(\frac{1}{|m\tau+ n|^3}\left(1-3\frac{(m\tau_1+ n)^2}{|m\tau+ n|^2}\right)-2\pi\I \frac{q_ab^a(m\tau_1+ n)}{|m\tau+ n|^3}\right)\D \zeta^0\\
     &\quad-\frac{\I \tau_2}{4\pi}\frac{q_at^a}{|m\tau+ n|^4}\left(m^2\tau_2^2-|m\tau+ n|(m\tau_1+ n)-2(m\tau_1+ n)^2\right)\D \zeta^0\\
     &\quad+\frac{\I \tau_2}{2}q_at^aq_b\frac{m\tau_1+ n+|m\tau+ n|}{|m\tau+ n|^2}\left(\I b^b+\frac{(m\tau_1+ n)t^b}{|m\tau+ n|}\right)\D \zeta^0\\
    &\quad+\frac{\tau_2}{4\pi}\frac{q_a}{|m\tau+ n|^2}\left(2\pi q_bt^b(m\tau_1+ n+|m\tau+ n|)+\frac{m\tau_1+ n}{|m\tau + n|}\right)\mathrm{\D }\zeta^a\Bigg]\\
        &\quad +\frac{\I \tau_2}{(4\pi)^2}\chi \sum_{\substack{m\in \mathbb{Z}-\{0\}\\ n\in \mathbb{Z}}}\frac{1}{|m\tau+n|^3}\left(1-3\frac{(m\tau_1+ n)^2}{|m\tau+ n|^2}\right)\D \zeta^0\,. \numberthis \label{typeiibthetaplusinst1}\\
    \end{align*}
Similarly, for the remaining term of $\theta_{+}^{P,\text{inst}}$ in \eqref{theta+split} we have that 

\begin{equation}
\begin{split}
    &2R^2\sum_{\gamma}\Omega(\gamma)\widetilde{Z}_{\gamma}\sum_{n>0}e^{-2\pi \I n\zeta_{\gamma}}|\widetilde{Z}_{\gamma}|K_1(4\pi Rn|\widetilde{Z}_{\gamma}|)\left(\frac{\D \widetilde{Z}_{\gamma}}{\widetilde{Z}_{\gamma}}+\frac{\D \overline{\widetilde{Z}}_{\gamma}}{\overline{\widetilde{Z}}_{\gamma}}+\frac{2}{\tau_2}\D \tau_2\right)\\
    &=2R^2\sum_{\gamma}\Omega(\gamma)\sum_{n>0}e^{-2\pi\I n\zeta_{\gamma}}|\widetilde{Z}_{\gamma}|K_1(4\pi Rn|\widetilde{Z}_{\gamma}|)\D \widetilde{Z}_{\gamma}+2R^2\sum_{\gamma}\Omega(\gamma)\frac{\widetilde{Z}_{\gamma}^2}{|\widetilde{Z}_{\gamma}|}\sum_{n>0}e^{-2\pi\I n\zeta_{\gamma}}K_1(4\pi Rn|\widetilde{Z}_{\gamma}|)\D \overline{\widetilde{Z}}_{\gamma}\\
    &\quad +\tau_2\sum_{\gamma}\Omega(\gamma)\widetilde{Z}_{\gamma}\sum_{n>0}e^{-2\pi\I n\zeta_{\gamma}}|\widetilde{Z}_{\gamma}|K_1(4\pi Rn|\widetilde{Z}_{\gamma}|)\D \tau_2\\
    &=\frac{\I \tau_2^2}{16\pi^2}\sum_{\hat{\gamma}\in \Lambda^+}n_{\hat{\gamma}}\partial_{\zeta^a}\partial_{\tau_2}\mathcal{I}_{\hat{\gamma}}^{(2)}\D z^a-\frac{1}{(2\pi)^3}\sum_{\hat{\gamma}\in \Lambda^+}n_{\hat{\gamma}}\partial_{\overline{z}^a}\partial_{\tau_2}\partial_{\overline{z}_{\hat{\gamma}}}\mathcal{I}_{\hat{\gamma}}^{(3)}\D \overline{z}^a\\
    &\quad -\frac{1}{(2\pi)^3}\sum_{\hat{\gamma}\in \Lambda^+}n_{\hat{\gamma}}\left(\partial_{\tau_2}\partial_{\overline{z}_{\hat{\gamma}}}\partial_{\tau_2}\mathcal{I}_{\hat{\gamma}}^{(3)}-\frac{1}{\tau_2}\partial_{\overline{z}_{\hat{\gamma}}}\partial_{\tau_2}\mathcal{I}_{\hat{\gamma}}^{(3)}\right)\D \tau_2-\frac{\I \tau_2\chi}{(4\pi)^2}\partial_{\zeta^0}\partial_{\tau_2}\mathcal{I}_{0}^{(2)}\D \tau_2\,,\\
\end{split}
\end{equation}
so using again Lemma \ref{PoissonLemma} one finds that

\begin{equation}\label{typeiibthetaplusinst2}
    \begin{split}
     &2R^2\sum_{\gamma}\Omega(\gamma)\widetilde{Z}_{\gamma}\sum_{n>0}e^{-2\pi \I n\zeta_{\gamma}}|\widetilde{Z}_{\gamma}|K_1(4\pi Rn|\widetilde{Z}_{\gamma}|)\left(\frac{\D \widetilde{Z}_{\gamma}}{\widetilde{Z}_{\gamma}}+\frac{\D \overline{\widetilde{Z}}_{\gamma}}{\overline{\widetilde{Z}}_{\gamma}}+\frac{2}{\tau_2}\D \tau_2\right)\\ 
     &=\sum_{\hat{\gamma}\in \Lambda^+}n_{\hat{\gamma}}\sum_{\substack{m\in \mathbb{Z}-\{0\}\\ n\in \mathbb{Z}}}e^{-S_{\hat{\gamma},m,n}}\Bigg[\frac{\tau_2^3}{8\pi}q_a\frac{m}{|m\tau+ n|^2}\left(\frac{1}{|m\tau+ n|}+2\pi q_bt^b\right)\D z^a\\
    &\quad +\frac{\tau_2}{8\pi}q_a\frac{m\tau_1+ n+|m\tau+ n|}{m|m\tau+ n|^2} \left(\frac{|m\tau+ n|-(m\tau_1+ n)}{|m\tau+ n|}-2\pi q_at^a(m\tau_1+ n+|m\tau+ n|)\right)\D \overline{z}^a\\
    &\quad+\frac{\I  \tau_2^2}{4\pi}\left(\frac{2\pi q_at^a}{|m\tau+ n|^2} +\frac{1}{|m\tau+ n|^3}\right)(m\tau_1+ n+|m\tau+ n|)\frac{mq_bt^b}{|m\tau+ n|}\D \tau_2\\
    &\quad+\frac{\I  \tau_2^2}{8\pi^2}m\left(\frac{4\pi q_at^a}{|m\tau+ n|^4} +\frac{3}{|m\tau+ n|^5}\right)(m\tau_1+ n+|m\tau+ n|)\D \tau_2-\frac{\I  \tau_2^2}{8\pi^2}m\left(\frac{4\pi q_bt^b}{|m\tau+ n|^3}+\frac{3}{|m\tau+ n|^4}\right)\D \tau_2\Bigg]\\
    &\quad-\frac{3\I \tau_2^2\chi}{(4\pi)^2}\sum_{\substack{m\in \mathbb{Z}-\{0\}\\ n\in \mathbb{Z}}}\frac{m(m\tau_1+ n)}{|m\tau+ n|^5}\D \tau_2\,.
    \end{split}
\end{equation}

On the other hand, using the definitions of $\sigma^{\text{inst}}$, $\widetilde{\zeta}_i^{\text{inst}}$ and $\mathfrak{F}^{\text{w.s.}}$ in \eqref{typeiiacoordcor} and \eqref{Fws}, one can compute $\theta_{+}^{P,\text{w.s.}}$. Summing the contributions of $\theta_{+}^{P,\text{w.s.}}$ to \eqref{typeiibthetaplusinst1} and \eqref{typeiibthetaplusinst2}, and multiplying the whole result by $\I$, one then obtains \eqref{typeiibthetaplusquant}. 
\end{proof}

\subsubsection{Proof of Theorem \ref{theorem1}}\label{proofsec}

Now we have all the preliminary results needed to prove Theorem \ref{theorem1}:

\begin{proof}

First, we remark that the functions $(\xi^i,\widetilde{\xi}_i^{\text{cl}},\alpha^{\text{cl}})$ satisfy

\begin{equation}\label{dccl}
    -2\pi \I(\D \alpha^{\text{cl}}+\widetilde{\xi}_i^{\text{cl}}\D \xi^i-\xi^i\D \widetilde{\xi}_i^{\text{cl}})=f^{\text{cl}}\frac{\D t}{t}+t^{-1}\I\theta_+^{P,\text{cl}} -2\I\theta_3^{P,\text{cl}} -t\I\theta_-^{P,\text{cl}}\,.
\end{equation}

The proof of \eqref{dccl} is almost the same as Proposition \ref{typeIIA1loop}, but we replace in the proof $t\to -\I t$ (recall that we did this rescaling at the beginning of Section \ref{QMsec}) and drop the $\log(t)$ term of $\alpha$ in \eqref{Darbouxcoords1loop},  while keeping in mind that $f^{\text{cl}}$ differs from $f$ in Proposition \ref{typeIIA1loop} by $16\pi c_\ell$, compare \eqref{fsplitting2}, where $f^{\text{w.s}}$ and $f^{\text{inst}}$ need to be set to zero to get 
the $f$ in Proposition \ref{typeIIA1loop}. We then obtain immediately the following identity

\begin{equation}
    4\pi \I \left(-\frac{1}{2}\D(\alpha^{\text{cl}}-\xi^i\widetilde{\xi}_i^{\text{cl}}) - \widetilde{\xi}_i^{\text{cl}}\D \xi^i \right)=f^{\text{cl}}\frac{\D t}{t}+t^{-1}\I\theta_+^{P,\text{cl}} -2\I\theta_3^{P,\text{cl}} -t\I\theta_-^{P,\text{cl}}\,.
\end{equation}
In particular, to prove Theorem \ref{theorem1} it is enough to show using the decompositions at the beginning of Section \ref{prellemsec} that

\begin{equation}\label{contactsimplication}
    4\pi\I\left(\D \alpha^{\text{inst}} -\widetilde{\xi}_i^{\text{inst}}\D \xi^i\right)=(-16\pi c_{\ell}+f^{\text{w.s.}}+f^{\text{inst}})\frac{\D t}{t}+t^{-1}\I (\theta_+^{P,\text{w.s.}}+\theta_+^{P,\text{inst}}) -2\I (\theta_3^{P,\text{w.s.}}+\theta_3^{P,\text{inst}}) -t\I \overline{(\theta_+^{P,\text{w.s.}}+\theta_+^{P,\text{inst}})},
\end{equation}
where we have defined
\begin{equation}\label{dcdiff}
    \alpha^{\text{inst}}:=\alpha+\frac{1}{2}(\alpha^{\text{cl}}-\xi^{i}\widetilde{\xi}_i^{\text{cl}}), \quad \widetilde{\xi}_i^{\text{inst}}:=\widetilde{\xi}_i-\widetilde{\xi}_i^{\text{cl}}\,.
\end{equation}

In order to do this, one can use the explicit formulas \eqref{typeiibdc} to compute the left-hand side of \eqref{contactsimplication}. In terms of the coordinates $(\tau_2,b^a,t^a,\zeta^i,\widetilde{\zeta}_i,\sigma)$, the left-hand side has  only $\D t$, $\D t^a$, $\D b^a$, $\D \zeta^i$ and $\D \tau_2$ components, since $\alpha^{\text{inst}}$ and $\xi^i$ do not depend on $\widetilde{\zeta}_i$ and $\sigma$. Recall that $R=\tau_2/2$. For the $\mathrm{d}t$ component, one obtains
\begin{equation}\label{dtcompcomputation}
\begin{split}
       4\pi \I&\left(\D \alpha^{\text{inst}}-\widetilde{\xi}_i^{\text{inst}}\D \xi^i\right)\Big|_{\D t}\\
       =&\frac{\tau_2^2}{2(2\pi)^2}\sum_{\hat{\gamma}\in \Lambda^{+}\cup\{0\}}n_{\hat{\gamma}}\sum_{(m,n)\in \mathbb{Z}^2-\{0\}}(m\tau_1+n)(t^{-2}+1)\frac{1+t_{+}^{m, n}t}{t-t_{+}^{m, n}}\frac{e^{-S_{\hat{\gamma},m,n}}}{|m\tau+ n|^4}\\
       &+\frac{\tau_2^2}{2(2\pi)^2}\sum_{\hat{\gamma}\in \Lambda^{+}\cup\{0\}}n_{\hat{\gamma}}\sum_{(m,n)\in \mathbb{Z}^2-\{0\}}\left((m\tau_1+n)(t^{-1}-t) -2m\tau_2\right)\frac{1+(t_{+}^{m,n})^2}{(t-t_{+}^{m, n})^2}\frac{e^{-S_{\hat{\gamma},m,n}}}{|m\tau+ n|^4}\\
       &-\frac{\tau_2^2}{2(2\pi)^2}(t^{-2}+1)\sum_{\hat{\gamma}\in \Lambda^+\cup\{0\}}n_{\hat{\gamma}}\sum_{(m,n)\in \mathbb{Z}^2-\{0\}}\left(\frac{1}{m\xi^0+n}+\frac{m\tau_1+n}{|m\tau+ n|^2}\right)\frac{1+t_{+}^{m, n}t}{t-t_{+}^{m, n}}\frac{e^{-S_{\hat{\gamma},m,n}}}{|m\tau + n|^2}\\
      &+\frac{\tau_2^2}{2\pi}t^{-1}\sum_{\hat{\gamma}\in \Lambda^+\cup\{0\}}n_{\hat{\gamma}}q_at^a\sum_{(m,n)\in \mathbb{Z}^2-\{0\}}\frac{e^{-S_{\hat{\gamma},m,n}}}{|m\tau+ n|^2}\\
      =&\frac{\tau_2^2}{2(2\pi)^2}\sum_{\Lambda^{+}\cup\{0\}}n_{\hat{\gamma}}\sum_{(m,n)\in \mathbb{Z}^2-\{0\}}\left[\left(\frac{(m\tau_1+n)(t^{-1}-t)-2m\tau_2}{|m\tau+n|^2}\right)\frac{1+(t_{+}^{m,n})^2}{(t-t_{+}^{m, n})^2}-\frac{(t^{-2}+1)}{m\xi^0+n}\frac{1+t_{+}^{m, n}t}{t-t_{+}^{m, n}}\right]\frac{e^{-S_{\hat{\gamma},m,n}}}{|m\tau+ n|^2}\\
      &+\frac{\tau_2^2}{2\pi}t^{-1}\sum_{\hat{\gamma}\in \Lambda^+\cup\{0\}}n_{\hat{\gamma}}q_at^a\sum_{(m,n)\in \mathbb{Z}^2-\{0\}}\frac{e^{-S_{\hat{\gamma},m,n}}}{|m\tau+ n|^2}\,,
\end{split}
\end{equation}
where in the first equality we are using in the case $m=0$ equation \eqref{m=0def} and
\begin{equation}\label{m=0def2}
    \frac{1+(t_{+}^{0, n})^2}{(t-t_{+}^{0, n})^2}:=\begin{cases}
        -\frac{\D}{\D t}\left(\frac{1+t_{+}^{0, n}t}{t-t_{+}^{0, n}}\right)=1/t^2, \quad n<0\\
        -\frac{\D}{\D t}\left(\frac{1+t_{+}^{0, n}t}{t-t_{+}^{0, n}}\right)=1, \quad n>0
    \end{cases}, 
\end{equation}  while in last equation we left the last sum the same and grouped the rest of the terms (notice that two sums cancel each other). For the term in square brackets in \eqref{dtcompcomputation}, we use for $m\neq 0$ that $t\cdot (m\xi^0+n)=-mR(t-t_{+}^{m,n})(t-t_{-}^{m,n})$ and $t_+^{m,n}t_{-}^{m,n}=-1$, together with the fact that $t_+^{m,n}+t_{-}^{m,n}=2(m\tau_1+n)/m\tau_2$ and $t_+^{m,n}-t_{-}^{m,n}=2|m\tau+n|/m\tau_2$, see \eqref{troots}. We then obtain the following for the case $m\neq 0$:

\begin{align*}
     &\left(\frac{(m\tau_1+n)(t^{-1}-t)-2m\tau_2}{|m\tau+n|^2}\right)\frac{1+(t_{+}^{m,n})^2}{(t-t_{+}^{m, n})^2}-\frac{(t^{-2}+1)}{m\xi^0+n}\frac{1+t_{+}^{m, n}t}{t-t_{+}^{m, n}}\\
     &=\frac{2t_{+}^{m,n}((t_{+}^{m,n}+t_{-}^{m,n})(t^{-1}-t)-4)}{m\tau_2(t_{+}^{m,n}-t_{-}^{m,n})(t-t_{+}^{m,n})^2}+\frac{2t_{+}^{m,n}(t^{-1}+t)}{m\tau_2(t-t_{+}^{m,n})^2}\\
     &=\frac{4}{m\tau_2 t(t_{+}^{m,n}-t_{-}^{m,n})}=\frac{2}{t|m\tau+n|}\,.\numberthis
 \end{align*}
In the case $m=0$ we use \eqref{m=0def} and \eqref{m=0def2}, so that
\begin{equation}
     \left[\left(\frac{(m\tau_1+n)(t^{-1}-t)-2m\tau_2}{|m\tau+n|^2}\right)\frac{1+(t_{+}^{m,n})^2}{(t-t_{+}^{m, n})^2}-\frac{(t^{-2}+1)}{m\xi^0+n}\frac{1+t_{+}^{m, n}t}{t-t_{+}^{m, n}}\right]\Big|_{m=0}=\frac{2}{t|n|}\,.
\end{equation}
Joining everything together we find 

\begin{equation}
    4\pi \I\left(\D \alpha^{\text{inst}}-\widetilde{\xi}_i^{\text{inst}}\D \xi^i\right)\Big|_{\D t}=t^{-1}\frac{\tau_2^2}{(2\pi)^2}\sum_{\hat{\gamma}\in \Lambda^+\cup\{0\}}n_{\hat{\gamma}}\sum_{(m,n)\in \mathbb{Z}^2-\{0\}}\left(\frac{1}{|m\tau+ n|}+2\pi q_at^a\right)\frac{e^{-S_{\hat{\gamma},m,n}}}{|m\tau + n|^2}\,.
\end{equation}

In particular, we find using Proposition \ref{typeiibfquantumprop} that 

\begin{equation}
       4\pi \I \left(\D \alpha^{\text{inst}}-\widetilde{\xi}_i^{\text{inst}}\D \xi^i\right)\Big|_{\D t}=t^{-1}\left(f^{\text{w.s.}}+f^{\text{inst}}-\frac{\chi}{12}\right)\;,
\end{equation}
matching the required term on the right-hand side of \eqref{contactsimplication} provided that $c_{\ell}=\frac{\chi}{192\pi}$. On the other hand, for the other components $\D t^a$, $\D b^a$, $\D \zeta^i$ and $\D \tau_2$ of the left-hand side of \eqref{contactsimplication}, one needs to show that they decompose into three summands with factors $t^{-1}$, $t^0$ and $t$, which should match the corresponding component of $\I (\theta_+^{P,\text{w.s.}}+\theta_+^{P,\text{inst}})$, $-2\I (\theta_3^{P,\text{w.s.}}+\theta_3^{P,\text{inst}})$, and $-\I \overline{(\theta_+^{P,\text{w.s.}}+\theta_+^{P,\text{inst}})}$ in \eqref{contactsimplication}, respectively. In particular, the summand with the $t$ factor should correspond to the conjugate of the $t^{-1}$ summand, for each component. For example the computation for the $\D t^a$ component gives

\begin{align*}
    &4\pi\I\left(\D \alpha^{\text{inst}}-\widetilde{\xi}_i^{\text{inst}}\D \xi^i\right)\Big|_{\D t^a}\\
    &=\frac{\tau_2^2}{4\pi}\sum_{\hat{\gamma}\in \Lambda^+\cup\{0\}}n_{\hat{\gamma}}q_a\sum_{(m,n)\neq (0,0)} \left(\frac{m\tau_1+ n}{|m\tau+ n|}(t^{-1}-t)-\frac{2m\tau_2}{|m\tau+ n|}+(t^{-1}+t)\right)\frac{1+t_+^{m, n}t}{t-t_+^{m, n}}\frac{e^{-S_{\hat{\gamma},m,n}}}{|m\tau+ n|^2}\\
    &=\frac{\tau_2^2}{4\pi}\sum_{\hat{\gamma}\in \Lambda^{+}\cup\{0\}}n_{\hat{\gamma}}q_a\sum_{(m,n)\neq (0,0)} \left(-t^{-1}\frac{m\tau_2}{|m\tau+ n|}-2\frac{m\tau_1+ n}{|m\tau+ n|} +t\frac{m\tau_2}{|m\tau+ n|}\right)\frac{e^{-S_{\hat{\gamma},m,n}}}{|m\tau+ n|^2}\\ \numberthis
\end{align*}
where for the last equality we used the identities for $t_{\pm}^{m,n}$ used in the computation of the $\D t$ component. Comparing with Proposition \ref{typeiibtheta3quantumprop} and \ref{typeiibthetapquantumprop} one readily sees that the $t^{-1}$ and $t^{0}$ summands match the corresponding $\D t^a$ component. Furthermore, the summand with the $t$-factor is seen to correspond to the conjugate of the summand with the $t^{-1}$ factor, since $\overline{S_{\hat{\gamma},m,n}}=S_{\hat{\gamma},-m,-n}$, so the summand with the $t$-factor also gives the required contribution. \\

By a similar (albeit tedious) computation one can check using Proposition \ref{typeiibtheta3quantumprop} and \ref{typeiibthetapquantumprop} that the remaining components on the left-hand side and right-hand side of  \eqref{contactsimplication} match, showing that \eqref{typeiibdc} indeed define Darboux coordinates for the contact structure associated to instanton corrected q-map spaces. 
\end{proof}

\subsection{S-duality}\label{Sdsec}

Recall that we have the S-duality action on $\overline{\mathcal{N}}_{\text{IIB}}^{\text{cl}}$, defined by \eqref{sl2can}. We start by defining a lift of the S-duality action from $\overline{\mathcal{N}}_{\text{IIB}}^{\text{cl}}$ to $\overline{\mathcal{N}}_{\text{IIB}}^{\text{cl}}\times \mathbb{C}P^1$, following \cite{QMS,Sduality}.

\begin{definition}
Given an element 

\begin{equation}
    A=\begin{pmatrix}
a & b \\
c & d \\
\end{pmatrix}\in \mathrm{SL}(2,\mathbb{Z})
\end{equation} we lift the action of $A$ from $\overline{\mathcal{N}}_{\text{IIB}}^{\text{cl}}$ to  $\overline{\mathcal{N}}_{\text{IIB}}^{\text{cl}}\times \mathbb{C}P^1$ by defining the action on the fiber coordinate $t\in \mathbb{C}P^1$ over $(\tau_1+\I\tau_2,b^a+\I t^a, c^a,c_a,c_0,\psi)\in \overline{\mathcal{N}}_{\text{IIB}}^{\text{cl}}$ by:
\begin{equation}\label{sdualitylift}
     \begin{pmatrix}
a & b\\
c & d\\
\end{pmatrix}\cdot t:=\begin{dcases}
    t \quad \text{if $c=0$ and $a>0$}\\
    -1/t \quad \text{if $c=0$ and $a<0$}\\
    \frac{1+t^{c,d}_{+}t}{t^{c,d}_{+}-t} \quad \text{if $c\neq 0$}
\end{dcases} \,.
\end{equation}
\end{definition}
\begin{remark}\label{liftremark} Since $\mathrm{SL}(2,\mathbb{Z})$ is generated by 
\begin{equation}\label{sl2gen}
    T=\begin{pmatrix}
    1 & 1\\
    0 & 1\\
    \end{pmatrix}, \quad S=\begin{pmatrix}
    0 & -1\\
    1 & 0\\
    \end{pmatrix}
\end{equation}
it is not hard to check that \eqref{sdualitylift} defines a lift of the S-duality action. We note that the transformation on $t$ depends on the base point in the case where $c\neq 0$, since $t_{+}^{c,d}$ depends on $\tau$. We further remark that by using the fact that $t_{+}^{c,d}t_{-}^{c,d}=-1$ the $t$-variable transformation when $c\neq 0$ can be rewritten as
\begin{equation}
     \begin{pmatrix}
a & b\\
c & d\\
\end{pmatrix}\cdot t=\frac{1+t^{c,d}_{+}t}{t^{c,d}_{+}-t}=-\frac{t_{-}^{c,d}-t}{1+t_{-}^{c,d}t}\,.
\end{equation}
\end{remark}
Now let $(\widetilde{N},g_{\overline{N}})$ be an instanton corrected q-map space constructed in Section \ref{settingsec}.  By using the mirror map \eqref{MM}, we can identify $\widetilde{N}\subset \overline{\mathcal{N}}_{\text{IIA}}$ with an open subset of $\mathcal{M}^{-1}(\widetilde{N})\subset \overline{\mathcal{N}}_{\text{IIB}}\subset \overline{\mathcal{N}}_{\text{IIB}}^{\text{cl}}$. It then follows that its twistor space satisfies $\mathcal{Z}\cong \widetilde{N}\times \mathbb{C}P^1\subset \overline{\mathcal{N}}_{\text{IIB}}^{\text{cl}}\times \mathbb{C}P^1$. If $A\in \mathrm{SL}(2,\mathbb{Z})$ leaves $\widetilde{N}\subset \overline{\mathcal{N}}_{\text{IIB}}^{\text{cl}}$ invariant, then we get an induced action of $A$ on $\mathcal{Z}$. Assuming that $A$ acts on $\mathcal{Z}$, we now show the key transformation property of the Darboux coordinates, following again \cite{QMS,Sduality}:
\begin{proposition}\label{twistorsdualityprop} Let $(\widetilde{N},g_{\overline{N}})$ be an instanton corrected q-map space with $1$-loop parameter $c_{\ell}=\frac{\chi}{192\pi}$ such that $\widetilde{N}$ is invariant under 
\begin{equation}
    A=\begin{pmatrix}
a & b \\
c & d \\
\end{pmatrix}\in \mathrm{SL}(2,\mathbb{Z})\,.
\end{equation}
Furthermore consider the Darboux coordinates $(\xi^i,\widetilde{\xi}_i,\widetilde{\alpha})$ of $\mathcal{Z}$, where  $\widetilde{\alpha}:=\alpha -\xi^i\widetilde{\xi}_i$ and $(\xi^i,\widetilde{\xi}_i,\alpha)$ are the Darboux coordinates given by \eqref{typeiibdc}. Then $(\xi^i,\widetilde{\xi}_i,\widetilde{\alpha})$  satisfy the following transformation under the lift of the action of $A$ given in \eqref{sdualitylift}:
\begin{equation}\label{twistorSduality}
    \begin{split}
        \xi^0&\to \frac{a\xi^0+b}{c\xi^0+d}, \quad \xi^a \to \frac{\xi^a}{c\xi^0+d}, \quad \widetilde{\xi}_a \to \widetilde{\xi}_a +\frac{c}{2(c\xi^0+d)}\kappa_{abc}\xi^b\xi^c\\
        \begin{pmatrix}
        \widetilde{\xi}_0 \\
        \widetilde{\alpha}
        \end{pmatrix}&\to \begin{pmatrix}
        d & -c\\
        -b & a \\
        \end{pmatrix}\begin{pmatrix}
        \widetilde{\xi}_0 \\
        \widetilde{\alpha}
        \end{pmatrix} +\frac{1}{6}\kappa_{abc}\xi^a\xi^b\xi^c\begin{pmatrix}
        c^2/(c\xi^0+d)\\
        -[c^2(a\xi^0+b)+2c]/(c\xi^0+d)^2\\
        \end{pmatrix}.
    \end{split}
\end{equation}

\end{proposition}

\begin{proof}

To show the required transformation rule we follow the same argument of \cite{QMS}, which we include here with more detail. It is straightforward to check that the classical coordinates $(\xi^i,\widetilde{\xi}_i^{\text{cl}},\widetilde{\alpha}^{\text{cl}})$, where $\widetilde{\alpha}^{\text{cl}}:=-\frac{1}{2}(\alpha^{\text{cl}}-\xi^i\widetilde{\xi}_i^{\text{cl}})-\xi^{i}\widetilde{\xi}_i^{\text{cl}}=-\frac{1}{2}(\alpha^{\text{cl}}+\xi^i\widetilde{\xi}_i^{\text{cl}})$ satisfy \eqref{twistorSduality} \cite{APSV}. Hence, it is enough to show that $\widetilde{\xi}_i^{\text{inst}}=\widetilde{\xi}_i-\widetilde{\xi}_i^{\text{cl}}$ and $\widetilde{\alpha}^{\text{inst}}=\widetilde{\alpha}-\widetilde{\alpha}^{\text{cl}}$ transform under under the action of $A\in \mathrm{SL}(2,\mathbb{Z})$ by 

\begin{equation}\label{insttransrules}
    \widetilde{\xi}_a^{\text{inst}}\to \widetilde{\xi}_a^{\text{inst}}, \quad  \begin{pmatrix}
        \widetilde{\xi}_0^{\text{inst}} \\
        \widetilde{\alpha}^{\text{inst}}
        \end{pmatrix}\to \begin{pmatrix}
        d & -c\\
        -b & a \\
        \end{pmatrix}\begin{pmatrix}
        \widetilde{\xi}_0^{\text{inst}} \\
        \widetilde{\alpha}^{\text{inst}}
        \end{pmatrix}.
\end{equation}
Using that \eqref{dcdiff} and \eqref{typeiibdc}, we find that  $\widetilde{\alpha}^{\text{inst}}=\alpha^{\text{inst}}-\xi^i\widetilde{\xi}_i^{\text{inst}}$ and

\begin{equation}
    \begin{split}
        \widetilde{\alpha}^{\text{inst}}&=-\frac{\I\tau_2}{2(2\pi)^3}\sum_{\hat{\gamma}\in \Lambda^+\cup\{0\}}n_{\hat{\gamma}}^{(0)}\sum_{(m,n)\in \mathbb{Z}^2-\{0\}}\left(\frac{\xi^0}{m\xi^0+ n}+\frac{m|\tau|^2+ n\tau_1}{|m\tau+ n|^2}\right)\frac{1+t_{+}^{m, n}t}{t-t_{+}^{m, n}}\frac{e^{-S_{\hat{\gamma},m,n}}}{|m\tau + n|^2}\\
        &+\frac{\tau_2}{2(2\pi)^2}\sum_{\hat{\gamma}\in \Lambda_{+}\cup \{0\}}n_{\hat{\gamma}}^{(0)} \sum_{(m,n)\in \mathbb{Z}^2-\{0\}}\left(q_ac^a \frac{1+t_{+}^{m, n}t}{t-t_{+}^{m, n}}+\I q_at^a\left(\tau_1\frac{1-t_{+}^{m, n}t}{t-t_{+}^{m, n}}-\tau_2\frac{t+t_{+}^{m, n}}{t-t_{+}^{m, n}}\right)\right)\frac{e^{-S_{\hat{\gamma},m,n}}}{|m\tau+ n|^2}\,.
    \end{split}
\end{equation}
where, in addition to \eqref{m=0def}, we have defined
\begin{equation}
    \frac{t+t_{+}^{0,n}}{t-t_{+}^{0,n}}:=\begin{dcases}
        -1 \quad n>0\\
        1 \quad n<0\\
    \end{dcases}\,.
\end{equation}

We start by showing that $\widetilde{\xi}_a^{\text{inst}}$ is invariant under the S-duality $\mathrm{SL}(2,\mathbb{Z})$ action, where

\begin{equation}\label{xiainst}
    \widetilde{\xi}_a^{\text{inst}}=\frac{\tau_2}{8\pi^2}\sum_{\hat{\gamma}\in \Lambda^{+}\cup\{0\}}n_{\hat{\gamma}}q_a\sum_{(m,n)\in \mathbb{Z}^2-\{0\}}\frac{e^{-S_{\hat{\gamma},m,n}}}{|m\tau+ n|^2}\frac{1+t_{+}^{m, n}t}{t-t_{+}^{m,n}}\,.
\end{equation}
Denoting 
\begin{equation}
    \begin{pmatrix}
    m'\\
    n'\\
    \end{pmatrix}=\begin{pmatrix}
    a &  b\\
    c & d\\ 
    \end{pmatrix}^\intercal\begin{pmatrix}
    m\\
    n\\
    \end{pmatrix}, 
\end{equation}
where the $\intercal$ subscript denotes the transpose matrix, we find the following transformation rules under $A\in \mathrm{SL}(2,\mathbb{Z})$:

\begin{equation}\label{transrules1}
    \tau_2 \to \frac{\tau_2}{|c\tau+d|^2}, \quad |m\tau + n|\to \frac{|m'\tau+ n'|}{|c\tau +d|}, \quad S_{\hat{\gamma},m,n}\to S_{\hat{\gamma},m',n'}\;.
\end{equation}
We now discuss the transformation properties of the factor
\begin{equation}
    \frac{1+t_{+}^{m, n}t}{t-t_{+}^{m, n}}\,.
\end{equation}
We first consider the case $m\neq 0$. Because $t_{\pm}^{m,n}=t_{\pm}^{km,kn}$ for any $k>0$, we can assume that $(m,n)$ are coprime, so that there exist $p,q\in \mathbb{Z}$ so that

\begin{equation}
    \begin{pmatrix}
    p & q\\
    m & n\\
    \end{pmatrix}\in \mathrm{SL}(2,\mathbb{Z})\,.
\end{equation}
Since \eqref{sdualitylift} defines an action on $t$, we find that (making the dependence of $t_{+}^{m,n}$ on $\tau$ explicit)
\begin{equation}
    \frac{1+t_{+}^{m', n'}(\tau)t}{t_{+}^{m', n'}(\tau)-t}=\left(\begin{pmatrix}
    p & q\\
    m & n \\
    \end{pmatrix}\cdot \begin{pmatrix}
    a & b\\
    c & d\\
    \end{pmatrix}\right)\cdot t=\begin{pmatrix}
    p & q\\
    m & n \\
    \end{pmatrix}\cdot \left(\begin{pmatrix}
    a & b\\
    c & d\\
    \end{pmatrix}\cdot t\right)=\frac{1+t_{+}^{m,n}\left(\frac{a\tau+b}{c\tau+d}\right)\left(\begin{pmatrix}
    a & b\\
    c & d\\
    \end{pmatrix}\cdot t\right)}{t_{+}^{m,n}\left(\frac{a\tau+b}{c\tau+d}\right)-\begin{pmatrix}
    a & b\\
    c & d\\
    \end{pmatrix}\cdot t}
\end{equation}
so we have the transformation rule

\begin{equation}\label{transrules2}
    \frac{1+t_{+}^{m, n}t}{t-t_{+}^{m, n}}\to \frac{1+t_{+}^{m', n'}t}{t-t_{+}^{m', n'}}\,.
\end{equation}
The same transformation rule \eqref{transrules2} follows when $m=0$ by using the property that $t_{\pm}^{c,d}=t_{\pm}^{kc,kd}$ when $k>0$ and $t_{\pm}^{c,d}=t_{\mp}^{kc,kd}$ when $k<0$, together with \eqref{m=0def}. Hence, from \eqref{transrules1} and \eqref{transrules2} it follows that \eqref{xiainst} is invariant under the $\mathrm{SL}(2,\mathbb{Z})$-action.\\

We now verify the transformation rule of $\widetilde{\xi}_0^{\text{inst}}$ in \eqref{insttransrules}. For this we use that under the $\mathrm{SL}(2,\mathbb{Z})$-action:

\begin{equation}\label{transrules3}
    m\xi^0+n \to \frac{m'\xi^0+ n'}{c\xi^0+d}, \quad \frac{m\tau_1+ n}{|m\tau + n|^2}\to \frac{c(m'|\tau|^2 +n'\tau_1)+d(m'\tau_1+n')}{|m'\tau+n'|^2}\,.
\end{equation}
After a rather lengthy but straightforward computation, one can also check that 
\begin{equation}\label{transrules5}
     \frac{1-t_{+}^{m,n}t}{t-t_{+}^{m, n}}t^a\to (c\tau_1+d)\frac{1-t_{+}^{m', n'}t}{t-t_{+}^{m', n'}}t^a-c\frac{t+t_{+}^{m', n'}}{t-t_{+}^{m', n'}}\tau_2t^a\;.
\end{equation}

We therefore find using \eqref{transrules1}, \eqref{transrules2}, \eqref{transrules3} and \eqref{transrules5}, that the terms of $\widetilde{\xi}_0^{\text{inst}}$ transform as follows

\begin{equation}
\begin{split}
    &\frac{\I \tau_2}{16\pi^3}\sum_{\hat{\gamma}\in \Lambda^+\cup\{0\}}n_{\hat{\gamma}}^{(0)}\sum_{(m,n)\in \mathbb{Z}^2-\{0\}}\left(\frac{1}{m\xi^0+n}+\frac{m\tau_1+ n}{|m\tau+ n|^2}\right)\frac{1+t_{+}^{m, n}t}{t-t_{+}^{m, n}}\frac{e^{-S_{\hat{\gamma},m,n}}}{|m\tau + n|^2}\\
    &\to d\left(\frac{\I \tau_2}{16\pi^3}\sum_{\hat{\gamma}\in \Lambda^+\cup\{0\}}n_{\hat{\gamma}}^{(0)}\sum_{(m,n)\in \mathbb{Z}^2-\{0\}}\left(\frac{1}{m\xi^0+n}+\frac{m\tau_1+ n}{|m\tau+ n|^2}\right)\frac{1+t_{+}^{m, n}t}{t-t_{+}^{m, n}}\frac{e^{-S_{\hat{\gamma},m,n}}}{|m\tau + n|^2}\right)\\
    &\quad\quad\quad\quad+c\left(\frac{\I\tau_2}{16\pi^3}\sum_{\hat{\gamma}\in \Lambda^+\cup\{0\}}n_{\hat{\gamma}}^{(0)}\sum_{(m,n)\in \mathbb{Z}^2-\{0\}}\left(\frac{\xi^0}{m\xi^0+ n}+\frac{m|\tau|^2+ n\tau_1}{|m\tau+ n|^2}\right)\frac{1+t_{+}^{m, n}t}{t-t_{+}^{m, n}}\frac{e^{-S_{\hat{\gamma},m,n}}}{|m\tau + n|^2}\right)\;,\\
\end{split}
\end{equation}
while

\begin{equation}
    \begin{split}
        &-\frac{\tau_2}{8\pi^2}\sum_{\hat{\gamma}\in \Lambda_{+}\cup \{0\}}n_{\hat{\gamma}} \sum_{(m,n)\in \mathbb{Z}^2-\{0\}}\left(q_ab^a \frac{1+t_{+}^{m, n}t}{t-t_{+}^{m, n}}+\I q_at^a\frac{1-t_{+}^{m, n}t}{t-t_{+}^{m, n}}\right)\frac{e^{-S_{\hat{\gamma},m,n}}}{|m\tau+ n|^2}\\
        &\to d\left(-\frac{\tau_2}{8\pi^2}\sum_{\hat{\gamma}\in \Lambda_{+}\cup \{0\}}n_{\hat{\gamma}} \sum_{(m,n)\in \mathbb{Z}^2-\{0\}}\left(q_ab^a \frac{1+t_{+}^{m, n}t}{t-t_{+}^{m, n}}+\I q_at^a\frac{1-t_{+}^{m, n}t}{t-t_{+}^{m, n}}\right)\frac{e^{-S_{\hat{\gamma},m,n}}}{|m\tau+ n|^2}\right)\\
        &\quad\quad+c\left(-\frac{\tau_2}{8\pi^2}\sum_{\hat{\gamma}\in \Lambda_{+}\cup \{0\}}n_{\hat{\gamma}}^{(0)} \sum_{(m,n)\in \mathbb{Z}^2-\{0\}}\left(q_ac^a \frac{1+t_{+}^{m, n}t}{t-t_{+}^{m, n}}+\I q_at^a\left(\tau_1\frac{1-t_{+}^{m, n}t}{t-t_{+}^{m, n}}-\tau_2\frac{t+t_{+}^{m, n}}{t-t_{+}^{m, n}}\right)\right)\frac{e^{-S_{\hat{\gamma},m,n}}}{|m\tau+ n|^2}\right)
    \end{split}
\end{equation}
so overall
\begin{equation}
    \widetilde{\xi}_0^{\text{inst}}\to d\widetilde{\xi}_0^{\text{inst}}-c\widetilde{\alpha}^{\text{inst}}\,.
\end{equation}

Finally, we check the transformation rule for $\widetilde{\alpha}^{\text{inst}}$ in \eqref{insttransrules}. Given that we know the transformation rules of the rest of the variables, it is easy to check that this is equivalent to showing that $\alpha^{\text{inst}}=\widetilde{\alpha}^{\text{inst}}+\xi^i\widetilde{\xi}_i^{\text{inst}}$ transforms by
\begin{equation}
    \alpha^{\text{inst}}\to \frac{\alpha^{\text{inst}}}{c\xi^0+d}\,.
\end{equation}
To check the later, one uses the fact that under the action of $A\in \mathrm{SL}(2,\mathbb{Z})$

\begin{equation}\label{transrules4}
    (m\tau_1+n)(t^{-1}-t)-2m \tau_2 \to \frac{(m'\tau_1+ n')(t^{-1}-t)-2m'\tau_2}{c\xi^0+d}\,.
    \end{equation}
The result then follows immediately from the formula for $\alpha^{\text{inst}}$ obtained via \eqref{dcdiff} and \eqref{typeiibdc}, and the transformations \eqref{transrules1}, \eqref{transrules4}.
\end{proof}
\begin{theorem}\label{theorem2} Let $(\widetilde{N},g_{\overline{N}})$ be an instanton corrected q-map space with $1$-loop parameter $c_{\ell}=\frac{\chi}{192\pi}$. 
If, after possibly restricting $\widetilde{N}$,  we have that $\widetilde{N}$ is invariant under the action of $A\in \mathrm{SL}(2,\mathbb{Z})$ by the S-duality action \eqref{sl2can}, then $A$ also acts by isometries on $(\widetilde{N},g_{\overline{N}})$. In particular, if $\widetilde{N}$ is invariant under the full S-duality action, then $\mathrm{SL}(2,\mathbb{Z})$ acts by isometries on $(\widetilde{N},g_{\overline{N}})$.
\end{theorem} 
\begin{proof}

 Since $A$ leaves $\widetilde{N}$ invariant, we have $S_A:\mathcal{Z}\to \mathcal{Z}$ the diffeomorphism obtained by lifting the S-duality action of $A$ to the twistor space $\mathcal{Z}\cong \widetilde{N}\times \mathbb{C}P^1$ by \eqref{sdualitylift}. We want to show that $S_A$ is a twistor space automorphism (i.e. it is holomorphic, preserves the contact distribution, and commutes with the real structure of the twistor space), so that we can conclude that the action of $A$ on $(\widetilde{N},g_{\overline{N}})$ is isometric. \\
 
 We first want to show that $S_A$ is holomorphic. Notice that if $S^1\subset \mathbb{C}P^1$ denotes the compactification of the real line $\mathbb{R}\subset \mathbb{C}$ (with respect to the $t$-coordinate), the Darboux coordinates $(\xi^i,\widetilde{\xi}_i,\widetilde{\alpha})$ are defined on the open dense subset of $\mathcal{Z}$ given by
 \begin{equation}
     \mathcal{Z}_D:=\mathcal{Z}-\widetilde{N}\times S^1,
 \end{equation} since for each $p\in \widetilde{N}$, the singularities of the Darboux coordinates are at $t=0,\infty$ and $t\in \{t_+^{m,n}\}_{m\neq 0,n\in \mathbb{Z}}\subset \mathbb{R}\subset S^1$. Using the coordinates $(\xi^i,\widetilde{\xi}_i,\widetilde{\alpha})$ on $\mathcal{Z}$, and using that Darboux coordinates for a holomorphic contact structure must be holomorphic coordinates (see for example the proof of the second statement of \cite[Proposition 7]{CTSduality}), the transformation rule \eqref{twistorSduality} shows that the map 
 \begin{equation}
     S_A:\mathcal{Z}_D\cap S^{-1}_A(\mathcal{Z}_D)\to \mathcal{Z}_D
 \end{equation}
 is holomorphic. 
On the other hand, if $\mathcal{I}$ denotes the holomorphic structure of $\mathcal{Z}$, we have that the diffeomorphism $S_A$ is holomorphic if and only if
\begin{equation}\label{holcond}
    \mathcal{I}\circ \D S_A=\D S_A \circ \mathcal{I}\,.
\end{equation}
Since $\mathcal{Z}_D\cap S^{-1}_A(\mathcal{Z}_D)$ is the intersection of two open dense subsets of $\mathcal{Z}$, we have that $\mathcal{Z}_D\cap S^{-1}_A(\mathcal{Z}_D)$ must be dense in $\mathcal{Z}$. We then have that \eqref{holcond} holds on a dense set of $\mathcal{Z}$. Since $S_A$ and $\mathcal{I}$ are globally defined on $\mathcal{Z}$, we then conclude by continuity that \eqref{holcond} must hold on all of $\mathcal{Z}$, and $S_A$ must be holomorphic.\\

On the other hand, the fact that the coordinates $(\xi^i,\widetilde{\xi}_i,\widetilde{\alpha}=\alpha-\xi^i\widetilde{\xi}_i)$ transform via \eqref{twistorSduality} implies that  

\begin{equation}
    S_A^*(\D \widetilde{\alpha} + \xi^i\D \widetilde{\xi}_i)=\frac{\D \widetilde{\alpha} + \xi^i\D \widetilde{\xi}_i}{c\xi^0+d}\,.
\end{equation}
Hence $S_A$ preserves the contact distribution $\text{Ker}(\lambda)$ on a dense subset of $\mathcal{Z}$.  
We then conclude as before by continuity and the fact that the contact distribution is globablly defined, that the contact distribution must be globally preserved by $S_A$.\\

Finally, to check that the action of $A$ preserves the real structure, it is enough to check that (\ref{sdualitylift}) commutes with the antipodal map $t \to -1/\overline{t}$. Indeed, we have for $c\neq 0$ that

\begin{equation}
    \begin{pmatrix}
    a & b\\
    c & d\\ 
    \end{pmatrix}\cdot \Big(-\frac{1}{\overline{t}}\Big)=\frac{\overline{t}-t_{+}^{c,d}}{t_+^{c,d}\overline{t}+1}=-\overline{\Big[\begin{pmatrix}
    a & b\\
    c & d\\ 
    \end{pmatrix}\cdot t \Big]}^{-1}\,,
\end{equation}
where we have used that $t_{+}^{c,d}\in \mathbb{R}$. The case when $c=0$ follow by a trivial computation. \\

 Hence, we conclude that the action of $A$ is via twistor space automorphisms, and hence $\mathrm{SL}(2,\mathbb{Z})$ must act via isometries on $(\widetilde{N},g_{\overline{N}})$.\\
 
 For the final statement, if $\widetilde{N}$ is invariant under the $\mathrm{SL}(2,\mathbb{Z})$ S-duality action, then we can lift the $\mathrm{SL}(2,\mathbb{Z})$ action to the twistor space via \eqref{sdualitylift}. By the proof of the first statement, this lift acts by twistor space automorphisms, so $\mathrm{SL}(2,\mathbb{Z})$ must act by isometries on $(\widetilde{N},g_{\overline{N}})$.
\end{proof}

Recall that $\mathrm{SL}(2,\mathbb{Z})$ is generated by $T$ and $S$ given in \eqref{sl2gen}. The transformations generated by $T$ correspond to part of the usual Heisenberg isometries (see Section \ref{universalisosec} below). On the other hand, the transformation given by $S$ in the ``non-trivial" transformation that exchanges weak and strong coupling in the type IIB string theory setting. The following theorem guarantees that an instanton corrected q-map space always carries (after possibly restricting $\widetilde{N}$), an action by isometries by $S\in \mathrm{SL}(2,\mathbb{Z})$.

\begin{theorem}\label{Z4prop} Let $S\in \mathrm{SL}(2,\mathbb{Z})$ be given as in \eqref{sl2gen}, and consider an instanton corrected q-map space $(\widetilde{N},g_{\overline{N}})$  with $c_{\ell}=\frac{\chi}{192\pi}$. Then we can find a non-empty $S$-invariant open subset $\widetilde{N}_S\subset \widetilde{N}$ such that the restricted instanton corrected q-map space $(\widetilde{N}_S,g_{\overline{N}})$ is positive definite and carries a $\mathbb{Z}_4$-action by isometries generated by $S$. The open subset $\widetilde{N}_S$ is given by 

\begin{equation}\label{Sinvariantsub}
  \widetilde{N}_S=\{p\in \widetilde{N}\; | \; \epsilon<\tau_2, \;\; \epsilon < \frac{\tau_2}{|\tau_1|^2+|\tau_2|^2}, \;\; t^a>K, \;\; |\tau|t^a> K \} 
\end{equation}
for some $0<\epsilon<1$ and $K>0$.
\end{theorem}
\begin{proof}  We would first like to show that there is $K>0$ and $0<\epsilon<1$ such that $g_{\overline{N}}$ is defined and positive definite on 

\begin{equation}
    \widetilde{N}_{K,\epsilon}:=\{(\tau_2,b^a+\mathrm{i}t^a,\zeta^i,\widetilde{\zeta}_i,\sigma)\in \widetilde{N} \; | \; \tau_2>\epsilon, \;\; t^a>K, \;\; a=1,...,n\}\,.
\end{equation}

In order to show this we first study the CASK geometry of signature $(2,2n)$ defined by $(M,-\mathfrak{F})$. In terms of the natural holomorphic coordinates $Z^i$, $i=0,...,n$, the CASK geometry has a K\"{a}hler potential $(Z^i,\overline{Z}^i)$ given by

\begin{equation}
    k(Z^i,\overline{Z}^i)=-\text{Im}(\tau_{ij})Z^i\overline{Z}^j=-|Z^0|^2\text{Im}(\tau_{ij})z^i\overline{z}^j\,, \quad z^i=\frac{Z^i}{Z^0}\,.
\end{equation}
Since $Z^0\neq 0$ on $M$, we can use instead the holomorphic coordinates $(Z^0,z^a)$. In terms of $(Z^0,z^a)$ we find using the formula \eqref{prepotential} for $\mathfrak{F}$, that 

\begin{equation}\label{KpotCASK}
\begin{split}
    k&(Z^0,z^a)\\
    &=|Z^0|^2\left(4h(t)-\frac{\chi \zeta(3)}{(2\pi)^3}+\frac{2}{(2\pi)^3}\sum_{q_a\gamma^a \in \Lambda^+}n_{\gamma}\text{Re}(\mathrm{Li}_3(e^{2\pi iq_az^a}))+\frac{2}{(2\pi )^2}\sum_{q_a\gamma^a \in \Lambda^+}n_{\gamma}\text{Re}(\mathrm{Li}_2(e^{2\pi iq_az^a}))q_at^a\right)\,,\\
\end{split}
\end{equation}
where we recall that $h(t)$ is the cubic polynomical defining the PSR manifold (see Section \ref{settingsec}). From the previous formula, it immediately follows that the coefficients of the CASK metric $g_M$ in the coordinates $(Z^0,z^a)$ only depend periodically on $b^a=\text{Re}(z^a)$. Furthermore, as $t^a=\text{Im}(z^a)\to \infty$ the classical terms due to $\mathfrak{F}^{\text{cl}}$ dominate over the terms due to $\mathfrak{F}^{\text{w.s}}$, which are either independent of $t^a$, or exponentially decreasing as $t^a\to \infty$. Since the classical terms must satisfy the CASK conditions, we find that there is $K>0$, such that

\begin{equation}
    M_K:=\{ (Z^0,b^a+\I t^a)\in M^q  \; | \; t^a>K, \;\; a=1,...,n\}\subset M\,,
\end{equation}
where $M$ was defined in Section \ref{settingsec} as the maximal open set of $M^{\text{cl}}$ where the CASK geometry defined by $(M,\mathfrak{F})$ has signature $(2n,2)$ and $\text{Im}(\tau_{ij})Z^i\overline{Z}^j<0$.\\

Now let us look at the tensor $T$ defined in \eqref{non-deg}, determining the compatibility condition between the CASK structure and the BPS structure. The instanton constribution to $T$ due to the BPS indices is by an expression of the form

\begin{equation}
    \sum_{\gamma}\Omega(\gamma)\sum_{n>0}e^{2\pi \I n\zeta_{\gamma}}K_0(2\pi n \tau_2|q_0+q_a(b^a+\I t^a)|)|\D Z_{\gamma}|^2, \quad \gamma=q_0\gamma^0+q_a\gamma^a\,.
\end{equation}

Due to the exponential decay of the Bessel functions $K_0(x)$ as $x\to \infty$ and the convergence property of the BPS structure, by bounding below $\tau_2$ and increasing $K$ if necessary, we can make the CASK metric term of \eqref{non-deg} dominate over the instanton corrections uniformly in the rest of the parameters, and in particular make $T$ horizontally non-degenerate on this region. More precisely, if $0<\epsilon<1$, there is $K>0$ sufficiently big such that the (pseudo-)HK metric $g_{N}$ is defined on 
\begin{equation}
    N_{K,\epsilon}:=\{ (Z^0,z^a,\zeta^i,\widetilde{\zeta}_i)\in M\times (\mathbb{R}/\mathbb{Z})^{2n+2} \; | \; \epsilon<\tau_2=|Z^0|,\;\; t^a>K, \; a=1,...,n\}\,.
\end{equation}
By the same arguments as those given in the previous paragraphs, we find that as $t^a\to \infty$ for $a=1,...,n$, the functions $f$, $f_3$ and $g_{N}(V,V)$ defined on Section \ref{QKCASKrecap} are asymptotically approximated  (uniformly in the other parameters) by $f^{\text{cl}}$, $f_3^{\text{cl}}$ and $g_{N}^{\text{cl}}(V,V)$, where the superscript $^{\text{cl}}$ refers to the corresponding functions obtained by setting $n_{\hat{\gamma}}=0$ for all $\hat{\gamma}\in \Lambda^{+}\cup\{0\}$. Since we have $f^{\text{cl}}>0$, $f_3^{\text{cl}}<0$ and $g_{N}^{\text{cl}}(V,V)\neq 0$, it follows that we can pick $K$ such that on $N_{K,\epsilon}$ we obtain that
\begin{equation}
    f>0, \quad f_3 <0, \quad  g_N(V,V) \neq 0\,.
\end{equation}
This ensures by the end of Theorem \ref{QKCASKdomainformulas} that $g_{\overline{N}}$ is defined and positive definite on 
\begin{equation}
    \overline{N}_{K,\epsilon}:=\{(\tau_2,b^a+\mathrm{i}t^a,\zeta^i,\widetilde{\zeta}_i,\sigma)\in \overline{N} \; | \; \tau_2>\epsilon, \;\; t^a>K, \; a=1,...,n\}\,.
\end{equation}
 As before, we lift the metric $g_{\overline{N}}$ to the subset $\widetilde{N}_{K,\epsilon}\to N_{K,\epsilon}$ where the periodic coordinates are made non-periodic.\\

Now notice that since $S^4=\text{Id}$, the open subset 
\begin{equation}\label{Sinvariantint}
    \widetilde{N}_S:=\widetilde{N}_{K,\epsilon}\cap S\cdot \widetilde{N}_{K,\epsilon} \cap S^2\cdot \widetilde{N}_{K,\epsilon}\cap S^3\cdot \widetilde{N}_{K,\epsilon}
\end{equation}
is $S$-invariant. To see that it is also non-empty, notice that the points (in Type IIB coordinates) of the form $(\tau_1+\I\tau_2,b^a+\I t^a, c^a,c_a,c_0,\psi)=(0+\I, 0 +\I t^a, 0 ,0,0,0)$ are $S$-fixed. In particular, for $t^a>K$ we have $(0+\I, 0 +\I t^a, 0 ,0,0,0)\in \widetilde{N}_{\epsilon,K}$ and, since these points are $S$-fixed, they must lie on $\widetilde{N}_S$. Since $S$ acts on $\widetilde{N}_S$, we can apply Theorem \ref{theorem2} in the case of $A=S\in \mathrm{SL}(2,\mathbb{Z})$ to concluce that $(\widetilde{N}_S,g_{\overline{N}})$ carries a $\mathbb{Z}_4$-action by isometries generated by $S$.\\

Finally, the fact that $\widetilde{N}_S$ is given by \eqref{Sinvariantsub} follows immediately from the defining relations of $\widetilde{N}_{K,\epsilon}$ together with \eqref{Sinvariantint} and the S-duality transformations \eqref{sl2can} which imply that $S^2\cdot  \widetilde{N}_{K,\epsilon}=\widetilde{N}_{K,\epsilon}$.
\end{proof}

\section{Universal isometries of instanton corrected q-map spaces and S-duality}\label{universalisosec}

We start by recalling a certain universal group of isometries in the case of a tree-level q-map space (recall Definition \ref{rmapqmapdef}). 
In \cite[Theorem 3.17]{CTSduality} it was shown that a tree level q-map space of real dimension $4n+4$ with $n>0$ has a universal (i.e. independent of the PSR manifold) $3n+6$ dimensional connected Lie group of isometries $G$, whose Lie algebra $\mathfrak{g}$ has the form

\begin{equation}\label{uniisotree}
    \mathfrak{g}=\mathbb{R}\ltimes (\mathfrak{sl}(2,\mathbb{R})\ltimes (\mathbb{R}^n\ltimes \mathfrak{h}_{2n+2}))\,.
\end{equation}
The first factor corresponds to an action by dilations of the group $\mathbb{R}_{>0}$; the second to the S-duality $\mathrm{SL}(2,\mathbb{R})$-action; the third to an action by $\mathbb{R}^n$ which, among other things, shifts the real part of the PSK coordinates $z^a$; and the last factor to the action of a certain codimension $1$ subgroup $H_{2n+2}$ of the Heisberg group $\text{Heis}_{2n+3}(\mathbb{R})$ (to be defined below). In order to state this more precisely, we consider the coordinates $(\rho^{\text{cl}},z^a,\zeta^i,\widetilde{\zeta}_i^{\text{cl}},\sigma^{\text{cl}})$ on $\overline{\mathcal{N}}_{\text{IIA}}^{\text{cl}}:=\mathbb{R}_{>0}\times \overline{M}^{\text{cl}}\times \mathbb{R}^{2n+2}\times \mathbb{R}$, related to the type IIB coordinates $(\tau,b^a+\mathrm{i}t^a,c^a,c_a,c_0,\psi)$ on $\overline{\mathcal{N}}_{\text{IIB}}^{\text{cl}}$ (where the latter was given in Definition \ref{defsl2domain}) via the classical Mirror map (i.e. by \eqref{MM} with $c_{\ell}=n_{\hat{\gamma}}=\chi=0$). The groups then act as follows:

\begin{itemize}
    \item The multiplicative $\mathbb{R}_{>0}$ group acts by a dilation action on the $(\rho^{\text{cl}},\zeta^i,\widetilde{\zeta}_i^{\text{cl}},\sigma^{\text{cl}})$ variables via 
    \begin{equation}\label{scalingaction}
        r\cdot (\rho^{\text{cl}},\zeta^i,\widetilde{\zeta}_i^{\text{cl}},\sigma^{\text{cl}})=(r\rho^{\text{cl}}, \sqrt{r}\zeta^i,\sqrt{r}\widetilde{\zeta}_i^{\text{cl}},r\sigma^{\text{cl}})\,.
    \end{equation}
    \item The $\mathrm{SL}(2,\mathbb{R})$ factor corresponds to the S-duality action given in \eqref{sl2can}. We have $\mathrm{SL}(2,\mathbb{R})$ instead of $\mathrm{SL}(2,\mathbb{Z})$ due to the absence of quantum corrections.
    \item The vector $v=(v^a)\in \mathbb{R}^n$ acts via
    \begin{equation}\label{Rnaction}
    v\cdot \begin{pmatrix}
        z^a\\
        \rho^{\text{cl}} \\
        \zeta^0 \\
        \zeta^a\\
        \widetilde{\zeta}_0^{\text{cl}}\\
        \widetilde{\zeta}_a^{\text{cl}}\\
         \sigma^{\text{cl}} \\
        \end{pmatrix}=\begin{pmatrix}
        z^a+v^a\\
        \rho^{\text{cl}} \\
        \zeta^0 \\
        \zeta^a +\zeta^0v^a\\
        \widetilde{\zeta}_0^{\text{cl}}+\frac{1}{6}k_{abc}v^av^bv^c\zeta^0 + \frac{1}{2}k_{abc}v^av^b\zeta^c -\widetilde{\zeta}_a^{\text{cl}}v^a\\
        \widetilde{\zeta}_a ^{\text{cl}}-\frac{1}{2}\zeta^0k_{abc}v^bv^c - k_{abc}v^b\zeta^c\\
       \sigma^{\text{cl}} \\
        \end{pmatrix}\,,
    \end{equation}
    where we recall that $k_{abc}$ are the coefficients of the cubic polynomial \eqref{cubicpsr} defining the PSR manifold. 
    
    \item If $\text{Heis}_{2n+3}(\mathbb{R})\cong \mathbb{R}^{2n+3}$ denotes the Heisenberg group, then $(\eta^i,\widetilde{\eta}_i,\kappa)\in \text{Heis}_{2n+3}(\mathbb{R})$, $i=0,1,...,n$; acts on $(\zeta^i,\widetilde{\zeta}_i^{\text{cl}},\sigma^{\text{cl}})$ via
    
    \begin{equation}\label{heisact}
        (\eta^i,\widetilde{\eta}_i,\kappa)\cdot (\zeta^i,\widetilde{\zeta}_i^{\text{cl}},\sigma^{\text{cl}})=(\zeta^i+\eta^i, \widetilde{\zeta}_i^{\text{cl}}+\widetilde{\eta}_i,  \sigma^{\text{cl}} +\kappa +\widetilde{\zeta}_i^{\text{cl}}\eta^i-\zeta^i\widetilde{\eta}_i)\,.
    \end{equation}
    On the other hand, $H_{2n+2}\subset \text{Heis}_{2n+3}(\mathbb{R})$ is the codimension $1$ subgroup given by
    \begin{equation}\label{codim1heis}
        H_{2n+2}:=\{(\eta^i,\widetilde{\eta}_i,\kappa) \in \text{Heis}_{2n+3}(\mathbb{R}) \; | \; \eta^0=0\}\,.
    \end{equation}
    The transformations shifting $\zeta^0$ missing from $H_{2n+2}$ are already included in the $\mathrm{SL}(2,\mathbb{R})$ transformations. 
\end{itemize}
On the other hand, the semi-direct product of Lie algebras $\mathbb{R}^n\ltimes \mathfrak{h}_{2n+2}$ in \eqref{uniisotree} corresponds at the group level to the semi-direct product $\mathbb{R}^n\ltimes_{\varphi} H_{2n+2}\subset \mathbb{R}^n\ltimes_{\varphi} \text{Heis}_{2n+3}(\mathbb{R})$, where the automorphism $\varphi:\mathbb{R}^n\to \text{Aut}(\text{Heis}_{2n+3}(\mathbb{R}))$ is given by a similar formula to \eqref{Rnaction}, namely 
\begin{equation}\label{autheis}
    \varphi(v)\cdot \begin{pmatrix}
        
        \eta^0 \\
        \eta^a\\
        \widetilde{\eta}_0\\
        \widetilde{\eta}_a\\
         \kappa \\
        \end{pmatrix}=\begin{pmatrix}
        \eta^0 \\
        \eta^a +\eta^0v^a\\
        \widetilde{\eta}_0+\frac{1}{6}k_{abc}v^av^bv^c\eta^0 + \frac{1}{2}k_{abc}v^av^b\eta^c -\widetilde{\eta}_av^a\\
        \widetilde{\eta}_a -\frac{1}{2}\eta^0k_{abc}v^bv^c - k_{abc}v^b\eta^c\\
       \kappa \\
        \end{pmatrix}, \quad v\in \mathbb{R}^n, \quad (\eta^i,\widetilde{\eta}_i,\kappa)\in \text{Heis}_{2n+3}(\mathbb{R})\,.
    \end{equation}
The other semi-direct products of Lie algebras in \eqref{uniisotree} can be described via the Lie algebra structure given in \cite[Proposition 3.10]{CTSduality}.\\

We now consider $(M,\mathfrak{F})$ and a mutually local variation of BPS structures $(M,\Gamma,Z,\Omega)$ as in Section \ref{settingsec}, so that we obtain an instanton corrected q-map space $(\widetilde{N},g_{\overline{N}})$. We want to study how to instanton corrections affect the $3n+6$-dimensional isometry group $G$ from the tree-level q-map space case. In the following, unless otherwise specified, we assume that $\widetilde{N}$ is the (lift of the) maximal domain of definition of $g_{\overline{N}}$ obtained via HK/QK correspondence from the HK metric $(N,g_{N})$ from Section \ref{HK/QKsec}. We also define the following subgroups of $\text{Heis}_{2n+3}(\mathbb{R})$:
    
    \begin{equation}\label{heisinstgroup}
    \begin{split}
        \text{Heis}_{2n+3,D}&:=\{ (\eta^i,\widetilde{\eta}_i,\kappa)\in \text{Heis}_{2n+3}(\mathbb{R}) \; |\quad \; \text{$\eta^i \in \mathbb{Z}$ for $i=0,...,n$}\}\\
        H_{2n+2,D}&:=\{ (\eta^a,\widetilde{\eta}_i,\kappa)\in H_{2n+2} \; |\quad \; \text{$\eta^a \in \mathbb{Z}$ for $a=1,...,n$}\}\,.
    \end{split}
    \end{equation}
The letter $D$ in the above notation is meant to emphasize that the directions $\eta^i$ are broken to a discrete subgroup due to the inclusion of (part of) the ``D-instanton corrections" due to terms involving the BPS indices $\Omega(\gamma)$. \\

We begin by studying how $\text{Heis}_{2n+3,D}\subset \text{Heis}_{2n+3}(\mathbb{R})$ and $\mathbb{Z}^n\subset \mathbb{R}^n$ act on the corrected coordinates $(\rho,z^a,\zeta^i,\widetilde{\zeta}_i,\sigma)$, related to the type IIB coordinates via the quantum corrected mirror map \eqref{MM}. We use the notation $\rho^{\text{w.s.}}:=f^{\text{w.s.}}/16\pi$, so that $\rho=\rho^{\text{cl}}+\rho^{\text{w.s.}}-c_{\ell}$ (recall \eqref{rhodef} and \eqref{fsplitting2}).

\begin{lemma}\label{lemmainstcoord}
$\text{Heis}_{2n+3,D}$ acts on the functions $(\rho^{\text{w.s.}},\widetilde{\zeta}_i^{\text{inst}},\sigma^{\text{inst}})$ by

\begin{equation}\label{heisinst}
    (\eta^i,\widetilde{\eta}_i,\kappa)\cdot (\rho^{\text{w.s.}},\widetilde{\zeta}_i^{\text{inst}},\sigma^{\text{inst}})=(\rho^{\text{w.s.}},\widetilde{\zeta}_i^{\text{inst}},\sigma^{\text{inst}}+\widetilde{\zeta}_i^{\text{inst}}\eta^i)\,.
\end{equation}
On the other hand, $\mathbb{Z}^n\subset \mathbb{R}^n$ acts on the functions $(\rho^{\text{w.s.}},\widetilde{\zeta}_i^{\text{inst}},\sigma^{\text{inst}})$ by
\begin{equation}\label{zninst}
    (v^a)\cdot (\rho^{\text{w.s.}},\widetilde{\zeta}_0^{\text{inst}},\widetilde{\zeta}_a^{\text{inst}},\sigma^{\text{inst}})=(\rho^{\text{w.s.}},\widetilde{\zeta}_0^{\text{inst}}-v_a\widetilde{\zeta}_a^{\text{inst}},\widetilde{\zeta}_a^{\text{inst}},\sigma^{\text{inst}})
\end{equation}
\end{lemma}
\begin{proof}
The first statement \eqref{heisinst} follows easily from \eqref{typeiiacoordcor}, \eqref{fsplitting} and \eqref{heisact}. On the other hand note that \eqref{Rnaction} and the fact that $\rho^{\text{cl}}=\frac{\tau_2^2}{2}h(t)$ imply that $\tau_2$ is invariant under the action of $\mathbb{Z}^n$. Equation \eqref{zninst} then follows again easily from \eqref{typeiiacoordcor} and \eqref{fsplitting}. The restrictions to $\text{Heis}_{2n+3,D}\subset \text{Heis}_{2n+3}(\mathbb{R})$ and $\mathbb{Z}^n\subset \mathbb{R}^n$ are required, since in the computation we use the periodicity of the complex exponential function. 
\end{proof}
\begin{corollary}\label{coordcorollary}
The action of $\text{Heis}_{2n+3,D}\subset \text{Heis}_{2n+3}(\mathbb{R})$ and $\mathbb{Z}^n\subset \mathbb{R}^n$ on $(\rho,z^a,\zeta^i,\widetilde{\zeta}_i,\sigma)\in \overline{\mathcal{N}}_{\text{IIA}}=\mathbb{R}_{>-c_{\ell}}\times \overline{M}\times \mathbb{R}^{2n+2}\times \mathbb{R}$ has the same transformation rules, \eqref{heisact} and \eqref{Rnaction}, as their action on $(\rho^{\text{cl}},z^a,\zeta^i,\widetilde{\zeta}_i^{\text{cl}},\sigma^{\text{cl}})\in \overline{\mathcal{N}}_{\text{IIA}}^{\text{cl}}$. Namely, for $v\in \mathbb{Z}^n$ and $(\eta^i,\widetilde{\eta}_i,\kappa)\in \text{Heis}_{2n+3,D}$, 

\begin{equation}\label{ZnHeisact}
    v\cdot \begin{pmatrix}
        z^a\\
        \rho \\
        \zeta^0 \\
        \zeta^a\\
        \widetilde{\zeta}_0\\
        \widetilde{\zeta}_a\\
         \sigma \\
        \end{pmatrix}=\begin{pmatrix}
        z^a+v^a\\
        \rho \\
        \zeta^0 \\
        \zeta^a +\zeta^0v^a\\
        \widetilde{\zeta}_0+\frac{1}{6}k_{abc}v^av^bv^c\zeta^0 + \frac{1}{2}k_{abc}v^av^b\zeta^c -\widetilde{\zeta}_a^{\text{cl}}v^a\\
        \widetilde{\zeta}_a -\frac{1}{2}\zeta^0k_{abc}v^bv^c - k_{abc}v^b\zeta^c\\
       \sigma \\
        \end{pmatrix}\,, \quad (\eta^i,\widetilde{\eta}_i,\kappa)\cdot \begin{pmatrix}z^a\\
        \rho\\
        \zeta^i\\
        \widetilde{\zeta}_i\\
        \sigma\end{pmatrix}=\begin{pmatrix}z^a\\
        \rho\\
        \zeta^i+\eta^i\\ \widetilde{\zeta}_i+\widetilde{\eta}_i\\ 
        \sigma +\kappa +\widetilde{\zeta}_i\eta^i-\zeta^i\widetilde{\eta}_i\\
        \end{pmatrix}\,.
    \end{equation} 
    
    Furthermore, the action of $\text{Heis}_{2n+3,D}$ and $\mathbb{Z}^n$ on $(z^a,\rho,\zeta^i, \widetilde{\zeta}_i,\sigma)$ expressed in terms of type IIB coordinates via the quantum corrected mirror map \eqref{MM} coincides with the action \eqref{heisact} and \eqref{Rnaction} expressed in type IIB variables via the classical mirror map. 
\end{corollary}
\begin{proof}
The first statement follow from Lemma \ref{lemmainstcoord} together with the actions \eqref{heisact} and \eqref{Rnaction}. The last statement follows immediately from \eqref{ZnHeisact}, together with \eqref{heisact}, \eqref{Rnaction} and \eqref{MM}.
\end{proof}
\begin{remark}\leavevmode
\begin{itemize}
\item We remark that since the actions \eqref{heisact} leave the $t^a=\text{Im}(z^a)$ coordinates invariant, via the quantum corrected mirror map we get an action of $\mathbb{Z}^n$ and $\text{Heis}_{2n+3,D}$ not only on  $\overline{\mathcal{N}}_{\text{IIB}}=\mathcal{M}^{-1}(\overline{\mathcal{N}}_{\text{IIA}})$, but also the bigger manifold $\overline{\mathcal{N}}_{\text{IIB}}^{\text{cl}}\supset \overline{\mathcal{N}}_{\text{IIB}}$ where $\mathrm{SL}(2,\mathbb{Z})$ always acts via \eqref{sl2can}.
\item It is not hard to the check that $(\rho,z^a,\zeta^i,\widetilde{\zeta}_i,\sigma)$ do not have any nice transformation property under the $\mathbb{R}_{>0}$ scaling action \eqref{scalingaction}.
\end{itemize}

\end{remark}

\begin{proposition}\label{heisisoprop}
The group $\text{Heis}_{2n+3,D}$ acts by isometries on any instanton corrected c-map space $(\widetilde{N},g_{\overline{N}})$ (recall Definition \ref{liftdef}). Furthermore, if $n_{\hat{\gamma}}=0$ for all $\hat{\gamma}\in \Lambda^{+}$ (but with $\chi$ possibly non-zero), then $(\widetilde{N},g_{\overline{N}})$ carries an action by isometries of $\{(\eta^i,\widetilde{\eta}_i,\kappa)\in \text{Heis}_{2n+3}(\mathbb{R})\; | \; \eta^0\in \mathbb{Z}\}$, and in particular by $H_{2n+2}$.
\end{proposition}

\begin{proof}
By Corollary \ref{coordcorollary}, the group $\text{Heis}_{2n+3,D}$ acts on the $(\zeta^i,\widetilde{\zeta}_i,\sigma)$ part of the type IIA coordinates $(\rho, z^a,\zeta^i,\widetilde{\zeta}_i,\sigma)$ via \eqref{ZnHeisact}. Because the instanton corrections are periodic in $\zeta^i$ and not depend on $\widetilde{\zeta}_i$ and $\sigma$, it follows immediately that the tensor $T$ in \eqref{non-deg} as well as the sets $N'$ from \eqref{n'def} and $\overline{N}$ from \eqref{QKCASKdomain} are invariant under the action of $\text{Heis}_{2n+3,D}$. In particular, it follows that $\widetilde{N}$ must carry an action of $\text{Heis}_{2n+3,D}$. On the other hand, by the explicit formula \eqref{coordQKmetric2} for $g_{\overline{N}}$ and using again the fact that all instanton correction terms are periodic in $\zeta^i$ and independent from $\widetilde{\zeta}_i$ and $\sigma$, it can be easily checked that  the action \eqref{heisact} of $\text{Heis}_{2n+3,D}$ acts by isometries on $(\widetilde{N},g_{\overline{N}})$.\\

The last statement follows immediately from the previous argument, together with the fact that if $n_{\hat{\gamma}}=0$ for all $\hat{\gamma}\in \Lambda^+$, then the instanton correction terms only depend on $\zeta^0$.
\end{proof}

We now show the following proposition, dealing with the $\mathbb{R}^n$-factor of \eqref{uniisotree} in the case of instanton corrected q-map spaces.

\begin{proposition}\label{Rnisotheorem}
Let $(\widetilde{N},g_{\overline{N}})$ be an instanton corrected q-map space. Then $(\widetilde{N},g_{\overline{N}})$ carries an action by isometries by $\mathbb{Z}^n\subset \mathbb{R}^n$ given by \eqref{ZnHeisact}. Furthermore, if $n_{\hat{\gamma}}=0$ for all $\hat{\gamma}\in \Lambda^{+}$ (but with $\chi$ possibly non-zero), then $(\widetilde{N},g_{\overline{N}})$ carries an action by isometries by the full $\mathbb{R}^n$.
\end{proposition}

\begin{proof}
Notice that with respect to the coordinates $(Z^0,z^a)=(Z^0,Z^a/Z^0)$ of $M$, we can write the K\"{a}hler potential $k$ of the CASK manifold as \eqref{KpotCASK}. In particular, we see that $k(Z^0,b^a+\I t^a)$ is invariant under integer shifts of the $b^a$, and hence, $M$ must be invariant under the action of $\mathbb{Z}^n$, and it must be an isometry of the CASK metric.\\

On the other hand, when considering the associated instanton corrected HK geometry, one must consider the tensor $T$ given in \eqref{non-deg}. In our particular case, we can rewrite $T$ as follows
\begin{equation}\label{Tqmap}
\begin{split}
    T&=-\text{Im}(\tau_{ij})\D Z^i \D\overline{Z}^j +\frac{\chi}{2\pi}\sum_{q_0\in \mathbb{Z}-\{0\}}\sum_{n>0}e^{-2\pi \mathrm{i}nq_0\zeta^0}K_0(4\pi R n|q_0|)|q_0\D Z^0|^2 \\
    &\quad \quad -\frac{1}{2\pi}\sum_{\hat{\gamma}\in \Lambda^+}n_{\hat{\gamma}}\sum_{q_0\in \mathbb{Z}}\sum_{n>0}e^{-2\pi \mathrm{i}n(q_a\zeta^a +q_0\zeta^0)}K_0(4\pi R n|q_0+q_az^a|)|Z^0q_a\D z^a+(q_0+q_az^a)\D Z^0|^2\,.
\end{split}
\end{equation}
From the above formula together with \eqref{Rnaction}, it follows that the instanton part of $T$ in \eqref{Tqmap} is invariant under the action of $\mathbb{Z}^n$. Indeed, each summand in the term proportional to $\chi$ is invariant, while in the last term one finds that for each fixed $\hat{\gamma}=q_a\gamma^a\in \Lambda^{+}$, we have $q_0\to q_0+q_av^a$, which remains invariant since we have a sum over all $q_0\in \mathbb{Z}$.  Since $\mathbb{Z}^n$ also acts by isometries on the CASK metric, it follows that $T$ is invariant under $\mathbb{Z}^n$, and hence the maximal domain of definition $N$ of the HK metric is invariant under $\mathbb{Z}^n$. Similarly, the conditions defining the subsets $N'$ in \eqref{n'def} and $\overline{N}$ in \eqref{QKCASKdomain} required in the construction of the instanton corrected QK manifold via HK/QK correspondence are seen to be $\mathbb{Z}^n$-invariant, so that $\widetilde{N}$ carries an action of $\mathbb{Z}^n$. To show that the action of $\mathbb{Z}^n$ is by isometries, we show that one can lift the action to the twistor space acting by twistor space automorphisms. We define the lift by declaring that it acts trivially on the $\mathbb{C}P^1$ fibers of $\mathcal{Z}\cong \widetilde{N}\times \mathbb{C}P^1$. By using the explicit formulas for $(\xi^i,\widetilde{\xi}_i^{\text{cl}},\alpha^{\text{cl}})$ in \eqref{dccl} one finds that they have the following transformation rule under the action of $\mathbb{Z}^n$ (or even $\mathbb{R}^n$):

\begin{equation}
    v\cdot\begin{pmatrix}
    \xi^0\\
    \xi^a\\
    \widetilde{\xi}_0^{\text{cl}}\\
    \widetilde{\xi}_a^{\text{cl}}\\
    -\frac{1}{2}(\alpha^{\text{cl}}-\xi^i\widetilde{\xi}_i^{\text{cl}})\\
    \end{pmatrix}=\begin{pmatrix}
    \xi^0\\
    \xi^a+\xi^0v^a\\
    \widetilde{\xi}_0^{\text{cl}}+\frac{1}{6}k_{abc}v^av^bv^c\xi^0 + \frac{1}{2}k_{abc}v^av^b\xi^c -\widetilde{\xi}_a^{\text{cl}}v^a\\
    \widetilde{\xi}_a^{\text{cl}} -\frac{1}{2}\xi^0k_{abc}v^bv^c-k_{abc}v^b\xi^c\\
    -\frac{1}{2}(\alpha^{\text{cl}}-\xi^i\widetilde{\xi}_i^{\text{cl}})-\frac{1}{6}(\xi^0)^2k_{abc}v^av^bv^c -\frac{1}{2}k_{abc}v^av^b\xi^c\xi^0  -\frac{1}{2}k_{abc}v^b\xi^a\xi^c\\
    \end{pmatrix}
\end{equation}
where we have used $-\frac{1}{2}(\alpha^{\text{cl}}-\xi^i,\widetilde{\xi}_i^{\text{cl}})$ instead of $\alpha^{\text{cl}}$ because we want to relate to the Darboux coordinates \eqref{typeiibdc}. On the other hand, using \eqref{typeiibdc} and \eqref{Rnaction} one finds that under the action of $\mathbb{Z}^n$ the following holds

\begin{equation}
    v\cdot\begin{pmatrix}
    \widetilde{\xi}_0^{\text{inst}}\\
    \widetilde{\xi}_a^{\text{inst}}\\
    \alpha^{\text{inst}}\\
    \end{pmatrix}=\begin{pmatrix}
    \widetilde{\xi}_0^{\text{inst}} -\widetilde{\xi}_a^{\text{inst}}v^a\\
    \widetilde{\xi}_a^{\text{inst}} \\
    \alpha^{\text{inst}}\\
    \end{pmatrix}\,.
\end{equation}
Hence, joining everything together one finds that  $(\xi^i,\widetilde{\xi}_i,\alpha)$ from \eqref{typeiibdc} transforms under the $\mathbb{Z}^{n}$ action by
\begin{equation}\label{vatwistor}
    v\cdot\begin{pmatrix}
    \xi^0\\
    \xi^a\\
    \widetilde{\xi}_0\\
    \widetilde{\xi}_a\\
    \alpha\\
    \end{pmatrix}=\begin{pmatrix}
    \xi^0\\
    \xi^a+\xi^0v^a\\
    \widetilde{\xi}_0+\frac{1}{6}k_{abc}v^av^bv^c\xi^0 + \frac{1}{2}k_{abc}v^av^b\xi^c -\widetilde{\xi}_av^a\\
    \widetilde{\xi}_a -\frac{1}{2}\xi^0k_{abc}v^bv^c-k_{abc}v^b\xi^c\\
    \alpha-\frac{1}{6}(\xi^0)^2k_{abc}v^av^bv^c -\frac{1}{2}k_{abc}v^av^b\xi^c\xi^0  -\frac{1}{2}k_{abc}v^b\xi^a\xi^c\,.
    \end{pmatrix}
\end{equation}
In particular, the same holds for the Darboux coordinates $(\xi^i,\widetilde{\xi}_{i},\alpha_{c_{\ell}})$ obtained in Corollary \ref{theorem1cor}, since $t$ is invariant under the lift of the $\mathbb{Z}^n$ action.  From the transformation rules \eqref{vatwistor} it follows that
\begin{equation}
   \D\alpha_{c_{\ell}} -\widetilde{\xi}_i\D \xi^i
\end{equation}
is invariant under the action of $\mathbb{Z}^n$ and acts holomorphically on a dense open subset of $\mathcal{Z}$ (here we use that Darboux coordinates for a holomorphic contact structure must be holomorphic coordinates). Since the action is globally defined and smooth, it follows that $\mathbb{Z}^n$ must act holomorphically globally on $\mathcal{Z}$ and preserve the contact structure.  The real structure is also trivially preserved, since the action on the fibers is trivial. Hence, we conclude that $\mathbb{Z}^n$ acts by twistor space automorphisms, and hence by isometries on $(\widetilde{N},g_{\overline{N}})$.\\

Finally, the case where $n_{\hat{\gamma}}=0$ for all $\hat{\gamma}\in \Lambda^{+}$ follows by a simplified version of the previous argument. The restriction of translations to $\mathbb{Z}^n$ is no longer required since $\mathfrak{F}^{\text{w.s}}$ does not have polylogarithm terms in this case, and the instanton corrections due to $\Omega(\gamma)$ only depends on the $\tau_2$ and $\zeta^0=\tau_1$ variables, which are left invariant under the full $\mathbb{R}^n$-action \eqref{Rnaction}.
\end{proof}

Before joining the previous results and the ones from Section \ref{Sdualitysec} together, we will need the following lemma:

\begin{lemma}\label{collecttheoremlemma}
The group $\mathrm{SL}(2,\mathbb{R})\ltimes (\mathbb{R}^n\ltimes H_{2n+2})$ combining the S-duality action with the action of $\mathbb{R}^n\ltimes H_{2n+2}$ on $\overline{\mathcal{N}}_{\text{IIB}}^{\text{cl}}$ has $\mathrm{SL}(2,\mathbb{Z})\ltimes (\mathbb{Z}^n\ltimes H_{2n+2,D})$ as a subgroup. 

\end{lemma}
\begin{proof}
The fact that the $\mathrm{SL}(2,\mathbb{R})$ S-duality action and the $\mathbb{R}^n\ltimes H_{2n+2}$ action combine into an action of a group of the form $\mathrm{SL}(2,\mathbb{R})\ltimes (\mathbb{R}^n\ltimes H_{2n+2})$ follows from \cite[Proposition 3.10]{CTSduality}. On the other hand, to check that  
\begin{equation}
    \mathrm{SL}(2,\mathbb{Z})\ltimes (\mathbb{Z}^n\ltimes H_{2n+2,D})\subset \mathrm{SL}(2,\mathbb{R})\ltimes (\mathbb{R}^n\ltimes H_{2n+2})
\end{equation}
defines a subgroup, it is enough to check that the action by automorphisms of $\mathbb{Z}^n$ on $H_{2n+2}$ preserves the integrality contraint of $H_{2n+2,D}\subset H_{2n+2}$, and that the action by automorphisms of $\mathrm{SL}(2,\mathbb{Z})$ on $\mathbb{R}^n\ltimes H_{2n+2}$ preserves the integrality constraint of $\mathbb{Z}^n\ltimes H_{2n+2,D}\subset \mathbb{R}^n\ltimes H_{2n+2}$. The fact that the action by automorphisms of $\mathbb{Z}^n$ preserves the integrality constraint follows immediately from \eqref{autheis}. On the other hand, since $\mathrm{SL}(2,\mathbb{Z})$ is generated by $T$ and $S$ given in \eqref{sl2gen}, it is enough to check that the induced automorphism $\varphi_T, \varphi_S\in \text{Aut}(\mathbb{R}^n\ltimes H_{2n+2})$ preserve the integrality constraint of $\mathbb{Z}^n\ltimes H_{2n+2,D}$. From the fact that in the group $\mathrm{SL}(2,\mathbb{R})\ltimes (\mathbb{R}^n\ltimes H_{2n+2})$ we have that
\begin{equation}
    (A,0,0)\cdot (1,v^a,(\eta^a,\eta_i,\kappa))\cdot(A^{-1},0,0)=(1,\varphi_A(v^a,(\eta^a,\eta_i,\kappa))) 
\end{equation}
for $A\in \mathrm{SL}(2,\mathbb{R}), \; (v^a)\in \mathbb{R}^n, \; (\eta_a,\eta_i,\kappa)\in H_{2n+2}$, it is straightforward to use the actions of $\mathrm{SL}(2,\mathbb{R})$ and $\mathbb{R}\ltimes H_{2n+2}$ to compute that
\begin{equation}
    \varphi_T(v^a,(\eta^a,\eta_i,\kappa))=(v^a,(\eta^a-v^a,...)), \quad \varphi_S(v^a,(\eta^a,\eta_i,\kappa))=(-\eta^a,(v^a,...))\,.
\end{equation}
It then follows that the integrality constraints of $\mathbb{Z}^n\ltimes H_{2n+2,D}$ are preserved by the automorphisms of $\mathrm{SL}(2,\mathbb{Z})$, and hence we obtain the desired result.  
\end{proof}

We can now state our main theorem:
\begin{theorem}\label{collectiontheorem}
Consider an instanton corrected q-map space $(\widetilde{N},g_{\overline{N}})$ of dimension $4n+4$, where are before, we take $\widetilde{N}$ to be the maximal domain of definition of $g_{\overline{N}}$. Let $T,S\in \mathrm{SL}(2,\mathbb{Z})$ be as in \eqref{sl2gen}. Then:

\begin{itemize}
\item The group 

\begin{equation}\label{breakiso1}
    \langle T \rangle \ltimes (\mathbb{Z}^n \ltimes H_{2n+2,D})
\end{equation}
acts by isometries on $(\widetilde{N},g_{\overline{N}})$, 
where $\langle T \rangle \cong \mathbb{Z}$ is the subgroup generated by $T$. 
\item Assume that we take the one-loop parameter to be $c_{\ell}=\frac{\chi}{192\pi}$. We can always find a non-empty open subset $\widetilde{N}_S\subset \widetilde{N}$ where $(\widetilde{N}_S,g_{\overline{N}})$ carries an isometry group the the form 
 \begin{equation}
      \langle S \rangle \ltimes (\mathbb{Z}^n \ltimes H_{2n+2,D}),
 \end{equation}
 where $\langle S \rangle \cong \mathbb{Z}/4\mathbb{Z}$ is the group generated by $S\in \mathrm{SL}(2,\mathbb{Z})$. Furthermore, if $\widetilde{N}_{\mathrm{SL}(2,\mathbb{Z})}\subset \widetilde{N}$ is  $\mathrm{SL}(2,\mathbb{Z})$-invariant,  then $\mathrm{SL}(2,\mathbb{Z})$ acts by isometries on $(\widetilde{N}_{\mathrm{SL}(2,\mathbb{Z})},g_{\overline{N}})$. In particular, if $\widetilde{N}$ is already invariant under $\mathrm{SL}(2,\mathbb{Z})$ then \eqref{breakiso1} can be enhanced to
\begin{equation}
    \mathrm{SL}(2,\mathbb{Z})\ltimes (\mathbb{Z}^n\ltimes H_{2n+2,D})\,.
\end{equation}
\item Finally, if $n_{\hat{\gamma}}=0$ for all $\hat{\gamma}\in \Lambda^{+}$, then in the previous statements we can replace $\mathbb{Z}^n$ and $H_{2n+2,D}$ by $\mathbb{R}^n$ and $H_{2n+2}$, respectively. If furthermore we take  $\chi=c_{\ell}=0$ and $n_{\hat{\gamma}}=0$ for all $\hat{\gamma}\in \Lambda^{+}$, then we return to the tree-level q-map space case, where there is a connected $3n+6$ dimensional Lie group $G$ acting by isometries on $(\widetilde{N},g_{\overline{N}})$, see \cite[Theorem 3.17]{CTSduality}. The group $G$ in particular contains the S-duality action by $\mathrm{SL}(2,\mathbb{R})$, an action by $\mathbb{R}^{n}\ltimes H_{2n+2}$, and a dilation action by $\mathbb{R}_{>0}$.  \eqref{uniisotree}. 
\end{itemize}
\end{theorem}

\begin{proof}

The first statement of the theorem then follows from Lemma \ref{collecttheoremlemma}, Proposition \ref{Rnisotheorem},  Proposition \ref{heisisoprop}, and the fact that the action of $T$ and $H_{2n+2,D}$ generate the action of $\text{Heis}_{2n+3,D}$. To check the second statement, notice that by Theorem \ref{Z4prop}, there is a non-empty open S-invariant subset $\widetilde{N}_S\subset \widetilde{N}$ where $S$ acts by isometries on $(\widetilde{N}_S,g_{\overline{N}})$. This subset is characterized in Theorem \ref{Z4prop} by 
\begin{equation}\label{breakiso2}
    \widetilde{N}_S=\{(\rho,z^a,\zeta^i,\widetilde{\zeta}_i,\sigma)\in \widetilde{N}\; | \; \epsilon<\tau_2, \;\; \epsilon< \frac{\tau_2}{|\tau_1|^2+|\tau_2^2|},\;\; t^a>K, \;\; |\tau|t^a>K\}\,.
\end{equation}
In particular, using that the action of $\mathbb{Z}^n$ and $H_{2n+2,D}$ leave $\tau_1$, $\tau_2$ and $t^a$ invariant, it follows that $\widetilde{N}_S$ carries an action by both groups. By the same proof as in Proposition \ref{heisisoprop} and Proposition \ref{Rnisotheorem}, it follows that $\mathbb{Z}^n$ and $H_{2n+2,D}$ must act by isometries on $(\widetilde{N}_S,g_{\overline{N}})$. The result then follows from Lemma \ref{collecttheoremlemma}. The last part of the second point and the final statement follow from Theorem \ref{theorem2} and the aforementioned propositions and lemma.
\end{proof}
\section{An example of full S-duality}\label{examplesec}
As stated in Theorem \ref{collectiontheorem}, in order to guarantee that S-duality acts by isometries on the instanton corrected q-map metric, one needs to guarantee that the domain of the metric carries an action by the S-duality $\mathrm{SL}(2,\mathbb{Z})$. Here we explicitly give an example where the full  $\mathrm{SL}(2,\mathbb{Z})$ action by isometries can be guaranteed. We start by specifying the initial data as in Section \ref{settingsec}.
\begin{itemize}
\item The PSR manifold $(\mathcal{H},g_{\mathcal{H}})$ is specified by the cubic polynomial $h:\mathbb{R}\to \mathbb{R}$ given by
\begin{equation}
    h(t)=\frac{t^3}{6}\,.
\end{equation}
In particular $\mathcal{H}$ just reduces to a point. The corresponding PSK manifold $(\overline{M}^{\text{cl}},g_{\overline{M}^{\text{cl}}})$ obtained via the r-map has domain 
\begin{equation}
    \overline{M}^{\text{cl}}=\mathbb{R}+\I \mathbb{R}_{>0}\mathcal{H}=\mathbb{R} +\I\mathbb{R}_{>0}\,,
\end{equation}
with the corresponding CASK domain $(M^{\text{cl}},\mathfrak{F}^{\text{cl}})$ given by
\begin{equation}
    M^{\text{cl}}=\{(Z^0,Z^1)=Z^0(1,z^1) \in \mathbb{C}^{2} \; | \; Z^0\in \mathbb{C}^{\times}, \; z^1 \in \overline{M}^{\text{cl}}\}, \quad \mathfrak{F}^{\text{cl}}=-\frac{1}{6}\frac{(Z^1)^3}{Z^0}\,.
\end{equation}
Furthermore we have $\Lambda^{+}=\text{span}_{\mathbb{Z}_{>0}}\{\gamma^1\}$ with the prepotential $\mathfrak{F}$ given by
\begin{equation}\label{prepotentialex}
    \mathfrak{F}=-\frac{1}{6}\frac{(Z^1)^3}{Z^0}+\chi\frac{(Z^0)^2\zeta(3)}{2(2\pi \I )^3}\,, \quad \chi\in \mathbb{Z}_{>0}.
\end{equation}
Notice that we are restricting $\chi$ to be positive and we take $n_{\hat{\gamma}}=0$ for all $\hat{\gamma}\in \Lambda^{+}$. 

\item Since $n_{\hat{\gamma}}=0$ for all $\hat{\gamma}\in \Lambda^{+}$, we find that $M^{\text{cl}}=M^{q}$, where $M^{q}$ was defined in \eqref{defM}. $M$ is then the maximal open subset of $M^{\text{cl}}$ where we have that $\text{Im}(\tau)$ has signature $(1,1)$ and $\text{Im}(\tau_{ij})Z^i\overline{Z}^j<0$. For our simple example, we can describe $M$ explicitly. Setting $z^1=b+\mathrm{i}t$, we have  \begin{equation}
\text{Im}(\tau_{00})=\frac{t^3}{3}-b^2t+\frac{\chi \zeta(3)}{(2\pi)^3}, \quad \text{Im}(\tau_{01})=bt, \quad
    \text{Im}(\tau_{11})=-t,\,
\end{equation}
so that
\begin{equation}
    \text{Det}(\text{Im}(\tau))=-\left(\frac{t^4}{3}+\frac{t\chi\zeta(3)}{(2\pi)^3}\right), \quad \text{Im}(\tau_{ij})Z^i\overline{Z}^j=|Z^0|^2\left(-\frac{2t^3}{3}+\frac{\chi\zeta(3)}{(2\pi)^3}\right)
\end{equation}
Since $\chi>0$ and $\text{Im}(\tau)$ is a $2\times 2$ matrix, $M$ is then given by
\begin{equation}\label{Mex}
    M=\{(Z^0,Z^1)=Z^0(1,z^1)\in \mathbb{C}^2 \; | \; Z^0\in \mathbb{C}^{\times}, \;\; z^1\in \overline{M}^{\text{cl}}, \;\; \frac{t^3}{3}> \frac{\chi\zeta(3)}{2(2\pi)^3}\}\,.
\end{equation}
The prepotential $\mathfrak F$ induces on $M$ the structure of a CASK manifold
fibering over the complete PSK manifold $\overline M =\{ z^1\in \overline{M}^{\text{cl}} \mid \frac{t^3}{3}> \frac{\chi\zeta(3)}{2(2\pi)^3}\}$ with the K\"ahler potential $-\log (h(t) - \frac{\chi\zeta (3)}{4(2\pi )^3})$, as follows from the general results of \cite{CDM}. The completeness 
follows from \cite[Theorem~6.2]{CDM} due to our assumption $\chi>0$. For $\chi<0$ the PSK metric is incomplete \cite[Remark~5.7]{CDM}.
    \item Over $M$, we consider  the trivial local system $\Gamma\to M$ with global sections $(\gamma^0,\gamma^1,\widetilde{\gamma}_0,\widetilde{\gamma}_1)$ defined by the CASK structure (recall Section \ref{QKinstdomainsec}), together with the canonical central charge $Z:M\to \Gamma^*\otimes \mathbb{C}$ given by $Z_{\gamma^i}=Z^i$, $Z_{\widetilde{\gamma}_i}=-\frac{\partial{\mathfrak{F}}}{\partial Z^i}$ for $i=0,1$. The BPS indices are given by 
     \begin{equation}\label{bpsindexex}
        \begin{cases}
        \Omega(q_0\gamma^0)=-\chi, \quad q_0\neq 0 \\
    \Omega(\gamma)=0 \quad \text{else},\\
        \end{cases}
    \end{equation}
    which has the required structure from Section \ref{settingsec}.
\end{itemize}
We would like to now study the tensor $T$ from \eqref{non-deg} determining the domain of definition $N$ of the instanton corrected HK structure, as well as the functions $f$, $f_3$ and $g_{N}(V,V)$ from Section \ref{HK/QKsec}, determining the domain of the instanton corrected QK metric, as well as its signature. \\

Recalling that $Z^1/Z^0=z^1=b+\mathrm{i}t$, $|Z^0|=\tau_2$, $\zeta^0=\tau_1$, and $\tau=\tau_1+\mathrm{i}\tau_2$, we consider the open set $N\subset M^{\text{cl}}\times \mathbb{R}^{4}$ defined by
\begin{equation}
    N:=\{(Z^0,Z^1,\zeta^0,\zeta^1,\widetilde{\zeta}_0,\widetilde{\zeta}_1)\in M^{\text{cl}}\times \mathbb{R}^4 \; | \; R_1(t,\tau)>0,\;\; R_2(t,\tau)>0\;\}
\end{equation}
where 
\begin{equation}
    \begin{split}
        R_1(t,\tau)&:=\frac{t^3}{3}-\frac{\chi}{4(2\pi)^3}\sum_{(m,n)\in \mathbb{Z}^2-(0,0)}\frac{1}{|m\tau+n|^3}\\
        R_2(t,\tau)&:=\frac{t^3}{3} -\frac{3\chi}{4(2\pi)^3}\sum_{(m,n)\in \mathbb{Z}^2-(0,0)}\frac{1}{|m\tau+n|^3}-\frac{(3\chi)^2}{(4\pi)^6}\left(\sum_{(m,n)\in \mathbb{Z}^2}\frac{1}{|m\tau+n|^3}\right)^2\left(R_1\right)^{-1}\,.
    \end{split}
\end{equation}
The origin of these expressions will become clear with the following results. They will guarantee that the domain of the resulting QK metric carries an action of the S-duality $\mathrm{SL}(2,\mathbb{Z})$ and that the metric is positive definite. 
Furthermore, we remark that $N$ is non-empty, since for a fixed $\tau$ the inequalities can be clearly achieved for $t$ sufficiently big. Furthermore, we remark that $R_1(t,\tau)>0$ already implies that $(Z^0,Z^1)\in M$, since (using that $\chi>0$)
\begin{equation}
    \frac{\chi}{4(2\pi)^3}\sum_{(m,n)\in \mathbb{Z}^2-(0,0)}\frac{1}{|m\tau+n|^3}>\frac{\chi}{4(2\pi)^3}\sum_{n\in \mathbb{Z}-\{0\}}\frac{1}{|n|^3}=\frac{\chi\zeta(3)}{2(2\pi)^3}\;.
\end{equation}

\begin{proposition}\label{exprop1}
The instanton corrected HK metric associated to the previous data as in Section \ref{instHKsec} is well defined on $N$ and of signature $(4,4)$.
\end{proposition}
\begin{proof}
The tensor determining the domain of definition of the HK geometry given in \eqref{non-deg} (compare \eqref{CASKconvention}) reduces in our case to: 
\begin{equation}
    T=-\text{Im}(\tau_{ij})\mathrm{d}Z^i\mathrm{d}\overline{Z}^j+\frac{\chi}{2\pi}\sum_{q_0\in \mathbb{Z}-\{0\}}\sum_{n>0}e^{-2\pi \mathrm{i}nq_0\zeta^0}q_0^2K_0(2\pi n\tau_2|q_0|)|\mathrm{d}Z^0|^2\,.
\end{equation}
We see that only the $T_{0\overline{0}}$ component of this tensor receives corrections due to BPS indices, while $T_{0\overline{1}}=T_{1\overline{0}}=-\text{Im}(\tau_{01})$ and $T_{1\overline{1}}=-\text{Im}(\tau_{11})$. The key thing is that we can Poisson resum $T_{0\overline{0}}$ using Lemma \ref{PoissonLemma} as follows

\begin{equation}
    \begin{split}
        T_{0\overline{0}}&=-\text{Im}(\tau_{00})-\frac{\chi}{4(2\pi)^3}\partial_{\zeta^0}^2\mathcal{I}_0^{(2)}\\
        &=-\frac{t^3}{3}+b^2t-\frac{\chi}{2(2\pi)^3}\sum_{n\neq 0}\frac{1}{n^2|n|}-\frac{\chi}{4(2\pi)^3}\partial_{\zeta^0}^2\mathcal{I}_0^{(2)}\\
        &=-\frac{t^3}{3}+b^2t -\frac{\chi}{4(2\pi)^3}\sum_{(m,n)\in \mathbb{Z}^2-(0,0)}\left(\frac{3(m\tau_1+n)^2}{|m\tau+n|^5}-\frac{1}{|m\tau+n|^3}\right)\,.
    \end{split}
\end{equation}
From this computation, we find that the determinant along the horizontal directions gives
\begin{equation}
    \text{Det}(T)=T_{0\overline{0}}T_{1\overline{1}}-(T_{0\overline{1}})^2=-\frac{t^4}{3} -\frac{t\chi}{4(2\pi)^3}\sum_{(m,n)\in \mathbb{Z}^2-(0,0)}\left(\frac{3(m\tau_1+n)^2}{|m\tau+n|^5}-\frac{1}{|m\tau+n|^3}\right)\,.
\end{equation}
Since $\chi>0$ and $t>0$, we therefore have  
\begin{equation}
    \frac{t\chi}{4(2\pi)^3}\sum_{(m,n)\in \mathbb{Z}^2-(0,0)}\frac{3(m\tau_1+n)^2}{|m\tau+n|^5}>0
\end{equation}
and hence on the points of $N$ we have
\begin{equation}\label{detin}
    \text{Det}(T)(p)<-\frac{t^4}{3}+\frac{t\chi}{4(2\pi)^3}\sum_{(m,n)\in \mathbb{Z}^2-(0,0)}\frac{1}{|m\tau+n|^3}=-t\cdot R_1(t,\tau)<0\,.
\end{equation}
This shows that the tensor $T$ is horizontally non-degenerate on $N$, and hence by Theorem \ref{HKinsttheorem} we have that the instanton corrected HK metric $g_{N}$ is well defined on $N$. On the other hand by \cite[Equation 3.38]{CT}, if the signature of the real matrix $(T_{i\overline{j}})$ is $(n,m)$ (called $M_{ij}$ in \cite{CT}), then the signature of $g_{N}$ is $(4n,4m)$. In our case $(T_{i\overline{j}})$ is a $2\times 2$ matrix, and since $\text{Det}(T)<0$, we must have that $(T_{i\overline{j}})$ has signature $(1,1)$. It then follows that the signature of $g_{N}$ is $(4,4)$. 
\end{proof}
\begin{proposition}\label{exprop2}
If $f$, $f_3$ and $g_N(V,V)$ are the functions on $N$ defined on Section \ref{HK/QKsec}, and the $1$-loop parameter is taken to be $c_{\ell}=\frac{\chi}{192\pi}$, then on $N$ we have 
\begin{equation}
    f>0, \quad f_3<0, \quad g_N(V,V)>0\,.
\end{equation}
\end{proposition}
\begin{proof}
Using \eqref{Poissonresumedf} we obtain the following expressions for $f$ in our case:
\begin{equation}\label{fin}
\begin{split}
f&=2\pi\tau_2^2\left(\frac{2t^3}{3}\right)-\frac{\tau_2^2\chi}{2(2\pi)^2}\sum_{(m,n)\in \mathbb{Z}^2-\{0\}}\frac{1}{|m\tau+n|^3}\\
&=4\pi \tau_2^2\left(\frac{t^3}{3}-\frac{\chi}{8(2\pi)^3}\sum_{(m,n)\in \mathbb{Z}^2-(0,0)}\frac{1}{|m\tau+n|^3}\right)\,\\
&>4\pi\tau_2^2R_1(t,\tau).
\end{split}
\end{equation}
So it follows that $f>0$ on $N$.\\

On the other hand, we remark that if we show that $f_3<0$ on $N$, then $g_N(V,V)>0$ follows due to the relation $g_N(V,V)=2(f-f_3)$ and the fact that we showed that $f>0$ on $N$.\\

 To study $f_3$ we need to study the expression $g_{N}(V,V)$. Using \cite[Equation 3.35]{CT} adapted to our conventions gives the following expression for $g_{N}$:
\begin{equation}
    g_N=2\pi \left(T_{i\overline{j}}\mathrm{d}Z^i\mathrm{d}\overline{Z}^j +(W_i+W_i^{\text{inst}})T^{i\overline{j}}(\overline{W}_j+\overline{W}_j^{\text{inst}})\right),
\end{equation}
where (using the notation $\gamma=q_i(\gamma)\gamma^i$)
\begin{equation}
    W_i=\mathrm{d}\zeta_{\widetilde{\gamma}_i}+\tau_{ij}\mathrm{d}\zeta_{\gamma^j}, \quad W_i^{\text{inst}}=-\sum_{\gamma}\Omega(\gamma)q_i(\gamma)\left(A_{\gamma}^{\text{inst}}-\mathrm{i}V_{\gamma}^{\text{inst}}\mathrm{d}\zeta_{\gamma}\right)\,.
\end{equation}

Using that $V=2\mathrm{i}(Z^i\partial_{Z^i}-\overline{Z}^i\partial{\overline{Z}^i})$ we can compute the evaluation of $g_{N}(V,V)$, which in our case gives
\begin{equation}\label{gnvv}
\begin{split}
g_{N}(V,V)=&8\pi T_{i\overline{j}}Z^i\overline{Z}^j+2\pi|W_0^{\text{inst}}(V)|^2T^{0\overline{0}}\,\\
&=8\pi \tau_2^2\left(\frac{2t^3}{3} -\frac{\chi}{4(2\pi)^3}\sum_{(m,n)\in \mathbb{Z}^2-(0,0)}\left(\frac{3(m\tau_1+n)^2}{|m\tau+n|^5}-\frac{1}{|m\tau+n|^3}\right)\right)+2\pi|W_0^{\text{inst}}(V)|^2T^{0\overline{0}}
\end{split}
\end{equation}
where (using \eqref{Idef} and Lemma \ref{PoissonLemma} for the expression of $W_0^{\text{inst}}$) 
\begin{equation}
\begin{split}
    T^{0\overline{0}}&=\frac{t}{\text{Det}(T)}<0,\\
     W_0^{\text{inst}}(V)&=-\frac{\mathrm{i}\chi}{\pi}\tau_2\sum_{q_0\neq 0}\sum_{n>0}e^{2\pi \mathrm{i}n\zeta_{q_0\gamma^0}}q_0|q_0|K_1(2\pi n \tau_2|q_0|)\\
     &=-\frac{\chi\tau_2}{2(2\pi)^3}\partial_{\tau_1}\partial_{\tau_2}\mathcal{I}_{0}^{(2)}\\
     &=-\frac{3\chi \tau_2^2}{2(2\pi)^3}\sum_{(m,n)\in \mathbb{Z}^2}\frac{(m\tau_1+n)m}{|m\tau+n|^5}\,.
    \end{split}
\end{equation}
In the following, it will be convenient to find a lower bound for the negative term $|W_0^{\text{inst}}(V)|^2T^{0\overline{0}}$. We have 
\begin{align*}
    2\pi|W_0^{\text{inst}}(V)|^2T^{0\overline{0}}&=2\pi\left|\frac{3\chi \tau_2^2}{2(2\pi)^3}\sum_{(m,n)\in \mathbb{Z}^2}\frac{(m\tau_1+n)m}{|m\tau+n|^5}\right|^2\frac{t}{\text{Det}(T)}\\
    &>2\pi\left(\frac{3\chi \tau_2}{2(2\pi)^3}\right)^2\left|\sum_{(m,n)\in \mathbb{Z}^2}\frac{(m\tau_1+n)m\tau_2}{|m\tau+n|^5}\right|^2\left(-R_1(t,\tau)\right)^{-1}\\
    &>2\pi\left(\frac{3\chi \tau_2}{2(2\pi)^3}\right)^2\left(\sum_{(m,n)\in \mathbb{Z}^2}\frac{|m\tau_1+n||m|\tau_2}{|m\tau+n|^5}\right)^2\left(-R_1(t,\tau)\right)^{-1}\\
    &>2\pi\left(\frac{3\chi \tau_2}{2(2\pi)^3}\right)^2\left(\sum_{(m,n)\in \mathbb{Z}^2}\frac{(m\tau_1+n)^2+(m\tau_2)^2}{2|m\tau+n|^5}\right)^2\left(-R_1(t,\tau)\right)^{-1}\\
    &=16\pi\tau_2^2\frac{(3\chi)^2}{2(4\pi)^6}\left(\sum_{(m,n)\in \mathbb{Z}^2}\frac{1}{|m\tau+n|^3}\right)^2\left(-R_1(t,\tau)\right)^{-1}\\ \numberthis \label{usefullowerbound}
\end{align*}
where in the first inequality we have used that $\text{Det}(T)<-t\cdot R_1<0$; in the second we used the triangle inequality, and in the last one we have used that $\sqrt{xy}\leq (x+y)/2$ for $x=(m\tau_1+n)^2$ and $y=(m\tau_2)^2$.\\

Using \eqref{usefullowerbound} and \eqref{gnvv}, we then find that $f_3=f-\frac{1}{2}g_N(V,V)$ gives
\begin{align*}\label{f3in}
    f_3&=-4\pi \tau_2^2\left(\frac{t^3}{3} -\frac{\chi}{4(2\pi)^3}\sum_{(m,n)\in \mathbb{Z}^2-(0,0)}\left(\frac{3(m\tau_1+n)^2}{|m\tau+n|^5}-\frac{3}{2|m\tau+n|^3}\right)\right)-\pi|W_0^{\text{inst}}(V)|^2T^{0\overline{0}}\\
    &<-4\pi \tau_2^2\left(\frac{t^3}{3} -\frac{3\chi}{4(2\pi)^3}\sum_{(m,n)\in \mathbb{Z}^2-(0,0)}\frac{(m\tau_1+n)^2}{|m\tau+n|^5}\right)-\pi|W_0^{\text{inst}}(V)|^2T^{0\overline{0}}\\
    &<-4\pi \tau_2^2\left(\frac{t^3}{3} -\frac{3\chi}{4(2\pi)^3}\sum_{(m,n)\in \mathbb{Z}^2-(0,0)}\frac{1}{|m\tau+n|^3}\right)-\pi|W_0^{\text{inst}}(V)|^2T^{0\overline{0}}\\
    &<-4\pi \tau_2^2\left(\frac{t^3}{3} -\frac{3\chi}{4(2\pi)^3}\sum_{(m,n)\in \mathbb{Z}^2-(0,0)}\frac{1}{|m\tau+n|^3}\right)+8\pi\tau_2^2\frac{(3\chi)^2}{2^7(2\pi)^6}\left(\sum_{(m,n)\in \mathbb{Z}^2}\frac{1}{|m\tau+n|^3}\right)^2\left(R_1(t,\tau)\right)^{-1}\\
    &=-4\pi\tau_2^2 R_2(t,\tau)<0\,. \numberthis
    \end{align*}
The required inequalities therefore hold on $N$.
\end{proof}
  We then obtain as a corollary of Propositions \ref{exprop1}, \ref{exprop2}, and Theorem \ref{QKCASKdomainformulas} the following:
\begin{corollary}\label{QKex}
The instanton corrected q-map metric $g_{\overline{N}}$ associated to the prepotential \eqref{prepotentialex}, BPS indices \eqref{bpsindexex} and with $1$-loop constant $c_{\ell}=\frac{\chi}{192\pi}$ is defined and positive definite on
\begin{equation}\label{tildeNex}
\widetilde{N}:=\{(\tau_2,b+\mathrm{i}t,\tau_1,\zeta^1,\widetilde{\zeta}_0,\widetilde{\zeta}_1,\sigma)\in \mathbb{R}_{>0}\times \overline{M}^{\text{cl}}\times \mathbb{R}^4\times \mathbb{R} \; | \; R_1(t,\tau)>0,\quad R_2(t,\tau)>0 \; \}.
\end{equation}
\end{corollary}
\begin{proof}
Because of Propositions \ref{exprop1} and \ref{exprop2}, we can take $N'=N$, where $N'$ is defined on \eqref{n'def}. It then follows from Theorem \ref{QKCASKdomainformulas} that $g_{\overline{N}}$ from \eqref{coordQKmetric2} is defined and positive definite on $\widetilde{N}$ from \eqref{tildeNex}.  
\end{proof}

The most important thing about $\widetilde{N}$ is the following:
\begin{proposition}\label{exprop3}
$\widetilde{N}$ carries an action of S-duality. 
\end{proposition}
\begin{proof}
Since the constraints in \eqref{tildeNex} are given only on $t$ and $\tau$, we focus only on this variables. Under the action of $S\in \mathrm{SL}(2,\mathbb{Z})$ we have 
\begin{equation}
    t\to |\tau|t, \quad \tau \to -1/\tau\,,
\end{equation}
while under the action of $T\in \mathrm{SL}(2,\mathbb{Z})$ we have
\begin{equation}
    t \to t, \quad \tau \to \tau+1\,.
\end{equation}
From the definition of $R_1$ and $R_2$, it follows immediately that they are invariant under $T$, and under $S$ they satisfy
\begin{equation}
    R_i(|\tau|t,-1/\tau)=|\tau|^3R_i(t,\tau), \quad i=1,2
\end{equation}
so we conclude that $\widetilde{N}$ is invariant under $T$ and $S$.
Since $S$ and $T$ generate $\mathrm{SL}(2,\mathbb{Z})$ it then follows that $\widetilde{N}$ is invariant under $S$-duality.
\end{proof}

\begin{corollary}\label{excor2}
The positive definite QK metric $(\widetilde{N},g_{\overline{N}})$ from Corollary \ref{QKex} has an effective action by isometries by the group $\mathrm{SL}(2,\mathbb{Z})\ltimes (\mathbb{R}\ltimes H_4)$.
\end{corollary}
\begin{proof}
By Proposition \ref{exprop3}, $\widetilde{N}$ carries an action by $\mathrm{SL}(2,\mathbb{Z})$. Furthermore, since the definition of $\widetilde{N}$ in \eqref{tildeNex} only constrains $\tau$ and $t$, and the action of $\mathbb{R}\ltimes H_4$ given in Section \ref{universalisosec} acts trivially on $\tau$ and $t$, $\widetilde{N}$ must carry an action of the group $\mathbb{R}\ltimes H_{4}$. It then follows from the Theorem \ref{collectiontheorem} that $(\widetilde{N},g_{\overline{N}})$ carries and action by isometries by the groups $\mathrm{SL}(2,\mathbb{Z})$ and $\mathbb{R}\ltimes H_4$ and hence by $\mathrm{SL}(2,\mathbb{Z})\ltimes (\mathbb{R}\ltimes H_4)$. We remark that we get $\mathbb{R}\ltimes H_4$ instead of $\mathbb{Z}\ltimes H_{4,D}$ since in our case we have $n_{\hat{\gamma}}=0$ for all $\hat{\gamma}$.\\

\end{proof}
\begin{remark}\label{incompleterem}
\begin{itemize}
\item Notice that the QK metric from Corollary \ref{QKex} is incomplete. This follows from the fact that the functions $R_i(t,\tau)$ defining $\widetilde{N}$ in \eqref{tildeNex} are obtained from strict inequalities involving the conditions (see \eqref{detin},  \eqref{fin}, \eqref{f3in})
    \begin{equation}\label{collin}
    \text{Det}(T)<0,\quad  f>0,\quad f_3<0\,.
    \end{equation} This means that $g_{\overline{N}}$ extends to a positive definite metric on a bigger manifold containing \eqref{tildeNex}, and hence cannot be complete. It might be that \eqref{collin} is preserved by the S-duality action, but (except for $f>0$) this is not obvious from the formulas and the $\mathrm{SL}(2,\mathbb{Z})$ action.
    \item We also note that one can find a smaller S-duality invariant open subset $\widetilde{N}'\subset \widetilde{N}$ defined by a single relation $R(t,\tau)>0$. Namely, if we set
    \begin{equation}\label{singlerel}
        R(t,\tau):=\frac{t^3}{3}-\frac{9\chi}{8(2\pi)^3}\sum_{(m,n)\in \mathbb{Z}^2-(0,0)}\frac{1}{|m\tau+n|^3}
    \end{equation}
    and define
    \begin{equation}
        s(\tau):=\frac{\chi}{4(2\pi)^3}\sum_{(m,n)\in \mathbb{Z}^2-(0,0)}\frac{1}{|m\tau+n|^3},\,
    \end{equation}
    then we find that $R(t,\tau)>0$ implies that
    \begin{equation}
        R_1(t,\tau)=\frac{t^3}{3}-s(\tau)>\frac{t^3}{3}-\frac{9}{2}s(\tau)=R(t,\tau)>0,\,
    \end{equation}
    and furthermore 
    \begin{equation}
    \begin{split}
        R_1(t,\tau)R_2(t,\tau)&=\left(\frac{t^3}{3}-3s(\tau)\right)R_1  -\frac{(3\chi)^2}{(4\pi)^6}\left(\sum_{(m,n)\in \mathbb{Z}^2}\frac{1}{|m\tau+n|^3}\right)^2\\
        &>\left(\frac{t^3}{3}-3s(\tau)\right)^2-\frac{(3\chi)^2}{(4\pi)^6}\left(\sum_{(m,n)\in \mathbb{Z}^2}\frac{1}{|m\tau+n|^3}\right)^2\\
        &=\left(\frac{t^3}{3}-\frac{9}{2}s(\tau)\right)\left(\frac{t^3}{3}-\frac{3}{2}s(\tau)\right)\\
        &>R^2(t,\tau)>0\,.
    \end{split}
    \end{equation}
    It then follows that $R(t,\tau)>0$ also implies $R_2(t,\tau)>0$.
\end{itemize}
\end{remark}
\begin{theorem}\label{finitevol}
 Let $(\widetilde{N},g_{\overline{N}})$ be as in Corollary \ref{QKex}. Then: 
\begin{itemize}
    \item There is a free and properly discontinuous action by isometries of a discrete group of the form $\mathrm{SL}(2,\mathbb{Z})'\ltimes \Lambda$, where $\Lambda \subset \mathbb{R}\ltimes H_4$ is a lattice,  $\mathrm{SL}(2,\mathbb{Z})' \subset \mathrm{SL}(2,\mathbb{Z})$ is a finite index subgroup and the QK manifold $(\widetilde{N}/(\mathrm{SL}(2,\mathbb{Z})'\ltimes \Lambda),g_{\overline{N}})$ has finite volume. 
    \item Furthermore, there is a submanifold with boundary $\hat{N}\subset \widetilde{N}$ where $\mathrm{SL}(2,\mathbb{Z})'\ltimes \Lambda$ acts and the quotient $(\hat{N}/(\mathrm{SL}(2,\mathbb{Z})'\ltimes \Lambda),g_{\overline{N}})$ gives a complete QK manifold with boundary and of finite volume.   
\end{itemize} 
\end{theorem}

\begin{proof}
By using the same argument as in \cite[Theorem 3.21]{CTSduality}, we can find a lattice $\Lambda\subset \mathbb{R}\ltimes H_{4}$ such that $\mathrm{SL}(2,\mathbb{Z})\ltimes \Lambda$ is a lattice of the group of isometries $\mathrm{SL}(2,\mathbb{Z})\ltimes (\mathbb{R}\ltimes H_{4})$ (recall Corollary \ref{excor2}).
On the other hand, by using the $\mathrm{SL}(2,\mathbb{Z})$ action by isometries, we can assume that $\tau$ lies in the usual fundamental domain $\mathcal{F}_{\mathbb{H}}$ of the upper half place $\mathbb{H}\subset \mathbb{C}$ given by
\begin{equation}    \mathcal{F}_{\mathbb{H}}:=\{\tau \in \mathbb{H} \; | \; |\tau|>1, \;\; |\tau_1|<1/2\}\cup\{\tau \in \mathbb{H} \; | \; |\tau|\geq 1, \;\; \tau_1=-\frac{1}{2}\}\cup \{\tau \in \mathbb{H} \; | \; |\tau|=1, \;\; -\frac{1}{2}<\tau_1\leq 0\}\,.
\end{equation}
If we define 
\begin{equation}
    F:=\{(\tau_2,b+\mathrm{i} t,\tau_1,\zeta^1,\widetilde{\zeta}_0,\widetilde{\zeta}_1,\sigma)\in \widetilde{N} \; | \; \tau \in \mathcal{F}_{\mathbb{H}}\}
\end{equation}
we can then think of the quotient $\widetilde{N}/(\mathrm{SL}(2,\mathbb{Z})\ltimes \Lambda)$ as
\begin{equation}
\widetilde{N}/(\mathrm{SL}(2,\mathbb{Z})\ltimes \Lambda)=F/\Lambda\,.
\end{equation}
Note that $\Lambda$ actually acts on $F$, since $F$ is defined by restrictions on $\tau$ and $t$ and the latter are left invariant by the action of $\mathbb{R}\ltimes H_4$.
On $F/\Lambda$ the coordinates $b$, $\tau_1$, $\zeta^1$, $\widetilde{\zeta}_0$, $\widetilde{\zeta}_1$ and $\sigma$ are periodic. Furthermore, $\tau_2$ is bounded below by some positive constant (in fact, by $\sqrt{3}/2>0$), and the same holds for $t$, since the relation $R_1(t,\tau)>0$ defining $\widetilde{N}$ implies 
\begin{equation}\label{tbound}
    \frac{t^3}{3}>\frac{\chi}{4(2\pi)^3}\sum_{(m,n)\in \mathbb{Z}^2-(0,0)}\frac{1}{|m\tau+n|^3}>\frac{\chi}{4(2\pi)^3}\sum_{n\in \mathbb{Z}-\{0\}}\frac{1}{|n|^3}=\frac{\chi\zeta(3)}{2(2\pi)^3}>0\;.
\end{equation}
On the other hand, $F$ has a boundary where $R_1(t,\tau)=0$ or $R_2(t,\tau)=0$. Since we know that the metric admits an extension beyond this boundary (see Remark \ref{incompleterem}), it cannot cause $F/\Lambda$ to have infinite volume. Furthermore,  on the end of $F/\Lambda$, where $\tau_2,t\to \infty$, the quantum correction terms in $g_{\overline{N}}$ are dominated by the tree-level q-map metric $(\overline{\mathcal{N}}_{\text{IIB}}^{\text{cl}},g^{\text{cl}})$ associated to the PSR manifold $\mathcal{H}=\{\text{point}\}$ (defined by $h(t)=t^3/6$). Namely, in terms of the components of the metrics $g_{\overline{N}}$ and $g^{\text{cl}}$ we have $g_{\overline{N},ij}\sim g_{ij}^{\text{cl}}$ on the end of $F/\Lambda$ \footnote{This can be checked by looking at the $t$ and $\tau_2$ dependence of the quantum corrections. In our example, we only have the analog of a perturbative world-sheet correction encoded in the $\chi$ term of $\mathfrak{F}$, the analog of the perturbative 1-loop correction encoded in $c_{\ell}=\chi/192\pi$, and the analog of D(-1) instanton corrections encoded in terms with $\Omega(q_0\gamma^0)=-\chi$. The D(-1) corrections are exponentially suppressed as $\tau_2\to \infty$, while the pertubative corrections can be safely ignored when $\tau_{2},t\to \infty$, since these are dominated by the classical terms of the tree-level q-map metric.}. We can therefore use the volume density of $g^{\text{cl}}$ to asymptotically approximate the volume density of $g_{\overline{N}}$ on this end. But by \cite[Theorem 3.21]{CTSduality}, we know that $\overline{\mathcal{N}}_{\text{IIB}}^{\text{cl}}/(\mathrm{SL}(2,\mathbb{Z})\ltimes \Lambda)$ has finite volume with respect to $g_{\text{cl}}$, so that $\widetilde{N}/(\mathrm{SL}(2,\mathbb{Z})\ltimes \Lambda)\subset \overline{\mathcal{N}}_{\text{IIB}}^{\text{cl}}/(\mathrm{SL}(2,\mathbb{Z})\ltimes \Lambda)$ also has finite volume with respect to $g^{\text{cl}}$. It follows that the end of $F/\Lambda$ where $\tau_2,t\to \infty$ must give a finite volume contribution with respect to $g_{\overline{N}}$. Hence, $(\widetilde{N}/(\mathrm{SL}(2,\mathbb{Z})\ltimes \Lambda),g_{\overline{N}})$ has finite volume.\\

The action of $\mathrm{SL}(2,\mathbb{Z})\ltimes \Lambda$ on $\widetilde{N}$ is easily seen to be properly discontinuous, however it is not free, since the points satisfying $\tau_2=\mathrm{i}$, $b=\tau_1=c^1=c_0=\psi=0$ are fixed by the subgroup $\langle S \rangle=\mathbb{Z}/4\mathbb{Z}\subset \mathrm{SL}(2,\mathbb{Z})$. We can nevertheless pick a finite index subgroup $\mathrm{SL}(2,\mathbb{Z})'\subset \mathrm{SL}(2,\mathbb{Z})$ that acts freely by choosing a finite index subgroup that intersects $\langle S \rangle$ only at the identity (for example by picking the kernel of the homomorphism $\mathrm{SL}(2,\mathbb{Z})\to \mathrm{SL}(2,\mathbb{Z}/N\mathbb{Z})$ reducing the entries modulo $N$ for $N\geq 3$). We therefore get a free, properly discontinuous action of $\mathrm{SL}(2,\mathbb{Z})'\ltimes \Lambda$ on $\widetilde{N}$. Since $(\widetilde{N}/(\mathrm{SL}(2,\mathbb{Z})\ltimes \Lambda),g_{\overline{N}})$ has finite volume and $\mathrm{SL}(2,\mathbb{Z})'\subset \mathrm{SL}(2,\mathbb{Z})$ is a finite index subgroup, it follows that $(\widetilde{N}/(\mathrm{SL}(2,\mathbb{Z})'\ltimes \Lambda),g_{\overline{N}})$ is a QK manifold of finite volume. \\

We now prove the last statement. We define $\hat{N}$ by
\begin{equation}\label{hatNex}
\hat{N}:=\{(\tau_2,b+\mathrm{i} t,\tau_1,\zeta^1,\widetilde{\zeta}_0,\widetilde{\zeta}_1,\sigma)\in \widetilde{N} \; | \; R(t,\tau)\geq 0\; \},
\end{equation}
where $\widetilde{N}$ is as in Corollary \ref{QKex} and $R(t,\tau)$ was defined in \eqref{singlerel}. It is easy to check that the function $R(t,\tau)$ is regular on the level set $R(t,\tau)=0$ (for example, by using that $t>0$ on $\widetilde{N}$, due to \eqref{tbound}), so that $\hat{N}$ is a smooth manifold with boundary. \\

 By the same argument of Proposition \ref{exprop3} and Corollary \ref{excor2} it follows that the action of $\mathrm{SL}(2,\mathbb{Z})\ltimes (\mathbb{R}\ltimes H_4)$ on $\widetilde{N}$ restricts to $\hat{N}$. Furthermore, since $(\widetilde{N}/(\mathrm{SL}(2,\mathbb{Z})'\ltimes \Lambda),g_{\overline{N}})$ has finite volume,  $(\hat{N}/(\mathrm{SL}(2,\mathbb{Z})\ltimes \Lambda),g_{\overline{N}})$ also has finite volume. Finally, to show that $(\hat{N}/(\mathrm{SL}(2,\mathbb{Z})\ltimes \Lambda),g_{\overline{N}})$ is complete, recall that a Riemannian manifold with boundary is complete if it is a complete metric space with the induced distance function. This follows from the fact that it contains the boundary points satisfying $R(t,\tau)=0$; the coordinates $b$, $\tau_1$, $\zeta^1$, $\widetilde{\zeta}_0$, $\widetilde{\zeta}_1$, $\sigma$ are periodic; and the metric in the end where $\tau_2,t\to \infty$ is complete, since it can be approximated by the complete tree-level q-map metric $g^{\text{cl}}$ (the completeness of $g^{\text{cl}}$ follows from \cite[Theorem 27]{CDS}).
\end{proof}

We finish with the following remark about a related example associated to the resolved conifold.
\begin{remark}\label{finalremark}
 To the resolved conifold one can associate a natural holomorphic prepotential of the required form \eqref{prepotential},
    where
    \begin{equation}
        \mathfrak{F}=-\frac{1}{6}\frac{(Z^1)^3}{Z^0}+\chi\frac{(Z^0)^2\zeta(3)}{2(2\pi \I )^3}-\frac{(Z^0)^2}{(2\pi \I)^3}n_{\gamma^1}\mathrm{Li}_3(e^{2\pi \I Z^1/Z^0})\,, \quad \chi=2, \quad n_{\gamma^1}=1\,.
    \end{equation}
    This form of the prepotential can be motivated by considering a certain extension of the Picard-Fuchs operators, see \cite{PFext}. On the other hand, the BPS spectrum of the resolved conifold also has the required form \eqref{varBPS}, with 
    \begin{equation}
        \begin{cases}
     \Omega(q_0\gamma^0)=-\chi=-2, \quad q_0\in \mathbb{Z}-\{0\}\\
     \Omega(q_0\gamma^0\pm \gamma^1)=\Omega(\pm \gamma^1)=n_{\gamma^1}=1 \quad \text{for $\gamma^1 \in \Lambda^+$, $q_0\in \mathbb{Z}$}\\
    \Omega(\gamma)=0 \quad \text{else}.\\
        \end{cases}
    \end{equation}
   We can therefore apply the construction of Section \ref{Sdualitysec} and obtain an instanton corrected q-map space.
    We expect that a similar argument as in the example can be done in order to find an S-duality invariant open subset of the domain of definition of the associated instanton corrected q-map space. As in Theorem \ref{finitevol}, we also expect that it admits a quotient of finite volume. Furthermore, we remark that since the resolved conifold is a non-compact Calabi-Yau 3-fold without compact divisors, the possible quantum corrections in the type IIB string theory language simplify to just the world-sheet corrections and the 1-loop correction $c_\ell$ together with the D(-1) and D1-instanton corrections. Since these are all accounted for in the above construction, one could expect that that the resulting instanton corrected q-map space with its maximal domain of definition might be complete. We leave the study of this question for future work.
\end{remark}
\appendix

\section{Integral identities in terms of Bessel functions}\label{Besselappendix}

In this appendix we will prove the following lemma, relating certain integrals with infinite sums of Bessel functions. We recall that given $\gamma=q_i\gamma^i$ and the functions $\xi^i$ from \eqref{Darbouxcoords1loop}, we use the notation from Section \ref{typeiiacoordssec}, where $\xi_{\gamma}=q_i\xi^i=q_i(\zeta^i-\mathrm{i}R(t^{-1}z^i+t\overline{z}^i))=-\zeta_{\gamma}-\mathrm{i}R(t^{-1}\widetilde{Z}_{\gamma}+t\overline{\widetilde{Z}_{\gamma}})$.

\begin{lemma}
Consider the modified Bessel functions of the second kind $K_{\nu}:\mathbb{R}_{>0}\to \mathbb{R}$, which have the following integral representation

\begin{equation}\label{Besselintrep}
    K_{\nu}(x)=\int_{0}^{\infty}dt \exp(-x\cosh(t))\cosh(\nu t), \quad \nu=0,1,2,...
\end{equation}
Following the notation from Section \ref{typeiiacoordssec}, we have the following identities for $\gamma=q_i\gamma^i$
\begin{equation}\label{Bessel1}
    \int_{l_{\gamma}}\frac{d\zeta}{\zeta^{1-q}}\frac{\exp(2\pi\I \xi_{\gamma})}{1-\exp(2\pi\I \xi_{\gamma})}=
    \begin{cases}
        2\frac{|\widetilde{Z}_{\gamma}|^2}{\widetilde{Z}_{\gamma}^2}\sum_{n>0}e^{-2\pi\I n\zeta_{\gamma}}\left(K_0(4\pi Rn|\widetilde{Z}_{\gamma}|)+\frac{K_1(4\pi Rn|\widetilde{Z}_{\gamma}|)}{2\pi Rn|\widetilde{Z}_{\gamma}|}\right), \quad q=-2\\
        -2\frac{|\widetilde{Z}_{\gamma}|}{\widetilde{Z}_{\gamma}}\sum_{n>0}e^{-2\pi\I n\zeta_{\gamma}}K_1(4\pi Rn|\widetilde{Z}_{\gamma}|), \quad q=-1\\
        2\sum_{n>0}e^{-2\pi\I n\zeta_{\gamma}}K_0(4\pi Rn|\widetilde{Z}_{\gamma}|), \quad q=0\\
        -2\frac{\widetilde{Z}_{\gamma}}{|\widetilde{Z}_{\gamma}|}\sum_{n>0}e^{-2\pi\I n\zeta_{\gamma}}K_1(4\pi Rn|\widetilde{Z}_{\gamma}|),\quad q=1\,\\
         2\frac{\widetilde{Z}_{\gamma}^2}{|\widetilde{Z}_{\gamma}|^2}\sum_{n>0}e^{-2\pi\I n\zeta_{\gamma}}\left(K_0(4\pi Rn|\widetilde{Z}_{\gamma}|)+\frac{K_1(4\pi Rn|\widetilde{Z}_{\gamma}|)}{2\pi Rn|\widetilde{Z}_{\gamma}|}\right), \quad q=2\\
        \end{cases}
\end{equation}
and 
\begin{equation}\label{Bessel2}
        \int_{l_{\gamma}}\frac{d\zeta}{\zeta^{1-q}}\log(1-\exp(2\pi\I \xi_{\gamma}))=\begin{cases}
        2\frac{|\widetilde{Z}_{\gamma}|}{\widetilde{Z}_{\gamma}}\sum_{n>0}\frac{e^{-2\pi\I n\zeta_{\gamma}}}{n}K_1(4\pi Rn|\widetilde{Z}_{\gamma}|), \quad q=-1\\
        -2\sum_{n>0}\frac{e^{-2\pi\I n\zeta_{\gamma}}}{n}K_0(4\pi Rn|\widetilde{Z}_{\gamma}|), \quad q=0\\
        2\frac{\widetilde{Z}_{\gamma}}{|\widetilde{Z}_{\gamma}|}\sum_{n>0}\frac{e^{-2\pi\I n\zeta_{\gamma}}}{n}K_1(4\pi Rn|\widetilde{Z}_{\gamma}|),\quad q=1\,.
        \end{cases}
\end{equation}
\end{lemma}
\begin{proof}
We start by computing the integrals of the form

\begin{equation}
    \int_{l_{\gamma}}\frac{d\zeta}{\zeta^{1-q}}\frac{\exp(2\pi\I \xi_{\gamma})}{1-\exp(2\pi\I \xi_{\gamma})}, \quad q=-2,-1,0,1,2\,,
\end{equation}
where
\begin{equation}
    l_{\gamma}=\{t \; | \; \widetilde{Z}_{\gamma}/t\in \mathbb{R}_{<0}\}\,.
\end{equation}

Setting $\zeta=-s\widetilde{Z}_{\gamma}/|\widetilde{Z}_{\gamma}|$ for $s\in (0,\infty)$, and using that $|\exp(2\pi\I \xi_{\gamma})|_{l_{\gamma}}|<1$, we have

    \begin{align*}
    \int_{l_{\gamma}}\frac{d\zeta}{\zeta^{1-q}}\frac{\exp(2\pi\I \xi_{\gamma})}{1-\exp(2\pi\I \xi_{\gamma})}&=(-1)^{q}\frac{\widetilde{Z}_{\gamma}^q}{|\widetilde{Z}_{\gamma}|^q}\int_{0}^{\infty}\frac{ds}{s^{1-q}}\frac{\exp(-2\pi R|\widetilde{Z}_{\gamma}|(s^{-1}+s)-2\pi\I \zeta_{\gamma})}{1-\exp(-2\pi R|\widetilde{Z}_{\gamma}|(s^{-1}+s)-2\pi\I \zeta_{\gamma})}\\
    &=(-1)^{q}\frac{\widetilde{Z}_{\gamma}^q}{|\widetilde{Z}_{\gamma}|^q}\sum_{n>0}e^{-2\pi\I n\zeta_{\gamma}}\int_{0}^{\infty}\frac{ds}{s^{1-q}}\exp(-2\pi Rn|\widetilde{Z}_{\gamma}|(s^{-1}+s))\\
    &=(-1)^{q}\frac{\widetilde{Z}_{\gamma}^q}{|\widetilde{Z}_{\gamma}|^q}\sum_{n>0}e^{-2\pi\I n\zeta_{\gamma}}\int_{-\infty}^{\infty}dxe^{qx}\exp(-4\pi Rn|\widetilde{Z}_{\gamma}|\cosh(x))\\
    &=2(-1)^{q}\frac{\widetilde{Z}_{\gamma}^q}{|\widetilde{Z}_{\gamma}|^q}\sum_{n>0}e^{-2\pi\I n\zeta_{\gamma}}\int_{0}^{\infty}dx\exp(-4\pi Rn|\widetilde{Z}_{\gamma}|\cosh(x))\cosh(qx)\\
    &=2(-1)^{q}\frac{\widetilde{Z}_{\gamma}^q}{|\widetilde{Z}_{\gamma}|^q}\sum_{n>0}e^{-2\pi\I n\zeta_{\gamma}}\int_{0}^{\infty}dx\exp(-4\pi Rn|\widetilde{Z}_{\gamma}|\cosh(x))\cosh(|q|x)\\
    &=2(-1)^{q}\frac{\widetilde{Z}_{\gamma}^q}{|\widetilde{Z}_{\gamma}|^q}\sum_{n>0}e^{-2\pi\I n\zeta_{\gamma}}K_{|q|}(4\pi Rn|\widetilde{Z}_{\gamma}|)\numberthis\\
    \end{align*}

Hence, we obtain \eqref{Bessel1}, where for the case $q=\pm 2$ we have used the identity $K_2(x)=K_0(x)+2K_1(x)/x$.\\

We can similarly compute integrals of the form 

\begin{equation}
    \int_{l_{\gamma}}\frac{d\zeta}{\zeta^{1-q}}\log(1-\exp(2\pi\I \xi_{\gamma})), \quad q=-1,0,1
\end{equation}
 Indeed since $|\exp(\mathrm{2\pi\I }\xi_{\gamma})|_{l_{\gamma}}|<1$, we can expand $\log(1-x)=-\sum_{n>0}x^n/n$ and obtain

\begin{equation}
    \begin{split}
        \int_{l_{\gamma}}\frac{d\zeta}{\zeta^{1-q}}\log(1-\exp(2\pi\I \xi_{\gamma}))&=-\sum_{n>0}\frac{e^{-2\pi\I n\zeta_{\gamma}}}{n}\int_{l_{\gamma}}\frac{d\zeta}{\zeta^{1-q}}\exp(2\pi nR\widetilde{Z}_{\gamma}/\zeta +nR\overline{Z}_{\gamma}\zeta)\\
        &=(-1)^{q+1}\frac{\widetilde{Z}_{\gamma}^q}{|\widetilde{Z}_{\gamma}|^q}\sum_{n>0}\frac{e^{-2\pi\I n\zeta_{\gamma}}}{n}\int_{0}^{\infty}\frac{ds}{s^{1-q}}\exp(-2\pi Rn|\widetilde{Z}_{\gamma}|(s^{-1}+s))\\
        &=2(-1)^{q+1}\frac{\widetilde{Z}_{\gamma}^q}{|\widetilde{Z}_{\gamma}|^q}\sum_{n>0}\frac{e^{-2\pi\I n\zeta_{\gamma}}}{n}K_{|q|}(4\pi Rn|\widetilde{Z}_{\gamma}|)\\
    \end{split}
\end{equation}
so that we obtain \eqref{Bessel2}.
\end{proof}
\section{Type IIA Darboux coordinates for instanton corrected c-map spaces}\label{appendixtypeiiadc}

Here we give a proof of Theorem \ref{theorem3}.  
\begin{proof}
We note that $\widetilde{\xi}_i$ and $\alpha$ can be written as $\widetilde{\xi}_i=\widetilde{\xi}_i^{c}+\widetilde{\xi}^{\text{inst}}$ and $\alpha=\alpha^c+\alpha^{\text{inst}}$, where the index $c$ denotes the part of the coordinates that coincides with the c-map case of \eqref{Darbouxcoords1loop}, and $\text{inst}$ denotes the terms involving $\Omega(\gamma)$. We have a similar decomposition for $f$ and $\theta_i^P$ using \eqref{fthetaexpres}, and hence $\lambda$ by using \eqref{contactlocal}. Because of Proposition \ref{typeIIA1loop} it is then enough to show that 

\begin{equation}
    f^{\text{inst}}\frac{\D t}{t}+t^{-1}\theta_+^{P,\text{inst}}|_{\overline{N}} -2\mathrm{i}\theta_3^{P,\text{inst}}|_{\overline{N}} +t\theta_-^{P,\text{inst}}|_{\overline{N}}=-2\pi \I \left(\D \alpha^{\text{inst}}+\widetilde{\xi}_i^{\text{inst}}\D \xi^i -\xi^i\D \widetilde{\xi}_i^{\text{inst}}\right)\,.
\end{equation}
We will check this by a direct computation.\\

First we compute $\D \alpha^{\text{inst}}$. From the expression of $\mathcal{W}$ in \eqref{wdef} we compute $\D\mathcal{W}$:

\begin{align*}
        \D \mathcal{W}=&R\sum_{\gamma}\Omega(\gamma)\sum_{n>0}\frac{e^{-2\pi\I n\zeta_{\gamma}}}{n}K_0(4\pi Rn|\widetilde{Z}_{\gamma}|)\left(\frac{\widetilde{Z}_{\gamma}}{2(\rho+c_{\ell})}\D \rho+\frac{\widetilde{Z}_{\gamma}}{2}\D \mathcal{K} +\D \widetilde{Z}_{\gamma}\right)\\
        &-2\pi\I R\sum_{\gamma}\Omega(\gamma)\widetilde{Z}_{\gamma}\sum_{n>0}e^{-2\pi\I n\zeta_{\gamma}}K_0(4\pi Rn|\widetilde{Z}_{\gamma}|)\D \zeta_{\gamma}\\
        &-2\pi R^2\sum_{\gamma}\Omega(\gamma)\widetilde{Z}_{\gamma}\sum_{n>0}e^{-2\pi\I n\zeta_{\gamma}}|\widetilde{Z}_{\gamma}|K_1(4\pi Rn|\widetilde{Z}_{\gamma}|)\left(\frac{\D \rho}{\rho+c_{\ell}}+\D \mathcal{K}+\frac{\D \widetilde{Z}_{\gamma}}{\widetilde{Z}_{\gamma}}+\frac{\D \overline{\widetilde{Z}}_{\gamma}}{\overline{\widetilde{Z}}_{\gamma}}\right)\\\numberthis
    \end{align*}
where we have used on the last line that $(K_0(x))'=-K_1(x)$.\\

On the other hand, we have

\begin{equation}
    \begin{split}
    \D &\left(-\frac{1}{2\pi}\sum_{\gamma}\Omega(\gamma)\int_{l_{\gamma}}\frac{\D \zeta}{\zeta}\frac{t+\zeta}{t-\zeta}\mathrm{L} (\exp(2\pi\I \xi_{\gamma}))\right)\\
    =&-\frac{1}{2\pi}\sum_{\gamma}\Omega(\gamma)\D t\int_{l_{\gamma}}\frac{\D \zeta}{\zeta}\partial_t\left(\frac{t+\zeta}{t-\zeta}\right)\mathrm{L} (\exp(2\pi\I \xi_\gamma(\zeta)))\\
    &-\sum_{\gamma}\Omega(\gamma)\int_{l_{\gamma}}\frac{\D \zeta}{\zeta}\frac{t+\zeta}{t-\zeta}\left(-\frac{\I }{2}\log(1-\exp(2\pi\I \xi_{\gamma})) +\frac{1}{2}\frac{\exp(\I \xi_{\gamma})}{1-\exp(2\pi\I \xi_{\gamma})}2\pi\xi_{\gamma}\right)\D \xi_{\gamma}(\zeta)
    \end{split}
\end{equation}

We can then join everything together to conclude that

\begin{align*}
        -2\pi \I \D \alpha^{\text{inst}}=&-\frac{1}{2\pi}\left(-t^{-2}\mathcal{W}+\overline{\mathcal{W}}\right)\D t\\
        &-t^{-1}\frac{R}{2\pi}\sum_{\gamma}\Omega(\gamma)\sum_{n>0}\frac{e^{-2\pi\I n\zeta_{\gamma}}}{n}K_0(4\pi Rn|\widetilde{Z}_{\gamma}|)\left(\frac{\widetilde{Z}_{\gamma}}{2(\rho+c_{\ell})}\D \rho+\frac{\widetilde{Z}_{\gamma}}{2}\D \mathcal{K} +\D \widetilde{Z}_{\gamma}\right)\\
        &+t^{-1}\I R\sum_{\gamma}\Omega(\gamma)\widetilde{Z}_{\gamma}\sum_{n>0}e^{-2\pi\I n\zeta_{\gamma}}K_0(4\pi Rn|\widetilde{Z}_{\gamma}|)\D \zeta_{\gamma}\\
        &+t^{-1}R^2\sum_{\gamma}\Omega(\gamma)\widetilde{Z}_{\gamma}\sum_{n>0}e^{-2\pi\I n\zeta_{\gamma}}|\widetilde{Z}_{\gamma}|K_1(4\pi Rn|\widetilde{Z}_{\gamma}|)\left(\frac{\D \rho}{\rho+c_{\ell}}+\D \mathcal{K}+\frac{\D \widetilde{Z}_{\gamma}}{\widetilde{Z}_{\gamma}}+\frac{\D \overline{\widetilde{Z}}_{\gamma}}{\overline{\widetilde{Z}}_{\gamma}}\right)\\
        &+t\cdot\left(\text{complex conjugate of $t^{-1}$ term factor}\right)\\ &+\frac{1}{4\pi^2}\sum_{\gamma}\Omega(\gamma)\D t\int_{l_{\gamma}}\frac{\D \zeta}{\zeta}\partial_t\left(\frac{t+\zeta}{t-\zeta}\right)\mathrm{L} (\exp(2\pi\I \xi_\gamma(\zeta)))\\
    &+\frac{1}{2\pi}\sum_{\gamma}\Omega(\gamma)\int_{l_{\gamma}}\frac{\D \zeta}{\zeta}\frac{t+\zeta}{t-\zeta}\left(-\frac{\I }{2}\log(1-\exp(2\pi\I \xi_{\gamma})) +\frac{1}{2}\frac{\exp(2\pi\I \xi_{\gamma})}{1-\exp(2\pi\I \xi_{\gamma})}2\pi\xi_{\gamma}\right)\D \xi_{\gamma}(\zeta)\numberthis\label{conttypeiia1}
    \end{align*}

On the other hand, for the terms $ \widetilde{\xi}_i^{\text{inst}}\D\xi^i- \xi^i\D\widetilde{\xi}_i^{\text{inst}}$ we have

\begin{equation}\label{conttypeiia2}
    \begin{split}
        -2\pi \I\left(\widetilde{\xi}_i^{\text{inst}}\D \xi^i- \xi^i\D \widetilde{\xi}_i^{\text{inst}}\right)=&-\frac{1}{4\pi \I }\sum_{\gamma}\Omega(\gamma)\D \xi_{\gamma}(t)\int_{l_{\gamma}}\frac{\D \zeta}{\zeta}\frac{t+\zeta}{t-\zeta}\log(1-\exp(2\pi\I \xi_{\gamma}))\\
        &+\frac{1}{4\pi \I }\sum_{\gamma}\Omega(\gamma)\xi_{\gamma}(t)\D t \int_{l_{\gamma}}\frac{\D \zeta}{\zeta}\partial_t\left(\frac{t+\zeta}{t-\zeta}\right)\log(1-\exp(2\pi\I \xi_{\gamma}(\zeta)))\\
        &-\frac{1}{2 }\sum_{\gamma}\Omega(\gamma)\xi_{\gamma}(t) \int_{l_{\gamma}}\frac{\D \zeta}{\zeta}\frac{t+\zeta}{t-\zeta}\frac{\exp(2\pi\I \xi_{\gamma}(\zeta))}{1-\exp(2\pi\I \xi_{\gamma}(\zeta))}\D \xi_{\gamma}(\zeta)\,.
    \end{split}
\end{equation}

To start combining terms, we notice that

\begin{equation}
    \frac{t+\zeta}{t-\zeta}(\D \xi_{\gamma}(t)-\D \xi_{\gamma}(\zeta))=\left(\frac{1}{\zeta}+\frac{1}{t}\right)\I \D (R\widetilde{Z}_{\gamma})-(t+\zeta)\I \D (R\overline{\widetilde{Z}}_{\gamma})-\frac{t+\zeta}{t-\zeta}\I R(-t^{-2}\widetilde{Z}_{\gamma} +\overline{\widetilde{Z}}_{\gamma})\D t
\end{equation}
and

\begin{equation}
    \frac{t+\zeta}{t-\zeta}(\xi_{\gamma}(t)-\xi_{\gamma}(\zeta))=\left(\frac{1}{\zeta}+\frac{1}{t}\right)\I R\widetilde{Z}_{\gamma}-(t+\zeta)\I R\overline{\widetilde{Z}}_{\gamma}
\end{equation}
so that we can combine the first and third term of \eqref{conttypeiia2} with the last term of \eqref{conttypeiia1} into 

\begin{align*}
        &-\frac{1}{4\pi \I }\sum_{\gamma}\Omega(\gamma)\D \xi_{\gamma}(t)\int_{l_{\gamma}}\frac{\D \zeta}{\zeta}\frac{t+\zeta}{t-\zeta}\log(1-\exp(2\pi\I \xi_{\gamma}))-\frac{1}{2 }\sum_{\gamma}\Omega(\gamma)\xi_{\gamma}(t) \int_{l_{\gamma}}\frac{\D \zeta}{\zeta}\frac{t+\zeta}{t-\zeta}\frac{\exp(2\pi\I \xi_{\gamma}(\zeta))}{1-\exp(2\pi\I \xi_{\gamma}(\zeta))}\D \xi_{\gamma}(\zeta)\\
        &\quad +\frac{1}{2\pi}\sum_{\gamma}\Omega(\gamma)\int_{l_{\gamma}}\frac{\D \zeta}{\zeta}\frac{t+\zeta}{t-\zeta}\left(-\frac{\I }{2}\log(1-\exp(2\pi\I \xi_{\gamma})) +\frac{1}{2}\frac{\exp(2\pi\I \xi_{\gamma})}{1-\exp(2\pi \I\xi_{\gamma})}2\pi\xi_{\gamma}(\zeta)\right)\D \xi_{\gamma}(\zeta)\\
        &=-\frac{1}{4\pi \I }\sum_{\gamma}\Omega(\gamma)\int_{l_{\gamma}}\frac{\D \zeta}{\zeta}\frac{t+\zeta}{t-\zeta}\log(1-\exp(2\pi\I \xi_{\gamma}))(\D \xi_{\gamma}(t)-\D \xi_{\gamma}(\zeta)) \\
        &\quad -\frac{1}{2 }\sum_{\gamma}\Omega(\gamma) \int_{l_{\gamma}}\frac{\D \zeta}{\zeta}\frac{t+\zeta}{t-\zeta}\frac{\exp(2\pi\I \xi_{\gamma}(\zeta))}{1-\exp(2\pi\I \xi_{\gamma}(\zeta))}(\xi_{\gamma}(t)-\xi_{\gamma}(\zeta))\D \xi_{\gamma}(\zeta)\\
        &=-\frac{1}{4\pi \I }\sum_{\gamma}\Omega(\gamma)\int_{l_{\gamma}}\frac{\D \zeta}{\zeta}\log(1-\exp(2\pi\I \xi_{\gamma}))\left(\left(\frac{1}{\zeta}+\frac{1}{t}\right)\I \D (R\widetilde{Z}_{\gamma})-(t+\zeta)\I \D (R\overline{\widetilde{Z}}_{\gamma})-\frac{t+\zeta}{t-\zeta}\I R(-t^{-2}\widetilde{Z}_{\gamma} +\overline{\widetilde{Z}}_{\gamma})\D t\right) \\
        &\quad +\frac{1}{2 }\sum_{\gamma}\Omega(\gamma) \int_{l_{\gamma}}\frac{\D \zeta}{\zeta}\frac{\exp(2\pi\I \xi_{\gamma}(\zeta))}{1-\exp(2\pi\I \xi_{\gamma}(\zeta))}\left(\left(\frac{1}{\zeta}+\frac{1}{t}\right)\I R\widetilde{Z}_{\gamma}-(t+\zeta)\I R\overline{\widetilde{Z}}_{\gamma}\right)\left(\D \zeta_{\gamma} +\I \zeta^{-1}\D (R\widetilde{Z}_{\gamma}) +\I \zeta \D (R\overline{\widetilde{Z}}_{\gamma})\right)\,.\numberthis\\
        \end{align*}
        We can rewrite the previous integrals in the last equality in terms of Bessel functions by using the identities \eqref{Bessel1} and \eqref{Bessel2}, obtaining
        \begin{align*}
        =&\frac{1}{4\pi }\sum_{\gamma}\Omega(\gamma)\D t\int_{l_{\gamma}}\frac{\D \zeta}{\zeta}\log(1-\exp(2\pi\I \xi_{\gamma}))\left(\frac{t+\zeta}{t-\zeta}R(-t^{-2}\widetilde{Z}_{\gamma} +\overline{\widetilde{Z}}_{\gamma})\right)\\
        &- \frac{1}{2\pi}\sum_{\gamma}\Omega(\gamma)\D (R\widetilde{Z}_{\gamma})\left(\sum_{n>0}\frac{|\widetilde{Z}_{\gamma}|}{\widetilde{Z}_{\gamma}}\frac{e^{-2\pi\I n\zeta_{\gamma}}}{n}K_1(4\pi Rn|\widetilde{Z}_{\gamma}|)-\frac{1}{t}\sum_{n>0}\frac{e^{-2\pi\I n\zeta_{\gamma}}}{n}K_0(4\pi Rn|\widetilde{Z}_{\gamma}|)\right)\\
        &-\frac{1}{2\pi}\sum_{\gamma}\Omega(\gamma)\D (R\overline{\widetilde{Z}}_{\gamma})\left(-\sum_{n>0}\frac{\widetilde{Z}_{\gamma}}{|\widetilde{Z}_{\gamma}|}\frac{e^{-2\pi\I n\zeta_{\gamma}}}{n}K_1(4\pi Rn|\widetilde{Z}_{\gamma}|)+t\sum_{n>0}\frac{e^{-2\pi\I n\zeta_{\gamma}}}{n}K_0(4\pi Rn|\widetilde{Z}_{\gamma}|)\right)\\
        &-\sum_{\gamma}\Omega(\gamma)\frac{|\widetilde{Z}_{\gamma}|^2}{\widetilde{Z}_{\gamma}}R\D (R\widetilde{Z}_{\gamma})\sum_{n>0}e^{-2\pi\I n\zeta_{\gamma}}\left(K_0(4\pi Rn|\widetilde{Z}_{\gamma}|)+\frac{K_1(4\pi Rn|\widetilde{Z}_{\gamma}|)}{2\pi Rn|\widetilde{Z}_{\gamma}|}\right)\\
        &+\sum_{\gamma}\Omega(\gamma)\left(-\I \widetilde{Z}_{\gamma}R\D \zeta_{\gamma}+\left(\widetilde{Z}_{\gamma}/t-t\overline{\widetilde{Z}}_{\gamma}\right)R\D (R\widetilde{Z}_{\gamma})\right)\frac{|\widetilde{Z}_{\gamma}|}{\widetilde{Z}_{\gamma}}\sum_{n>0}e^{-2\pi\I n\zeta_{\gamma}}K_1(4\pi Rn|\widetilde{Z}_{\gamma}|)\\
        &-\sum_{\gamma}\Omega(\gamma)\left(-(\widetilde{Z}_{\gamma}/t-t\overline{\widetilde{Z}}_{\gamma})\I R\D \zeta_{\gamma} +R\widetilde{Z}_{\gamma}\D (R\overline{\widetilde{Z}}_{\gamma}) - \overline{\widetilde{Z}}_{\gamma}R\D (R\widetilde{Z}_{\gamma})\right)\sum_{n>0}e^{-2\pi \I n\zeta_{\gamma}}K_0(4\pi Rn|\widetilde{Z}_{\gamma}|)\\
        &+\sum_{\gamma}\Omega(\gamma)\left(\overline{\widetilde{Z}}_{\gamma}\I R\D \zeta_{\gamma}+(\widetilde{Z}_{\gamma}/t-t\overline{\widetilde{Z}}_{\gamma})R\D (R\overline{\widetilde{Z}}_{\gamma})\right)\frac{\widetilde{Z}_{\gamma}}{|\widetilde{Z}_{\gamma}|}\sum_{n>0}e^{-2\pi\I n\zeta_{\gamma}}K_1(4\pi Rn|\widetilde{Z}_{\gamma}|)\\
        &+\sum_{\gamma}\Omega(\gamma)\overline{\widetilde{Z}}_{\gamma}R\D (R\overline{\widetilde{Z}}_{\gamma})\frac{\widetilde{Z}_{\gamma}^2}{|\widetilde{Z}_{\gamma}|^2}\sum_{n>0}e^{-2\pi\I n\zeta_{\gamma}}\left(K_0(4\pi Rn|\widetilde{Z}_{\gamma}|)+\frac{K_1(4\pi Rn|\widetilde{Z}_{\gamma}|)}{2\pi Rn|\widetilde{Z}_{\gamma}|}\right) \numberthis
    \end{align*}

Overall, we obtain the following $t^{-1}$ term for $-2\pi \I(\D \alpha^{\text{inst}}+\widetilde{\xi}_i^{\text{inst}}\D\xi^i-\xi^i\D \widetilde{\xi}_i^{\text{inst}})$: 

\begin{align*}
        &\frac{R}{2\pi}\sum_{\gamma}\Omega(\gamma)\sum_{n>0}\frac{e^{-2\pi\I n\zeta_{\gamma}}}{n}K_0(4\pi Rn|\widetilde{Z}_{\gamma}|)\left(\frac{\widetilde{Z}_{\gamma}}{2(\rho+c_{\ell})}\D \rho+\frac{\widetilde{Z}_{\gamma}}{2}\D \mathcal{K} +\D \widetilde{Z}_{\gamma}\right)\\
        &+R^2\sum_{\gamma}\Omega(\gamma)\widetilde{Z}_{\gamma}\sum_{n>0}e^{-2\pi\I n\zeta_{\gamma}}|\widetilde{Z}_{\gamma}|K_1(4\pi Rn|\widetilde{Z}_{\gamma}|)\left(\frac{\D \rho}{2(\rho+c_{\ell})}+\frac{\D \mathcal{K}}{2} +\frac{\D \widetilde{Z}_{\gamma}}{\widetilde{Z}_{\gamma}}\right)\\
        &+\I R\sum_{\gamma}\Omega(\gamma)\widetilde{Z}_{\gamma}\sum_{n>0}e^{-2\pi\I n\zeta_{\gamma}}K_0(4\pi Rn|\widetilde{Z}_{\gamma}|)\D \zeta_{\gamma}\\
        &+R^2\sum_{\gamma}\Omega(\gamma)\widetilde{Z}_{\gamma}\sum_{n>0}e^{-2\pi\I n\zeta_{\gamma}}|\widetilde{Z}_{\gamma}|K_1(4\pi Rn|\widetilde{Z}_{\gamma}|)\left(\frac{\D \rho}{2(\rho+c_{\ell})}+\frac{\D \mathcal{K}}{2} +\frac{\D \overline{\widetilde{Z}}_{\gamma}}{\overline{\widetilde{Z}}_{\gamma}}\right)\\
        &-\frac{R}{2\pi}\sum_{\gamma}\Omega(\gamma)\sum_{n>0}\frac{e^{-2\pi\I n\zeta_{\gamma}}}{n}K_0(4\pi Rn|\widetilde{Z}_{\gamma}|)\left(\frac{\widetilde{Z}_{\gamma}}{2(\rho+c_{\ell})}\D \rho+\frac{\widetilde{Z}_{\gamma}}{2}\D \mathcal{K} +\D \widetilde{Z}_{\gamma}\right)\\
        &+\I R\sum_{\gamma}\Omega(\gamma)\widetilde{Z}_{\gamma}\sum_{n>0}e^{-2\pi\I n\zeta_{\gamma}}K_0(4\pi Rn|\widetilde{Z}_{\gamma}|)\D \zeta_{\gamma}\\
        &+R^2\sum_{\gamma}\Omega(\gamma)\widetilde{Z}_{\gamma}\sum_{n>0}e^{-2\pi\I n\zeta_{\gamma}}|\widetilde{Z}_{\gamma}|K_1(4\pi Rn|\widetilde{Z}_{\gamma}|)\left(\frac{\D \rho}{\rho+c_{\ell}}+\D \mathcal{K}+\frac{\D \widetilde{Z}_{\gamma}}{\widetilde{Z}_{\gamma}}+\frac{\D \overline{\widetilde{Z}}_{\gamma}}{\overline{\widetilde{Z}}_{\gamma}}\right)\\
        =&2\I R\sum_{\gamma}\Omega(\gamma)\widetilde{Z}_{\gamma}\sum_{n>0}e^{-2\pi\I n\zeta_{\gamma}}K_0(4\pi Rn|\widetilde{Z}_{\gamma}|)\D \zeta_{\gamma}\\
        &+2R^2\sum_{\gamma}\Omega(\gamma)\widetilde{Z}_{\gamma}\sum_{n>0}e^{-2\pi\I n\zeta_{\gamma}}|\widetilde{Z}_{\gamma}|K_1(4\pi Rn|\widetilde{Z}_{\gamma}|)\left(\frac{\D \rho}{\rho+c_{\ell}}+\D \mathcal{K}+\frac{\D \widetilde{Z}_{\gamma}}{\widetilde{Z}_{\gamma}}+\frac{\D \overline{\widetilde{Z}}_{\gamma}}{\overline{\widetilde{Z}}_{\gamma}}\right)\\
\end{align*}

By comparing with \eqref{fthetaexpres} we see that this matches $\theta_{+}^{P,\text{inst}}|_{\overline{N}}$. The $t$ term follows from the $t^{-1}$, by noticing that all the $t$-term are conjugates of the $t^{-1}$-term. In particular, the $t$-term will match $\theta_{-}^{\text{inst}}=\overline{\theta_{+}^{\text{inst}}}$.\\

Now we collect the $t^0$ term from our previous expressions for $-2\pi \I\left(\D \alpha^{\text{inst}}+\widetilde{\xi}_i^{\text{inst}}\D \xi^i -\xi^i\D \widetilde{\xi}_i^{\text{inst}}\right)$. We obtain the following:

\begin{align*}
        &-\frac{R}{2\pi}\sum_{\gamma}\Omega(\gamma)\sum_{n>0}\frac{e^{-2\pi\I n\zeta_{\gamma}}}{n}|\widetilde{Z}_{\gamma}|K_1(4\pi Rn|\widetilde{Z}_{\gamma}|)\left(\frac{\D \rho}{2(\rho+c_{\ell})}+\frac{\D \mathcal{K}}{2} +\frac{\D \widetilde{Z}_{\gamma}}{\widetilde{Z}_{\gamma}}\right)\\
        &+\frac{R}{2\pi}\sum_{\gamma}\Omega(\gamma)\sum_{n>0}\frac{e^{-2\pi\I n\zeta_{\gamma}}}{n}|\widetilde{Z}_{\gamma}|K_1(4\pi Rn|\widetilde{Z}_{\gamma}|)\left(\frac{\D \rho}{2(\rho+c_{\ell})}+\frac{\D \mathcal{K}}{2} +\frac{\D \overline{\widetilde{Z}}_{\gamma}}{\overline{\widetilde{Z}}_{\gamma}}\right)\\
         &-R^2\sum_{\gamma}\Omega(\gamma)|\widetilde{Z}_{\gamma}|^2\sum_{n>0}e^{-2\pi\I n\zeta_{\gamma}}\left(K_0(4\pi Rn|\widetilde{Z}_{\gamma}|)+\frac{K_1(4\pi Rn|\widetilde{Z}_{\gamma}|)}{2\pi Rn|\widetilde{Z}_{\gamma}|}\right)\left(\frac{\D \rho}{2(\rho+c_{\ell})}+\frac{\D \mathcal{K}}{2} +\frac{\D \widetilde{Z}_{\gamma}}{\widetilde{Z}_{\gamma}}\right)\\
         &-\I R\sum_{\gamma}\Omega(\gamma)\sum_{n>0}e^{-2\pi\I n\zeta_{\gamma}}|\widetilde{Z}_{\gamma}|K_1(4\pi Rn|\widetilde{Z}_{\gamma}|)\D \zeta_{\gamma}\\
         &+R^2\sum_{\gamma}\Omega(\gamma)|\widetilde{Z}_{\gamma}|^2\left(\left(\frac{\D \rho}{2(\rho+c_{\ell})}+\frac{\D \mathcal{K}}{2} +\frac{\D \widetilde{Z}_{\gamma}}{\widetilde{Z}_{\gamma}}\right) - \left(\frac{\D \rho}{2(\rho+c_{\ell})}+\frac{\D \mathcal{K}}{2} +\frac{\D \overline{\widetilde{Z}}_{\gamma}}{\overline{\widetilde{Z}}_{\gamma}}\right)\right)\sum_{n>0}e^{-2\pi\I n\zeta_{\gamma}}K_0(4\pi Rn|\widetilde{Z}_{\gamma}|)\\
         &+\I R\sum_{\gamma}\Omega(\gamma)\sum_{n>0}e^{-2\pi\I n\zeta_{\gamma}}|\widetilde{Z}_{\gamma}|K_1(4\pi Rn|\widetilde{Z}_{\gamma}|)\D \zeta_{\gamma}\\
         &+R^2\sum_{\gamma}\Omega(\gamma)|\widetilde{Z}_{\gamma}|^2\left(\frac{\D \rho}{2(\rho+c_{\ell})}+\frac{\D \mathcal{K}}{2} +\frac{\D \overline{\widetilde{Z}}_{\gamma}}{\overline{\widetilde{Z}}_{\gamma}}\right)\sum_{n>0}e^{-2\pi\I n\zeta_{\gamma}}\left(K_0(4\pi Rn|\widetilde{Z}_{\gamma}|)+\frac{K_1(4\pi Rn|\widetilde{Z}_{\gamma}|)}{2\pi Rn|\widetilde{Z}_{\gamma}|}\right)\\
         =&-\frac{R}{\pi}\sum_{\gamma}\Omega(\gamma)\sum_{n>0}\frac{e^{-2\pi\I n\zeta_{\gamma}}}{n}|\widetilde{Z}_{\gamma}|K_1(4\pi Rn|\widetilde{Z}_{\gamma}|)\left(\frac{\D \widetilde{Z}_{\gamma}}{\widetilde{Z}_{\gamma}}-\frac{\D \overline{\widetilde{Z}}_{\gamma}}{\overline{\widetilde{Z}}_{\gamma}}\right)\numberthis\\
\end{align*}
so comparing with \eqref{fthetaexpres} we see that the $t^0$ terms matches $-2\I\theta_3^{P,\text{inst}}$.\\

Finally, we need to collect the $\D t$ term from $-2\pi \I \left(\D \alpha^{\text{inst}}+\widetilde{\xi}_i^{\text{inst}}\D \xi^i -\xi^i\D \widetilde{\xi}_i^{\text{inst}}\right)$. This one corresponds to 

\begin{align*}
    &-t^{-2}\frac{R}{4\pi}\sum_{\gamma}\Omega(\gamma)\widetilde{Z}_{\gamma}\int_{l_{\gamma}}\frac{\D \zeta}{\zeta}\log(1-\exp(2\pi\I \xi_{\gamma}(\zeta)))-\frac{R}{4\pi}\sum_{\gamma}\Omega(\gamma)\overline{\widetilde{Z}}_{\gamma}\int_{l_{\gamma}}\frac{\D \zeta}{\zeta}\log(1-\exp(2\pi\I \xi_{\gamma}(\zeta)))\\
    &+\frac{1}{4\pi^2}\sum_{\gamma}\Omega(\gamma)\int_{l_{\gamma}}\frac{\D \zeta}{\zeta}\partial_t\left(\frac{t+\zeta}{t-\zeta}\right)\mathrm{L} (\exp(2\pi\I \xi_\gamma(\zeta)))+\frac{1}{4\pi \I }\sum_{\gamma}\Omega(\gamma)\xi_{\gamma}(t) \int_{l_{\gamma}}\frac{\D \zeta}{\zeta}\partial_t\left(\frac{t+\zeta}{t-\zeta}\right)\log(1-\exp(2\pi\I \xi_{\gamma}(\zeta)))\\
    &+\frac{1}{4\pi }\sum_{\gamma}\Omega(\gamma)\int_{l_{\gamma}}\frac{\D \zeta}{\zeta}\frac{t+\zeta}{t-\zeta}R(-t^{-2}\widetilde{Z}_{\gamma} +\overline{\widetilde{Z}}_{\gamma})\log(1-\exp(2\pi\I \xi_{\gamma}(\zeta)))\\
    =&-\frac{R}{2\pi}\sum_{\gamma}\Omega(\gamma)\int_{l_{\gamma}}\frac{\D \zeta}{\zeta}\frac{1}{t-\zeta}(t^{-1}\widetilde{Z}_{\gamma}-\zeta\overline{\widetilde{Z}_{\gamma}})\log(1-\exp(2\pi\I \xi_{\gamma}(\zeta)))-\frac{1}{2\pi^2}\sum_{\gamma}\Omega(\gamma)\int_{l_{\gamma}}\D \zeta\frac{1}{(t-\zeta)^2}\mathrm{L} (\exp(2\pi\I \xi_\gamma(\zeta)))\\
    &-\frac{1}{2\pi \I }\sum_{\gamma}\Omega(\gamma)\xi_{\gamma}(t) \int_{l_{\gamma}}\D \zeta\frac{1}{(t-\zeta)^2}\log(1-\exp(2\pi\I \xi_{\gamma}(\zeta))) \numberthis
\end{align*}

Integrating by parts the second and third term of the last equality, we obtain

\begin{align*}
        &-\frac{R}{2\pi}\sum_{\gamma}\Omega(\gamma)\int_{l_{\gamma}}\frac{\D \zeta}{\zeta}\frac{1}{t-\zeta}(t^{-1}\widetilde{Z}_{\gamma}-\zeta\overline{\widetilde{Z}_{\gamma}})\log(1-\exp(2\pi\I \xi_{\gamma}(\zeta)))\\
    &+\frac{1}{2\pi}\sum_{\gamma}\Omega(\gamma)\int_{l_{\gamma}}\D \zeta\frac{1}{t-\zeta}\left(-\I \log(1-\exp(2\pi\I \xi_{\gamma}))) + \frac{\exp(2\pi\I \xi_{\gamma})}{1-\exp(2\pi\I \xi_{\gamma})}2\pi\xi_{\gamma}(\zeta)\right)(\zeta^{-2}\I R\widetilde{Z}_{\gamma}-\I R\overline{\widetilde{Z}_{\gamma}})\\
    &-\sum_{\gamma}\Omega(\gamma)\xi_{\gamma}(t) \int_{l_{\gamma}}\D \zeta\frac{1}{t-\zeta}\frac{\exp(2\pi\I \xi_{\gamma})}{1-\exp(2\pi\I \xi_{\gamma})}(\zeta^{-2}\I R\widetilde{Z}_{\gamma}-\I R\overline{\widetilde{Z}_{\gamma}})\\
    =&t^{-1}\frac{R}{2\pi}\sum_{\gamma}\Omega(\gamma)\widetilde{Z}_{\gamma}\int_{l_{\gamma}}\frac{\D \zeta}{\zeta^2}\log(1-\exp(2\pi\I \xi_{\gamma}(\zeta))) \\
    &-\sum_{\gamma}\Omega(\gamma) \int_{l_{\gamma}}\D \zeta\frac{1}{t-\zeta}(\xi_{\gamma}(t)-\xi_{\gamma}(\zeta))\frac{\exp(2\pi\I \xi_{\gamma})}{1-\exp(2\pi\I \xi_{\gamma})}(\zeta^{-2}\I R\widetilde{Z}_{\gamma}-\I R\overline{\widetilde{Z}_{\gamma}})\,.\numberthis \label{appeq}
    \end{align*}

Using that

\begin{equation}
    \frac{1}{t-\zeta}(\xi_{\gamma}(t)-\xi_{\gamma}(\zeta))=\I R\frac{\widetilde{Z}_{\gamma}}{\zeta t}-\I R\overline{\widetilde{Z}}_{\gamma}
\end{equation}

we then obtain that the right-hand side of the equality of \eqref{appeq} becomes

\begin{align*}
        &t^{-1}\frac{R}{2\pi}\sum_{\gamma}\Omega(\gamma)\widetilde{Z}_{\gamma}\int_{l_{\gamma}}\frac{\D \zeta}{\zeta^2}\log(1-\exp(2\pi\I \xi_{\gamma}(\zeta))) \\
    &-\sum_{\gamma}\Omega(\gamma) \int_{l_{\gamma}}\D \zeta(\I R\frac{\widetilde{Z}_{\gamma}}{\zeta t}-\I R\overline{\widetilde{Z}}_{\gamma})\frac{\exp(2\pi\I \xi_{\gamma})}{1-\exp(2\pi\I \xi_{\gamma})}(\zeta^{-2}\I R\widetilde{Z}_{\gamma}-\I R\overline{\widetilde{Z}_{\gamma}})\\
    =&t^{-1}\frac{R}{2\pi}\sum_{\gamma}\Omega(\gamma)\widetilde{Z}_{\gamma}\int_{l_{\gamma}}\frac{\D \zeta}{\zeta^2}\log(1-\exp(2\pi\I \xi_{\gamma}(\zeta))) +t^{-1}R^2\sum_{\gamma}\Omega(\gamma) \widetilde{Z}_{\gamma}^2\int_{l_{\gamma}}\frac{\D \zeta}{\zeta^3} \frac{\exp(2\pi\I \xi_{\gamma})}{1-\exp(2\pi\I \xi_{\gamma})}\\
    &-t^{-1}R^2\sum_{\gamma}\Omega(\gamma) |\widetilde{Z}_{\gamma}|^2\int_{l_{\gamma}}\frac{\D \zeta}{\zeta} \frac{\exp(2\pi\I \xi_{\gamma})}{1-\exp(2\pi\I \xi_{\gamma})}-R^2\sum_{\gamma}\Omega(\gamma) |\widetilde{Z}_{\gamma}|^2\int_{l_{\gamma}}\frac{\D \zeta}{\zeta^2} \frac{\exp(2\pi\I \xi_{\gamma})}{1-\exp(2\pi\I \xi_{\gamma})}\\
    &+R^2\sum_{\gamma}\Omega(\gamma) \overline{\widetilde{Z}}_{\gamma}^2\int_{l_{\gamma}}\D \zeta \frac{\exp(2\pi\I \xi_{\gamma})}{1-\exp(2\pi\I \xi_{\gamma})} \numberthis
    \end{align*}

The above integrals can be solved explicitly by using \eqref{Bessel1} and \eqref{Bessel2}, giving 

\begin{align*}
    =&t^{-1}\frac{R}{\pi}\sum_{\gamma}\Omega(\gamma)\sum_{n>0}\frac{e^{-2\pi\I n\zeta_{\gamma}}}{n}|\widetilde{Z}_{\gamma}|K_1(4\pi Rn|\widetilde{Z}_{\gamma}|) \\
    &+2t^{-1}R^2\sum_{\gamma}\Omega(\gamma) |\widetilde{Z}_{\gamma}|^2\sum_{n>0}e^{-2\pi\I n\zeta_{\gamma}}\left(K_0(4\pi Rn|\widetilde{Z}_{\gamma}|)+\frac{K_1(4\pi Rn|\widetilde{Z}_{\gamma}|)}{2\pi Rn|\widetilde{Z}_{\gamma}|}\right)\\
    &-2t^{-1}R^2\sum_{\gamma}\Omega(\gamma) |\widetilde{Z}_{\gamma}|^2\sum_{n>0}e^{-2\pi\I n\zeta_{\gamma}}K_0(4\pi Rn|\widetilde{Z}_{\gamma}|)\\
    &-R^2\sum_{\gamma}\Omega(\gamma) |\widetilde{Z}_{\gamma}|^2\left(-2\frac{|\widetilde{Z}_{\gamma}|}{\widetilde{Z}_{\gamma}}\sum_{n>0}e^{-2\pi\I n\zeta_{\gamma}}K_1(4\pi Rn|\widetilde{Z}_{\gamma}|)\right)+R^2\sum_{\gamma}\Omega(\gamma) \overline{\widetilde{Z}}_{\gamma}^2\left(-2\frac{\widetilde{Z}_{\gamma}}{|\widetilde{Z}_{\gamma}|}\sum_{n>0}e^{-2\pi\I n\zeta_{\gamma}}K_1(4\pi Rn|\widetilde{Z}_{\gamma}|)\right)\\
    =&t^{-1}\frac{2R}{\pi}\sum_{\gamma}\Omega(\gamma)\sum_{n>0}\frac{e^{-2\pi\I n\zeta_{\gamma}}}{n}|\widetilde{Z}_{\gamma}|K_1(4\pi Rn|\widetilde{Z}_{\gamma}|)\numberthis
    \end{align*}

Hence, by comparing with \eqref{fthetaexpres}, we see that the $\D t$ component of $-2\pi \I\left(\D \alpha^{\text{inst}}+\widetilde{\xi}_i^{\text{inst}}\D \xi^i -\xi^i\D \widetilde{\xi}_i^{\text{inst}}\right)$ matches with $f^{\text{inst}}/t$. This completes the proof.
\end{proof}

\bibliography{References}
\bibliographystyle{alpha}
\end{document}